\documentclass[reqno,twoside,11pt]{amsart}

\theoremstyle{plain}

\def\endproof{\hspace*{\fill}\mbox{\ \rule{.1in}{.1in}}\medskip }

\numberwithin{equation}{section}

\usepackage{amsmath, amsfonts, amssymb, esint, amsthm, nicefrac, bbm}
\usepackage{booktabs}
\usepackage{epsfig, graphicx}
\usepackage{accents, csquotes} 
\usepackage{color}

\DeclareMathAlphabet{\mathpzc}{OT1}{pzc}{m}{it}
\newcommand{\hordiv}{\mbox{div}_{\mathbb{H}}}

\newcommand{\hhor}{\nabla_{\mathbb{H}}}
\newcommand{\heis}{\mathbb{H}}
\newcommand{\dd}{~\mbox{d}}

\newcommand{\osc}{\mbox{osc}}
\newcommand{\p}{\mathrm{\bf p}}

\theoremstyle{plain}

\newtheorem{Teo}{Theorem}[section]
\newtheorem{Lemma}[Teo]{Lemma}
\newtheorem{Cor}[Teo]{Corollary}
\newtheorem{Prop}[Teo]{Proposition}

\theoremstyle{definition}
\newtheorem{Rem}[Teo]{Remark}
\newtheorem{Def}[Teo]{Definition}
\newtheorem{example}[Teo]{Example}

\def\sideremark#1{\ifvmode\leavevmode\fi\vadjust{\vbox to0pt{\vss
 \hbox to 0pt{\hskip\hsize\hskip1em
\vbox{\hsize2cm\small\raggedright\pretolerance10000
 \noindent #1\hfill}\hss}\vbox to8pt{\vfil}\vss}}}
 
\begin{document}
\title[Random Walks in the Heisenberg Group]{Random walks and random
  Tug of War \\ in the Heisenberg group} 
\author{Marta Lewicka, Juan Manfredi and Diego Ricciotti}
\address{M.L. and J.M.: University of Pittsburgh, Department of Mathematics, 
139 University Place, Pittsburgh, PA 15260}
\address{D.R.: University of South Florida, Department of Mathematics
  and Statistics, 4202 E Fowler Ave, Tampa, FL 33620}
\email{lewicka@pitt.edu, manfredi@pitt.edu, ricciotti@usf.edu} 

\begin{abstract}
We study the mean value properties 
of $\p$-harmonic functions on the first Heisenberg group $\heis$, in
connection to the dynamic programming principles of certain
stochastic processes. We implement the approach of Peres-Sheffield \cite{PS}
to provide the game-theoretical interpretation of the sub-elliptic
$\p$-Laplacian; and of Manfredi-Parviainen-Rossi \cite{MPR0} to cha\-rac\-terize its viscosity
solutions via the asymptotic mean value expansions.
\end{abstract}

\date{\today}
\maketitle
%\tableofcontents

\section{Introduction}

In this paper, we are concerned with the mean value properties 
of $\p$-harmonic functions on the Heisenberg group $\heis$, in
connection to the dynamic programming principles of certain
stochastic processes. More precisely, we develop asymptotic {\em mean value expansions} of the type:
\begin{equation}\label{intro}
Average (v,r)(q) = v(q) + c r^2\Delta_{\heis, \p}^N v(q) + o(r^2) \qquad
\mbox{ as }\; r\to 0+,
\end{equation}
for the normalized version $\Delta^N_{\heis, \p}$ of the $\p$-$\heis$-Laplacian $\Delta_{\heis, \p}$ in
(\ref{pLapdef})-(\ref{pLapdefN}), for any $1<\p<\infty$. The ``$Average$'' denotes here a
suitable mean value operator, acting on a given function
$v:\heis\to\mathbb{R}$, on a set of radius $r$ and centered at a point
$q\in\heis$. This operator may be either ``stochastic'': $\fint$, or
``deterministic'': $\frac{1}{2}(\sup + \inf)$, or it may be given through
various compositions or  further averages of such types.
The averaging set may be one of the following: the $3$-dimensional
Kor\'{a}nyi ball $B_r(q)$; the $2$-dimensional ellipse in the tangent plane $T_q$
passing through $q$, whose horizontal projection coincides with
the $2$-dimensional Euclidean ball of radius $r$; the $1$-dimensional boundary of such ellipse; or
the $3$-dimensional Kor\'{a}nyi ellipsoid that is the image of $B_r(q)$
under a suitable linear map. % suitably contracting the horizontal directions. 

\smallskip

For particular expansions in (\ref{intro}), we
study solutions to the boundary value problems for the related {\em
  mean value equations}, posed on a bounded domain $\mathcal{D}\subset\heis$,
with data $F\in\mathcal{C}(\heis)$:
\begin{equation}\label{intro2}
Average (u^\epsilon,\epsilon)= u^\epsilon  \; \mbox{ in }\;
\mathcal{D}, \qquad u^\epsilon = F \;\mbox{ on } \heis\setminus \mathcal{D}.
\end{equation}
We identify the solution $u^\epsilon$ as the value of a
process with, in general, both random and deterministic
components. The purely random component is related to the ``stochastic''
averaging part of the operator $Average$ as described above, whereas
the deterministic component is related to the
``deterministic'' part and can be interpreted as the {\em Tug of War 
game}. Recall that the Tug of War is a zero-sum, two-players game process, in
which the position of the particle in $\mathcal{D}$ is shifted
according to the deterministic {\em strategies} of the two players.
The players take turns with equal probabilities and strive to maximize or
minimize the game outcome given by the value of $F$ at the
particle's final (stopping) position.  

\smallskip

We then examine convergence of the family $\{u^\epsilon\}_{\epsilon\to
0}$. For domains with {\em game-regular boundary} $\partial\mathcal{D}$,
we show the uniform convergence to the viscosity solution of the Dirichlet problem:
\begin{equation}\label{intro3}
\Delta_{\heis, \p} u = 0  \; \mbox{ in }\; \mathcal{D}, \qquad u = F \;\mbox{ on } \partial\mathcal{D}.
\end{equation}
The definition of game-regularity is process-related, and it replaces
the celebrated Wiener capacitary criterion \cite{HKM}, in the probabilistic
setting that we are pursuing. Heuristically, game-regularity is equivalent to the
equicontinuity of the family $\{u^\epsilon\}_{\epsilon\to 0}$ on
$\bar{\mathcal{D}}$, where the only obstruction is due to the possibly
high probability of the event where the particle exits a prescribed neighbourhood of a
boundary point while still in $\mathcal{D}$.
In particular, we show that this scenario cannot happen when $\mathcal{D}$
satisfies the {\em exterior $\heis$-corkscrew condition}; indeed such domains
are automatically game-regular.

\smallskip

The program outlined above, familiar in the linear setting of $\p=2$,
where it reflects the well-studied correspondence between the Laplace operator
and the Brownian motion \cite{Dbook}, mimics the approach put forward in the
seminal papers \cite{PSSW, PS} by Peres, Schramm, Sheffield and
Wilson. There, the authors introduced the game-theoretical
interpretation of the $\infty$-Laplacian and
the $\p$-Laplacian in the Euclidean geometry, and during the past decade
many follow up works appeared in the literature \cite{MPR, PR, 1, 2, 3,
  4, 5, Lew}. In the present context of Heisenberg geometry -- in relation to the
operators \textit{Average} in (1.1) and their game-theoretical
description -- a preliminary mean value characterization of
viscosity $\p$-$\heis$-harmonic functions appeared in \cite{FLM} for
$\p\geq 2$, without addressing the issue of convergence. The contribution
of this paper is that we carry out the indicated program in full, including the
general case of exponents $1<\p<\infty$ and proving convergence, in
relation to game-theoretical interpretation. We
also believe that our careful clarification of certain proofs in
\cite{PS} (including the inductive techniques in Lemma
\ref{rings_inductive} and Theorem \ref{th_concat}, and their
application in the proofs of Theorems \ref{cork_good_2}, \ref{annulus}
and \ref{corkthengamereg}), albeit in the present sub-Riemannian context, will benefit the reader 
less familiar with probability techniques.

\subsection{The structure and results of this paper} Our contribution is
divided into three parts. 

Part I consists of four sections, in which we develop different mean value expansions  (\ref{intro}).
In section \ref{sec22} we begin with three averaging operators in
connection to the linear case exponent $\p=2$. The $1$-dimensional
and $2$-dimensional expansions are proved in
Proposition \ref{PropositionMV1Expansion};  validity of the
related mean value properties as in (\ref{intro2}) is then automatically
equivalent to $\heis$-harmonicity. A similar statement for the
$3$-dimensional average on Kor\'{a}nyi balls, only holds in the
viscosity sense (Proposition \ref{MV3expansion}), and can be
seen as a counterpart to the Gauss-Koebe-Levi-Tonelli theorem
\cite{book}, where the average is taken with respect to a
non-uniform probability measure.

In section \ref{sec33} we treat the case of the fully nonlinear operator $\Delta_{\heis,
  \infty}$, utilizing the ``deterministic'' averaging rather than the
``stochastic'' ones as in section \ref{sec22}. 
This description is in agreement with the absolutely minimizing Lipschitz
extension (AMLE) property of the $\infty$-harmonic functions $u$, which states
that for every open subset $U$, the restriction $u_{\mid U}$
has the smallest Lipschitz constant among all the extensions of
$u_{\mid\partial U}$ on $\bar U$ (see \cite{ACJ} for the Euclidean and
\cite{DMV} for the Heisenberg setting). In section \ref{sec44} we combine
the averages for $\p=2$ and $\p=\infty$ and propose two
mean value expansions for $\Delta_{\heis, \p}$, via
superpositions that are both modeled on the interpolation property of
$\Delta_{\heis, \p}$ in (\ref{pLap}). These expansions are 
relevant for $\p\geq 2$ because only then the related coefficients can
be interpreted as probabilities. Expansion (\ref{dpp1}) was already present in the
Euclidean setting in \cite{MPR}. The general case of $1<\p<\infty$ is
treated in  section \ref{sec_anyp}, where we follow the
Euclidean construction of \cite{Lew}, superposing
the  ``deterministic'' with ``stochastic averaging'' on the Kor\'{a}nyi ellipsoids
whose orientations and aspects ratios vary within the ``deterministic
averaging'' sampling sets. The same expansions hold if we
replace the constant exponent $\p$ by a variable exponent $\p(\cdot)$,
pertaining to the so-called {\em strong
  $\p(\cdot)$-$\heis$-Laplacian}, as pointed out in Remark \ref{rem_variab}.

\smallskip

Part II consists of four further sections, in which we display the
stochastic interpretation of the $2$-dimensional mean value expansion
(\ref{MV-1P0})$_2$ from section \ref{sec22}.
In section \ref{sec5} we define the $3$-dimensional walk in $\heis$,
whose increments are $2$-dimensional, with the third variable slaved
to the first two via the Levy area process. Our process has  infinite
horizon, but it almost surely accumulates on $\partial\mathcal{D}$,
whereas its expectation yields, in the limit of shrinking sampling radii, 
an $\heis$-harmonic function. The convergence is addressed in section
\ref{secconv_2}; in view of equiboundedness,  it suffices to prove
equicontinuity. We first observe in Lemma \ref{A0},
that this property is equivalent to the seemingly weaker property of equicontinuity at the
boundary. We then introduce the standard notion of 
{\em walk-regularity} of the boundary points, which turns out to be equivalent to the
aforementioned boundary equicontinuity. In section \ref{seccork_2} we
show that domains satisfying the {exterior $\heis$-corkscrew condition} are walk-regular. We prove in section
\ref{sec_ide} that any limit in question must be the viscosity solution to the
$\heis$-harmonic equation with boundary data $F$. By uniqueness of such
solutions \cite{B1, B2}, we obtain the uniform convergence  in the walk-regular case.

\smallskip

In Part III, we follow the same outline as in Part II, but for
the $3$-dimensional asymptotic expansion (\ref{dpp2}) and the nonlinear
operator $\Delta_{\heis, \p}$. In section \ref{sec_setup_p}, we define
the related Tug of War game with noise and its upper and lower
values. These values turn out to be both equal, as shown
in Theorem \ref{are_equal} by a classical martingale argument, 
to the unique, continuous solution of the mean value equation in Theorem
\ref{thD}. The equation (\ref{DPP2}) can be hence seen as a finite difference approximation to
the $\p$-$\heis$-Laplace Dirichlet problem with boundary data $F$;
existence, uniqueness and regularity of its solutions $u_\epsilon$ at each sampling scale $\epsilon$, 
is obtained independently via analytical techniques. 
In particular, each $u_\epsilon$ is continuous up to the boundary,
where it assumes the values of $F$.
In Theorem \ref{conv_nondegene} we show that for $F$ that is already
a restriction of some $\p$-$\heis$-harmonic function with non-vanishing gradient,
the family $\{u_\epsilon\}_{\epsilon\to 0}$ uniformly converges to $F$ at the
rate that is of first order in $\epsilon$. Our proof uses an analytical
argument and it is based on the observation that for $s$ sufficiently
large, the mapping $q\mapsto |q|_K^s$ yields the variation that pushes the
$\p$-$\heis$-harmonic function $F$ into the region of $\p$-subharmonicity. 

In section \ref{sec_convp} we discuss equicontinuity (and thus convergence) of the family
$\{u_\epsilon\}_{\epsilon\to 0}$, for general $F$.
Similarly to Lemma \ref{A0}, this property is equivalent to equicontinuity at the
boundary, which is shown in Theorem \ref{transfer_p}  through the
analytical argument, based on translation and well-posedness of (\ref{DPP2}). We proceed by
defining the {\em game-regularity} of the boundary points; Definition
\ref{def_gamereg}, Lemma \ref{thm_gamenotreg_noconv} and Theorem
\ref{thm_gamereg_conv} mimic the parallel statements in \cite{PS}. 
In section \ref{corkpH} we argue that, similarly to Theorem
\ref{cork_good_2}, domains that satisfy the exterior
$\heis$-corkscrew condition are game-regular. The proof in Theorem \ref{corkthengamereg}
uses the concatenating strategies technique and the annulus
walk estimate taken from \cite{PS}. We again carefully
provide the probabilistic details omitted in \cite{PS}, having in mind
a reader whose training is more analytically-oriented. In section \ref{sec_convp}
we finally conclude that the family $\{u_\epsilon\}_{\epsilon\to 0}$
converges uniformly to the unique viscosity solution to
(\ref{intro3}), in the game-regular case.

We remark that  identical constructions and results of Part III, can be
carried out for the process and the dynamic programming principle modelled on
(\ref{dpp3}) rather than (\ref{dpp2}), where the advantage is that it
covers any exponent in the range $1<\p<\infty$. We indicate the
necessary modifications in Remark 10.2 and 
leave further details to the interested reader; in the Euclidean setting we point to the
paper \cite{Lew}.

\subsection{Notation and preliminaries on the Heisenberg group $\mathbf{\heis}$}
Let $\mathbb{H}=(\mathbb{R}^3, *)$ be the first Heisenberg group,
whose points we typically denote by: 
$$q=(x,y,z)\in \mathbb{H}.$$ 
If needed, we also use the notation $q=(q_1,q_2,q_3) = (q_{hor},
q_3)$. The  group operation is:
$$q*q' = (x,y,z)*(x',y',z')=\big(x+x', y+y', z+z'+\frac{1}{2}(xy'-yx')\big),$$ 
and the Kor\'{a}niy metric $d$ on $\mathbb{H}$ is
given through the Kor\'{a}nyi gauge $|q|_K $ in:
$$d(q,q')=|q^{-1}*q'|_K,\quad |q|_K= \big((x^2+y^2)^2+16z^2\big)^{1/4}.$$
The metric $d$ is left-invariant and one-homogeneous with respect to
the anisotropic dilations: %$\{\rho_\lambda:\mathbb{H}\to\mathbb{H}\}_{\lambda>0}$ given by:
$$\rho_\lambda(x,y,z)=(\lambda x, \lambda y, \lambda^2 z), \qquad
\rho_\lambda:\mathbb{H}\to\mathbb{H}, \quad {\lambda>0}.$$ 
By $B_r(q) = \{q'\in \mathbb{H}; ~ d(q,q')<r\} = q*B_1(0)$ we denote an open
ball with respect to the metric $d$, whereas the Euclidean balls in $n=2,3$ dimensions
are denoted by $B^n_r(q)$. Both types of balls, viewed as subsets of $\mathbb{R}^3$, are convex sets.
%A function $v:\mathbb{H}\supset\mathcal{D}\to\mathbb{R}$ is said to be
%Lipschitz (with a Lipschitz constant $L>0$), whenever: 
%$$|v(q)-v(q')|\leq L \,d(q,q')\qquad \mbox{for all } q,q'\in\mathcal{D}.$$
%We remark that the notions of boundedness in Euclidean and Kor\'{a}niy distances
%are equivalent, while the Euclidean Lipschitz continuity only implies the
%Kor\'{a}niy Lipschitz continuity. The Kor\'{a}niy Lipschitz continuity of
%$v$ in general results only in $v\in \mathcal{C}^{0,1/2}_{loc}(\mathcal{D})$.

\medskip

The differential operators constituting a basis of the Lie algebra on $\mathbb{H}$, are:
$$X=\partial_{x}-\frac{y}{2}\partial_{z},\quad Y=\partial_{y}+\frac{x}{2}\partial_{z},\quad
Z=\partial_{z}.$$
Operators $X,Y$ correspond to differentiating at $q$ in the directions spanning the plane:
$$T_q= \mbox{span}\,\big((1, 0, -\frac{y}{2}), (0,1,
\frac{x}{2})\big) = \big(\frac{y}{2}, - \frac{x}{2},1\big)^\perp.$$
The horizontal gradient and the sub-Laplacian of
a function $v:\heis\to\mathbb{R}$ are:
$$ \nabla_{\mathbb{H}} v= (Xv, Yv), \qquad   \Delta_{\mathbb{H}} v=(X^2+Y^2)v.$$
We will be  concerned with the so-called
$\p$-sub-Laplacian  of $v$, with exponent $\p\in (1,\infty)$:
\begin{equation}\label{pLapdef}
\Delta_{\mathbb{H},\p}v =
X\big(|\nabla_{\mathbb{H}}v|^{\p-2}Xv\big) + Y\big(|\nabla_{\mathbb{H}}v|^{\p-2}Yv\big),
\end{equation}
and with its {\em normalized} (sometimes called {\em game-theoretical}) {\em version}:
\begin{equation}\label{pLapdefN}
\Delta^N_{\mathbb{H},\p}v =\frac{\Delta_{\mathbb{H},\p}v}{|\nabla_{\mathbb{H}}v|^{\p-2}},
\end{equation}
defined whenever $\nabla_{\heis, \p}v\neq 0$.
Clearly,  $\Delta_{\mathbb{H},2} = \Delta_{\mathbb{H},2}^N = \Delta_{\mathbb{H}}$ and it is also easy to check that:
\begin{equation}\label{pLap}
\Delta_{\mathbb{H},\p}v = |\nabla_{\mathbb{H}}v|^{\p-2} \Delta_{\mathbb{H},\p}^N v
=|\nabla_{\mathbb{H}}v|^{\p-2}\big(\Delta_{\mathbb{H}}v + 
(\p-2)\Delta_{\mathbb{H, \infty}}v\big),
\end{equation}
where $\Delta_{\mathbb{H},\infty}$ is the $\infty$-sub-Laplacian given by:
\begin{equation}\label{infLap}\Delta_{\mathbb{H},\infty}v = \big\langle \nabla_{\mathbb{H}}^2v :
\Big(\frac{\nabla_{\mathbb{H}}v}{|\nabla_{\mathbb{H}}v|}\Big)^{\otimes
2} \big\rangle. 
\end{equation}

\subsection{A brief review of results on nonlinear elliptic problems
  in $\mathbf \heis$} 
Many techniques and results valid in the Euclidean case can be extended \cite{7} in the
above context. We now indicate some general statements on the well-posedness of
the Dirichlet problem in $\heis$:
\begin{equation}\label{dirH}
\Delta_{\heis, \p} v =0 \quad \mbox{in }\; \mathcal{D},\qquad v-F\in HW^{1,\p}_0(\mathcal{D}).
\end{equation}
This problem has a unique
weak solution $v\in HW^{1,\p}(\mathcal{D})$, for every data function $F$ in the
horizontal Sobolev space $HW^{1,\p}(\mathcal{D}) = \{v\in
L^\p(\mathcal{D}); ~ \nabla_\heis v\in L^{\p}(\mathcal{D},
\mathbb{R}^2)\}$. We also have existence and uniqueness of solutions to
the corresponding obstacle problem. The $\p$-$\heis$-subsolutions and
supersolutions, as well as the $\p$-$\heis$-subharmonic and superharmonic
functions are defined in the usual manner. The $\p$-$\heis$-subsolutions
have upper semicontinuous representatives that are $\p$-$\heis$-subharmonic. Every
bounded $\p$-$\heis$-subharmonic function has locally $\p$-integrable
horizontal derivatives; in fact it is even quasicontinuous.

It is known that the  horizontal
derivatives of a $\p$-$\heis$-harmonic function are H\"older continuous.  
More precisely: $\osc_{B_r(q)}\nabla_\heis v\leq C
\big(\frac{r}{R}\big)^\alpha \big(\fint_{B_R(q)}|\nabla_\heis {v}|^\p\big)^{1/ \p}$ 
for all $B_r(q)\subset B_{R/2}(q)\subset B_{R}(q)$ compactly contained
in $\mathcal{D}$,  where $\alpha\in (0,1)$ and $C$ depend only on $\p$. 
This was proved for the range $4<\p<\infty$ in \cite{R18}, for
$2\leq \p<\infty$ in \cite{CCLO},  and for $1<\p<\infty$ in \cite{MZ}. 

The standard notion of capacity for the subelliptic setting is studied
in \cite{3da}. This notion coincides with the definition of capacity
based on Radon measures associated to $\p$-$\heis$-subharmonic
functions \cite{7} in $\heis$.
The Wolff potential estimate extends to the subelliptic case, 
and yields a Wiener-type criterion for the attainment of the boundary
values for any $F\in\mathcal{C}(\bar{\mathcal{D}})$ by Perron solutions to (\ref{dirH}). 
We also have a version of the Kellogg-type property stating that 
the set of irregular boundary points, where the boundary value $F$
is not attained, has zero capacity. Points satisfying the exterior
$\mathbb{H}$-corkscrew condition in Definition \ref{cork_def} are
regular \cite{MM}  (in fact, we reprove a version of this statement in section \ref{corkpH}).

General metric spaces with a doubling
measure and supporting a Poincar\'{e} inequality are considered in \cite{KM}.
Perron solutions in such metric spaces are studied in \cite{2},
while \cite{1} contains the adequate notion and discussion of the balayage theory.

\medskip

We remark that for elliptic symmetric equations in non-divergence form:
\begin{equation}\label{homog}
 \text{trace} ( A  \nabla^{2}_{\mathbb{H}}v ) =0
\end{equation}
many results that are classical in the Euclidean setting, remain
open in $\heis$. Let $A:\mathcal{D}\to \mathbb{R}^{2\times
  2}_{sym}$ be measurable, bounded and uniformly elliptic coefficient
matrix. It is not known whether a nonnegative smooth solution $v$ to (\ref{homog})
is locally H\"older continuous (with exponent depending only on the
ellipticity constant $C_0$ of $A$ and $
\|v\|_{L^{\infty}_{loc}}$). In the same context, the satisfaction of the 
Krylov-Safonov-Harnack inequality: $ \sup_{B_{r}(q)} v  \le C_1 \inf_{B_{r}(q)} v$,
where $C_1=C(C_{0},\|v\|_{L^{\infty}(B_{2r}(q))})$ is open. However, similarly
to the Euclidean case, the latter inequality implies the H\"older
continuity via a scaling argument.
When the right hand side of (\ref{homog}) is replaced by $f\in
L^{4}(\mathcal{D})$, the following Alexandroff-Bakelman-Pucci
inequality is expected:
$ \|v\|_{L^{\infty}(B_{r}(q))}\le C_{2} \big(\fint_{B_{r}(q)}
|f(q)|^{4}\dd q\big) ^{1/4}$ with $C_{2}=C(C_{0})$.
Positive resolution of this problem would be a step towards
establishing the Krylov-Safonov-Harnack inequality in $\mathbb{H}$.  
We also mention that there are further open questions regarding the isoperimetric inequality
and the uniqueness of the mean curvature flow in $\mathbb{H}$. 

\subsection{Acknowledgments}
The authors thank the anonymous referees for their thoughtful comments.
M.L. is grateful to Yuval Peres for many helpful discussions and for teaching
her Probability. M.L. was partially supported by the NSF grant DMS-1613153. 

\bigskip

\begin{center}
{\bf PART I: The mean value expansions in $\mathbb{H}$}
\end{center}

\section{The averaging operators $\mathcal{A}_i$ and the mean value
  expansions for $\Delta_{\mathbb{H}}$} \label{sec22}

Given a continuous function $v:\mathbb{H}\to\mathbb{R}$ and a radius
$r>0$, consider the averages at $q\in\mathbb{H}$:
\begin{equation*}%\label{DefMeanValue}
\begin{split}
{\mathcal{A}}_1(v, r)(q)  & = \fint_{\partial B^2_r(0)}v(q*(a,b,0))
~\dd\sigma \\ & = \fint_0^{2\pi}v\big(q+r(\cos\theta, \sin
\theta, \frac{1}{2}(x\sin\theta - y\cos\theta))\big)\dd\theta,\\
{\mathcal{A}}_2(v, r)(q)  & = \fint_{B^2_r(0)}v(q*(a,b,0)) ~\dd(a,b)
%= \fint_{B^2_\epsilon((x_,y))}v(a,b,z+\frac{1}{2}(xb-ya)) ~\dd(a,b).
= \fint_{B^2_1(0)}v\big(q+ r(a,b,\frac{1}{2}(xb-ya))\big) ~\dd(a,b), 
\end{split}
\end{equation*}
\begin{equation*}
\begin{split}
{\mathcal{A}}_3(v, r)(q)  & = \fint_{B_r(q)} v(p)\dd p = \fint_{B_1(0)} v(q*\rho_r(p))\dd p, \\
{\mathcal{A}}_{3,K}(v, r)(q)  & = \frac{1}{\int_{B_r(0)}\Psi(p)\dd p}\cdot
\int_{B_r(q)} \Psi(q^{-1}*p)v(p)\dd p =
\frac{4}{\pi}\int_{B_1(0)} \Psi(p)v(q*\rho_r(p))\dd p.
\end{split}
\end{equation*}
Above, $\Psi$ is the density in the Gauss-Koebe-Levi-Tonelli theorem \cite[Theorem 5.6.3]{book}: 
$$\Psi(q)=\frac{x^{2}+y^{2}}{\big((x^2+y^2)^2+16z^2\big)^{1/2}}
= \frac{|(x,y,0)|_K^2}{ |(x,y,z)|_{K}^{2}}\qquad\mbox{for all }\;
q=(x,y,z)\in\mathbb{H}\setminus \{0\}.$$ 
Other types of $3$-dimensional averages where the ball $B_r(q)$ is
replaced by its ``ellipsoidal'' image under a linear map, will
be considered in section \ref{sec_anyp}.
Recall first the fundamental relation between $\mathcal{A}_{3,K}$ and
the $\mathbb{H}$-harmonic functions:

\begin{Teo}[\cite{book}]\label{3K}
\begin{itemize}
\item[(i)]Let $v\in \mathcal{C}^2(\mathbb{H})$. Then for every $q\in\mathbb{H}$ there holds:
$$\lim_{r\to 0}\frac{1}{r^2}
\Big({\mathcal{A}}_{3,K}(v,r)(q)  - v(q)\Big) = \frac{\pi}{24}\Delta_{\mathbb{H}}v(q).$$
\item[(ii)] If $v\in \mathcal{C}^2(\mathcal{D})$ satisfies
  $\Delta_{\mathbb{H}}v=0$ in some open set $\mathcal{D}\subset\mathbb{H}$ then: 
\begin{equation}\label{basic3K}
v(q) = \mathcal{A}_{3,K}(v,r)(q) \qquad \mbox{for all }\; \bar{B}_r(q)\subset \mathcal{D}.
\end{equation} 
Conversely, if  (\ref{basic3K}) holds for
$v\in\mathcal{C}(\mathcal{D})$, then $v\in\mathcal{C}^\infty(\mathcal{D})$ and 
$\Delta_{\mathbb{H}}v=0$ in $\mathcal{D}$. 
\end{itemize}
\end{Teo}

We now want to develop similar properties of the operators $\mathcal{A}_i$, $i\in\{1,2,3\}$.
Note first that ${\mathcal{A}}_2$ averages the values of $v$ on a $2$-dimensional
ellipse in the plane $q+T_q$, whose horisontal projection
(i.e. projection along the normal direction $e_3$ in $\mathbb{R}^3$) equals
$B_r^2(x, y)$. This ellipse coincides with the intersection of $q+T_q$ and $B_r(q)$. The operator
${\mathcal{A}}_1$ averages $v$ on the boundary of the aforementioned
ellipse and it is also easy to observe that:
\begin{equation}\label{221}
{\mathcal{A}}_2(v, r)(q) = \frac{2}{r^2}\int_0^r s {\mathcal{A}}_1(v,s)(q) \dd s.
\end{equation}

\begin{Rem}
Functions $q\mapsto {\mathcal{A}}_i(v,r)(q)$
are continuous for $v$ continuous. On the other hand, taking $v=\mathbbm{1}_{\{z>0\}}$ we get
${\mathcal{A}}_1(v,\epsilon)(0,0, \cdot) =
{\mathcal{A}}_2(v,\epsilon)(0,0, \cdot) = \mathbbm{1}_{\{z>0\}}$, so
in general ${\mathcal{A}}_1$ and ${\mathcal{A}}_2$ do not return a continuous
function for $v$ discontinuous. Nevertheless, by a classical application
of the monotone class theorem, it follows that for any locally bounded
Borel $v$, the functions $q\mapsto {\mathcal{A}}_1(v,r)(q)$ and  $q\mapsto
{\mathcal{A}}_2(v,r)(q)$ are well defined and locally bounded Borel.
Finally, since the operators ${\mathcal{A}}_3$ and ${\mathcal{A}}_{3,K}$
average on the solid $3$-dimensional Kor\'{a}niy ball $B_r(q)$, they
both return a continuous function for every $v\in L^1_{loc}(\mathbb{H})$. 
\end{Rem}

\medskip

Our first observation is:

\begin{Prop}\label{PropositionMV1Expansion}
Let $v\in \mathcal{C}^2(\mathbb{H})$ and $q\in\mathbb{H}$. 
We have the following expansions, as $r\to 0$:
\begin{equation}\label{MV-1P0}
\begin{split}
 {\mathcal{A}}_1(v,r)(q) & = v(q)+\frac{r^2}{4}\Delta_{\mathbb{H}} v(q)+o(r^2),\\
 {\mathcal{A}}_2(v,r)(q) & = v(q)+\frac{r^2}{8}\Delta_{\mathbb{H}} v(q)+o(r^2).
\end{split}
\end{equation}
In particular, validity of any of the mean value properties $i\in\{1,2\}$ in:
\begin{equation}\label{MV-1P}
v(q)={\mathcal{A}}_i(v,r)(q) \qquad \mbox{for all }\; r\in (0, r_0),
\end{equation} 
implies $\Delta_{\mathbb{H}}v(q)=0$. Conversely, if $\Delta_{\mathbb{H}} v=0$ in some $B_{r_0}(q)$, then
(\ref{MV-1P}) holds for $i\in \{1,2\}$.
\end{Prop}
\begin{proof}
For a fixed $q\in\mathbb{H}$, let $\phi(r)={\mathcal{A}}_2(v, r)(q)$. Clearly, $\phi\in
\mathcal{C}^2(0,\infty)$ and:
\begin{equation}\label{meanValueComp}
\begin{split}
\phi (r)  & = \fint_{B^2_1(0)} v(q*r(a,b,0)) \dd(a,b),\\
\phi '(r)  & = \fint_{B^2_1(0)} \big\langle (a,b), \nabla_\heis \big\rangle v(q*r(a,b,0)) \dd(a,b),\\
\phi ''(r) &=\fint_{B^2_1(0)} \big\langle (a,b)^{\otimes 2} : \nabla_\heis^2 \big\rangle v(q*r(a,b,0))\dd(a,b).
\end{split}
\end{equation}
Passing to the limit, we obtain: $\displaystyle{\lim_{r\to 0}\phi (r) = v(q)}$, 
$\displaystyle{\lim_{r\to 0}\phi '(r) = 0}$ and since $\fint_{B_1^2(0)}
a^2\dd (a,b) = \frac{1}{4}$, it also follows that: $\displaystyle{
\lim_{r\to 0}\phi ''(r) =\frac{1}{4}\Delta_{\mathbb{H}}v(q)}$. We thus conclude
(\ref{MV-1P0})$_2$ by Taylor's theorem at $r=0$. A similar calculation applied to
$r\mapsto\mathcal{A}_1(v,r)(q)$ yields (\ref{MV-1P0})$_1$.
Assume now that $\Delta_{\mathbb{H}} v=0$ in $B_{r_0}(q)$. By the second
formula in \eqref{meanValueComp}, we get:
\begin{equation*}
\begin{split}
\phi '(r) & = \fint_{B^2_1(0)} \langle \hhor{v}(q*r(a,b,0)),(a,b)\rangle\dd(a,b),\\
& =\frac{1}{\pi}\int_0^1\int_{\partial B_s^2(0)}\big\langle
\hhor{v}(q*r(a,b,0))\,,\,\frac{(a,b)}{s}\big\rangle s\dd\sigma (a,b)\mbox{d}s\\ 
&=\frac{1}{\pi}\int_0^1 s \int_{B_s^2(0)} \mbox{div}\hhor{v}(q*r(a,b,0))\dd(a,b)\mbox{d}s\\
&=\frac{1}{\pi}\int_0^1r s\int_{B^2_s(0)}\Delta_{\mathbb{H}}
{v}(q*r(a,b,0))\dd(a,b)\mbox{d}s =0,
\end{split}
\end{equation*}
for all $r\in (0, r_0)$.
%where $\mbox{div}$ stands for the Euclidean divergence operator. 
Consequently, $\phi$ is constant so that:
${\mathcal{A}}_2(v,r)(q)=\phi(r)=\displaystyle{\lim_{r\to
  0}\phi(r)}=v(q)$ as claimed in (\ref{MV-1P}) with $i=2$. Differentiating (\ref{221})
further implies (\ref{MV-1P}) for $i=1$. 
\end{proof}

\medskip

In order to weaken the smoothness assumption in Proposition
\ref{PropositionMV1Expansion}, recall the mollification procedure in
$\mathbb{H}$. Let $J\in\mathcal{C}^\infty_0(\mathbb{H})$ be
a nonnegative test function, supported in $B_1(0)$ and such that
$\int_{\mathbb{H}}J(p)\dd p =1$. For $r>0$, define
$J_r=\frac{1}{r^4}J\circ \rho_{1/r}$ that is supported in $B_r(0)$ and still
satisfying $\int_{\mathbb{H}}J_r(p)\dd p =1$. Given $v\in
L^1_{loc}(\mathbb{H})$ the convolution with $J_r$ is:
$$(v\star J_r)(q) = \int_{\mathbb{H}}v(p) J_r(q*p^{-1})\dd p =
\int_{B_r(0)} v(p^{-1}*q) J_r(p)\dd p.$$
Similarly as in the Euclidean case: $v\star J_r\in\mathcal{C}^\infty(\mathbb{H})$.
Also, the family $v\star J_\epsilon$ converges as $\epsilon\to 0$ to
$v$ in $L^1_{loc}(\mathbb{H})$. When $v\in\mathcal{C}(\mathbb{H})$ then
the convergence is locally uniform and we also note that for  all $i\in \{1,2,3, (3,K)\}$:
\begin{equation}\label{changeorder}
\mathcal{A}_i(v\star J_\epsilon,r) = \mathcal{A}_i(v,r)\star
J_\epsilon\qquad \mbox{ for all }\; \epsilon, r>0.
\end{equation}

\begin{Cor}\label{conti_har}
Let $v\in\mathcal{C}(\mathcal{D})$ on an open set $\mathcal{D}\subset\mathbb{H}$. Validity
of any of the mean value properties $i\in \{1,2\}$ in:
$$v(q) = \mathcal{A}_i(v,r)(q) \qquad \mbox{for all }\; \bar{B}_r(q)\subset \mathcal{D}$$
implies that $v\in\mathcal{C}^\infty(\mathcal{D})$ and $\Delta_{\mathbb{H}}v=0$ in $\mathcal{D}$.
\end{Cor}
\begin{proof}
Fix an open set $U$, compactly contained in $\mathcal{D}$. By
(\ref{changeorder}) the smooth functions $v_\epsilon = v\star
J_\epsilon$ satisfy the mean value property (\ref{MV-1P}) for all
$r,\epsilon$ small enough and all $q\in U$. By Proposition
\ref{PropositionMV1Expansion}, we thus obtain
$\Delta_{\mathbb{H}}v_\epsilon=0$ in $U$. Consequently, (\ref{basic3K}) holds on
$U$ for each $v_\epsilon$ and passing to the uniform
limit with $\epsilon\to 0$, the same property is valid for $v$ as well. Applying Theorem \ref{3K}
(ii), the claim follows on $U$ and thus also on $\mathcal{D}$.
\end{proof}

\medskip

For completeness, we now state the mean value property related to the
operator $\mathcal{A}_3$. We also observe the viscosity version of the
same property, which will be used for the $\mathcal{A}_3$-like
averaging operator developed for the $\p$-Laplacian in Sections
\ref{sec44} and \ref{sec_anyp}.
In the Euclidean setting, viscosity solutions in the
sense of means have been discussed in \cite{KPM}.

\begin{Prop}\label{MV3expansion}
\begin{itemize}
\item[(a)] Let $v\in \mathcal{C}^2(\mathbb{H})$ and $q\in\mathbb{H}$. 
We have the expansion, as $r\to 0$:
\begin{equation}\label{MV-1P3}
 {\mathcal{A}}_{3}(v,r)(q)  = v(q)+\frac{r^2}{3\pi}\Delta_{\mathbb{H}}v(q)+o(r^2).
\end{equation}
In particular, validity of $v(q)={\mathcal{A}}_3(v,r)(q)$ for $r\in (0, r_0)$
implies $\Delta_{\mathbb{H}}v(q)=0$. 
\item[(b)] Let $v\in\mathcal{C}(\mathcal{D})$ on an open set $\mathcal{D}\subset\mathbb{H}$. Validity
of the mean value property (\ref{visc_2}) in the viscosity sense, as
defined below, at every $q\in\mathcal{D}$, is equivalent to:
$v\in\mathcal{C}^\infty(\mathcal{D})$ and $\Delta_{\mathbb{H}}v=0$ in
$\mathcal{D}$. Namely, we say that:
\begin{equation}\label{visc_2}
\mathcal{A}_3(v,r)(q) {\overset{visc}=} v(q) + o(r^2) \qquad \mbox{as }\; r\to 0,
\end{equation}
if and only if the following two conditions are satisfied:
(i) for every $\phi\in\mathcal{C}^2(\mathcal{D})$ such
that $\phi(q) = v(q)$ and $\phi<v $ in $\mathcal{D}\setminus \{q\}$, 
there holds: $\mathcal{A}_3(\phi, r)(q) - \phi(q)\leq o(r^2)$ as $r\to 0$; 
(ii) for every $\phi\in\mathcal{C}^2(\mathcal{D})$ such that $\phi(q) = v(q)$ and $\phi>v$ in 
$\mathcal{D}\setminus \{q\}$,
there holds: $\mathcal{A}_3(\phi, r)(q) - \phi(q)\geq o(r^2)$ as $r\to 0$. 
\end{itemize}
\end{Prop}
\begin{proof}
Expansion (\ref{MV-1P3}) follows in view of $\fint_{B_1(0)}a^2\dd (a,b,c)=2/(3\pi)$,
by an entirely similar calculation as in Proposition \ref{PropositionMV1Expansion} applied to
$\psi(r)=\mathcal{A}_3(v,r)(q)$. 

The proof of (b) is quite standard, hence we only sketch it. Firstly,
by the same argument as in the proof of Theorem 
\ref{unif-topharm}, condition (\ref{visc_2}) is equivalent to $v$ being
viscosity $\heis$-harmonic; this statement is also a
special case of the main result in \cite{FLM}. Secondly, 
let $\bar B_r(q_0)\subset\mathcal{D}$ and consider the $\heis$-harmonic
extension $u$ of $v_{\mid\partial B_r(q_0)}$ on $B_r(q_0)$, namely the unique solution to:
$$\Delta_{\heis}u = 0 \;\mbox{ in } \; B_r(q_0), \qquad u=v \; \mbox{ on } \partial B_r(q_0).$$
We claim that $v\leq u$. Indeed, if $\sup_{B_r(q_0)}(v-u)>0$ then also
the perturbed difference $q\mapsto v(q) - (u(q)-\epsilon |q -p_0|_K^4)$ attains its
maximum in $B_r(q_0)$, if only $\epsilon>0$ is sufficiently small. Here,
$p_0\not\in \bar{\mathcal{D}}$ is some fixed point. Call the said
maximum $\bar q\in B_r(q_0)$; using now $\phi (q) = u(q)-\epsilon
|q -p_0|_K^4$ as a test function in the definition of the viscosity solution, we obtain:
$$0\leq \Delta_{\heis} \phi(\bar q) = -\epsilon
\cdot 24 \big|(\bar{q}-p_0)_{hor}\big|^2 <0,$$
which is a contradiction, proving the claim. In a similar manner, it
follows that $v\geq u$.  Thus $v=u$ is $\heis$-harmonic in $B_r(q_0)$ and
hence in the whole $\mathcal{D}$. Thirdly, it is easy to check that a classical $\heis$-harmonic
function is  viscosity $\heis$-harmonic. This ends the proof.
%fact that the notion of viscosity harmonic 
%coincides with the combined notions of the sub-harmonic and
%super-harmonic, hence classical harmonic.
\end{proof}

\section{The mean value expansion for  $\Delta_{\mathbb{H},\infty}$}  \label{sec33}

In this section, we develop the expansion similar to (\ref{MV-1P3})
but for the fully nonlinear operator $\Delta_{\heis, \infty}$ in
(\ref{infLap}) replacing the linear $\Delta_\heis$. The averaging in
the left hand side of (\ref{MVminmax}) is then what we call the ``deterministic averaging''
$\frac{1}{2}(\sup +\inf)$, as it corresponds to the two players'
choices of moves, in the Tug of War game modelled on the expansion (\ref{MVminmax}),
which is then interpreted as the dynamic programming principle for the
related process. This construction is conceptually similar to
having the ``stochastic averaging'' $\mathcal{A}_3$ correspond to the
Brownian motion. The proof of Theorem \ref{basic_inf} is close to the
arguments in \cite{PS} valid in the Euclidean case; here the application of Lagrange
multipliers yields the bound on the non-horizontal component of any
minimizer/maximizer of $v$ on $B_r(q)$. 

\medskip

Given a function $v\in\mathcal{C}^2(\mathbb{H})$, it is useful to
observe the following Taylor expansions. Firstly, one directly checks
that $\langle \nabla v(q), p-q\rangle = \langle (\nabla_{\mathbb{H}},
Z) v(q), q^{-1}*p\rangle $ and  $\langle \nabla^2 v(q) :
(p-q)^{\otimes 2}\rangle = \langle (\nabla_{\mathbb{H}}, Z)^2 v(q) : (q^{-1}*p)^{\otimes 2} \rangle$. 
Consequently, there holds as $p\to q$:
$$v(p) = v(q) + \langle (\nabla_{\mathbb{H}},
Z) v(q), q^{-1}*p\rangle+ \frac{1}{2} \big\langle (\nabla_{\mathbb{H}},
Z)^2 v(q) : (q^{-1}*p)^{\otimes 2}\big\rangle + o(|p-q|^2).$$
However, since $(q^{-1}*p)^{\otimes 2} e_3 = o(d(p,q)^2)$ and also $o(|p-q|^2)\leq o(d(p,q)^2)$, we obtain the
reduced second order Taylor expansion, valid in $\mathbb{H}$ as $p\to q$:
\begin{equation}\label{TH}
v(p) = v(q) + \langle (\nabla_{\mathbb{H}},
Z) v(q), q^{-1}*p\rangle+ \frac{1}{2} \big\langle \nabla_{\mathbb{H}}^2
v(q)_{sym} : \big((q-p)_{hor}\big)^{\otimes 2}\big\rangle + o(|q^{-1}*p|_K^2).
\end{equation}

\begin{Teo}\label{basic_inf}
Let $v\in \mathcal{C}^2(\mathbb{H})$ and let $q\in\mathbb{H}$. If
$\nabla_{\mathbb{H}}v(q)\neq 0$, then we have the following expansion as $r\to 0$:
\begin{equation}\label{MVminmax}
\frac{1}{2}\big(\inf_{B_r(q)}v + \sup_{B_r(q)}v\big) = v(q) +
\frac{r^2}{2}\Delta_{\mathbb{H},\infty}v(q) + o(r^2).
\end{equation}
\end{Teo}
\begin{proof}
{\bf 1.} Without loss of generality, we may assume that
$v(q)=0$ and $|\nabla_{\mathbb{H}}v(q)|=1$. Consider an approximation of $v$ given by its Taylor
expansion in $\mathbb{H}$:
$$u(p) = \langle a, q^{-1}*p\rangle + \frac{1}{2}\big\langle A: \big((q-p)_{hor}\big)^{\otimes 2} \big\rangle,$$
where:
$$a=(\nabla_\mathbb{H},Z)v(q) =(\nabla_\mathbb{H},Z)u(q)
\in\mathbb{R}^3 \quad\mbox{and}\quad A =
\nabla^2_{\mathbb{H}}v(q)_{sym} = \nabla^2_{\mathbb{H}}u(q)_{sym} \in\mathbb{R}^{2\times 2}.$$
We denote $a=(a_{hor}, a_3)$ and observe that in view of $|a_{hor}|=1$:
$$\Delta_{\mathbb{H},\infty}u(q) = \Delta_{\mathbb{H},\infty}v(q) =
\big\langle A : (a_{hor})^{\otimes 2}\big\rangle.$$
Then by (\ref{TH}) it follows that
$\|u-v\|_{\mathcal{C}(B_r(q))}=o(r^2)$, and consequently:
$$\big|\inf_{B_r(q)}u - \inf_{B_r(q)}v\big| + \big|\sup_{B_r(q)}u -
\sup_{B_r(q)}v\big| = o(r^2).$$

It hence suffices to prove (\ref{MVminmax}) for the approximant $u$.
For each $r>0$ consider the rescaling:
\begin{equation}\label{for}
u_r(p) = \frac{1}{r}u(q*\rho_r(p)) = \langle a, (p_{hor}, rz)\rangle
+ \frac{r}{2}\langle A : p_{hor}\otimes p_{hor}\rangle,
\end{equation}
defined for all $p=(p_{hor},z)=(x,y,z)\in\mathbb{H}$, and note that
$\nabla_{\mathbb{H}}u_r(p) = \nabla_{\mathbb{H}}u(q*\rho_r(p))$ and
$\nabla^2_{\mathbb{H}}u_r(p) = r\nabla^2_{\mathbb{H}}u(q*\delta_r(p))$.
We will prove that as $r\to 0$:
\begin{equation}\label{for2}
\frac{1}{2}\big(\inf_{B_1(0)}u_r + \sup_{B_1(0)}u_r\big) =
\frac{r}{2}\langle A : (a_{hor})^{\otimes 2}\rangle + o(r), 
\end{equation}
which will imply (\ref{MVminmax}) for the function $u$, in view of:
$$\inf_{B_r(q)}u = r\inf_{B_1(0)}u_r, \quad  \sup_{B_r(q)}u =
r\sup_{B_1(0)}u_r \quad \mbox{and}\quad
\Delta_{\mathbb{H},\infty}u(q) = \frac{1}{r}\Delta_{\mathbb{H},\infty}u_r(0).$$

\medskip

{\bf 2.} Let $\bar p^r$, $p^r\in\bar B_1(0)$ be such that $ u_r(\bar
p^r) = \inf_{B_1(0)}u_r$ and $ u_r(p^r) = \sup_{B_1(0)}u_r$. Then for
every $r>0$ such that $r|A|<1$ it follows that
$\nabla u_r(p) =(a_{hor},ra_3) + r(Ap_{hor},0)\neq 0$ for $p\in B_1(0)$, so we actually have:
$$\bar p^r, p^r\in\partial B_1(0).$$
The method of Lagrange multipliers implies that the following vectors are parallel:
$$\nabla u_r(p^r) \parallel (\nabla |p|^4_K)(p^r), \qquad \nabla
u_r(\bar p^r) \parallel (\nabla |p|^4_K)(\bar p^r).$$
Writing $p^r=(p_{hor}^r, z^r)$ this yields: $(a_{hor} + rAp_{hor}^r, r a_3)\parallel
(|p_{hor}^r|^2p_{hor}^r, 8z^r)$ and further:
\begin{equation}\label{LM}
p_{hor}^r = \frac{|p_{hor}^r|}{|a_{hor} + rAp_{hor}^r|} (a_{hor} +
rAp_{hor}^r) \quad \mbox{and}\quad z^r = \frac{1}{8} ra_3 \frac{|p_{hor}^r|^3}{|a_{hor} + rAp_{hor}^r|}.
\end{equation}
Consequently, we get:
$$|a_{hor} - p^r_{hor}|\leq 4|A|r, \qquad |z^r|\leq \frac{1}{4}r|a_3|,$$
which implies:
\begin{equation*}
\begin{split}
0& \leq u_r(p^r) - u_r((a_{hor},0)) \\ & = \langle a_{hor}, p^r_{hor}\rangle
-1 + ra_3z^r + \frac{r}{2}\langle A : (p_{hor}^r)^{\otimes 2} -
(a_{hor})^{\otimes 2} \rangle\\ & \leq \langle a_{hor}, p^r_{hor}\rangle
-1 + \frac{1}{4}|a_3|^2r^2 + r|A| \cdot |p_{hor}^r- a_{hor}| \leq \big(4|A|^2+\frac{1}{4}|a_3|^2\big)r^2.
\end{split}
\end{equation*}

Likewise, for the minimizer $\bar p^r$ (rather than the maximizer
$p^r$ above) we have:
$$0\geq u_r(\bar p^r) - u_r((-a_{hor},0)) \geq \big(4|A|^2 -\frac{1}{4}|a_3|^2\big)r^2,$$
which results in (\ref{for2}) because $u_r((a_{hor},0)) +
u_r((-a_{hor},0)) = r\langle A: (a_{hor})^{\otimes 2}\rangle$. 
\end{proof}

\section{Two mean value expansions for $\Delta_{\mathbb{H},\p}$: $\p>2$}  \label{sec44} 

Combining the mean value expansions and the averaging operators
developed: for $\Delta_\heis$ in section \ref{sec22}, and
for $\Delta_{\heis,\infty}$ in section \ref{sec33}, we now state two
mean value expansions for $\Delta_{\heis,\p}$. Heuristically, the first formula
(\ref{dpp1}) below, views the normalisation $ \Delta^N_{\mathbb{H},
    \p}$ directly through the interpolation
(\ref{pLap}). The related averaging operator is then the superposition of:
\begin{itemize}
\item[(i)] ``simple averaging'' with prescribed weights $\alpha_\p$,
$(1-\alpha_\p)$,
\item[(ii)] ``stochastic averaging'' $\mathcal{A}_3$,
\item[(iii)] ``deterministic averaging'' $\frac{1}{2}(\sup +\inf)$. 
\end{itemize}
The expansion (\ref{dpp1}) holds for any $\p>1$, however
the ``simple averaging'' coefficients are feasible, in the sense that
having $\alpha_\p\in [0,1]$ allows for their interpretation as
probabilities of the stochastic versus deterministic sampling,
only for $\p\geq 2$. In the Euclidean setting, the parallel formula has
been implemented as the dynamic programming principle for Tug
of War game with noise in \cite{MPR}. 

Our second mean value expansion (\ref{dpp2}) reflects the uniform
``simple averaging" of: (i) ``stochastic averaging'' and (ii)
``deterministic averaging'' applied to the further stochastic one. The
fact that the smoothing $\mathcal{A}_3$ is present in all
three terms, results in automatic continuity of solutions to the
dynamic programming principle (see
section \ref{sec_dpp_analysis}); compare to the analysis in
\cite{MPR} that has been based on (\ref{dpp1}) and thus necessitated measurable
approximations. Again, the mean value expansion (\ref{dpp2}) works only in the
limited range of exponents $\p>2$. In section \ref{sec_anyp} we will
present yet another mean value operator in the Heisenberg
group, pertaining to the general case of $\p\in (1,\infty)$.

\begin{figure}[htbp]
\centering
\includegraphics[scale=0.34]{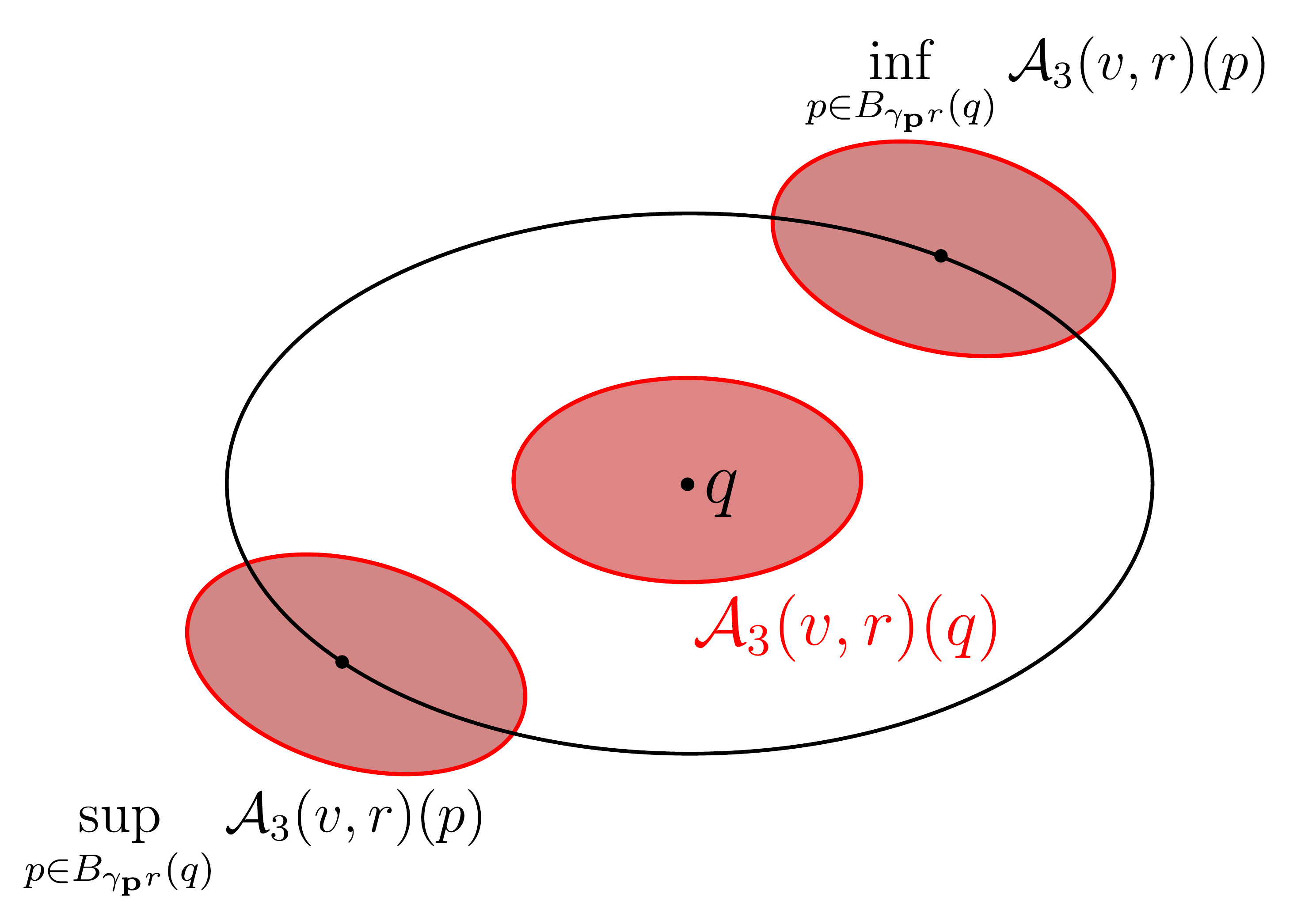}
    \caption{{The three  averaging contributions in the formula (\ref{dpp2}).}}
\label{f:dpp2}
\end{figure}

\begin{Teo}\label{teo41}
Let $v\in \mathcal{C}^2(\mathbb{H})$ and let $q\in\mathbb{H}$. If
$\nabla_{\mathbb{H}}v(q)\neq 0$ then
we have the following expansions below, valid as $r\to 0$: 
\begin{itemize}
\item[(i)] For $\p>1$ define
  $\displaystyle{\alpha_{\p}=\frac{3\pi}{2(\p-2)+3\pi}}$ and
  $\displaystyle{\beta_{\p}=\frac{2(\p-2)}{2(\p-2)+3\pi}}$. Then: 
\begin{equation}\label{dpp1}
\begin{split}
\alpha_\p\mathcal{A}_3(v,r)(q) +&\frac{\beta_\p}{2}\Big(
\inf_{p\in B_r(q)}v(p) + \sup_{p\in B_r(q)}v(p)\Big) \\ & = v(q) +
\frac{r^2}{2(\p-2)+3\pi}\cdot \frac{\Delta_{\mathbb{H}, 
  \p}v(q)}{|\nabla_{\mathbb{H}}v(q)|^{\p-2}} +  o(r^2).
\end{split}
\end{equation}
In particular, for $\p=2$ we recover the expansion (\ref{MV-1P3}).
\item[(ii)] For $\p>2$ define $\displaystyle{\gamma_{\p}=\big(\frac{\p-2}{\pi}\big)^{1/2}}$. Then: 
\begin{equation}\label{dpp2}
\begin{split}
\frac{1}{3}\mathcal{A}_3(v,r)(q)  + &\frac{1}{3}\inf_{p\in
  B_{\gamma_{\p} r}(q)}  \mathcal{A}_3(v,r)(p) +\frac{1}{3}\sup_{p\in
  B_{\gamma_{\p} r}(q)}  \mathcal{A}_3(v,r)(p) \\ & = v(q) + 
\frac{r^2}{3\pi}\cdot \frac{\Delta_{\mathbb{H}, 
  \p}v(q)}{|\nabla_{\mathbb{H}}v(q)|^{\p-2}} + o(r^2). 
\end{split}
\end{equation}
Again, the harmonic expansion (\ref{MV-1P3}) is recovered asymptotically as $\p\to 2^+$.
\end{itemize}
\end{Teo}
\begin{proof}
{\bf 1.} Summing expansions (\ref{MV-1P3}) and (\ref{MVminmax}) weighted
with coefficients $\alpha_\p$ and $\beta_\p$, we get:
\begin{equation*}
\begin{split}
\alpha_\p\mathcal{A}_3(v,r)(q) & +\frac{\beta_\p}{2}\Big(
\inf_{B_r(q)}v + \sup_{B_r(q)}v\Big)  = v(q) +
\Big(\frac{\alpha_\p}{3\pi} \Delta_{\mathbb{H}}v(q) +
\frac{\beta_\p}{2}\Delta_{\mathbb{H}, \infty}v(q)\Big)r^2 + o(r^2) \\
& = v(q) +\frac{\alpha_\p r^2}{3\pi} \Big( \Delta_{\mathbb{H}}v(q) +
(\p-2)\Delta_{\mathbb{H}, \infty}v(q)\Big)+o(r^2) \\ &
= v(q) + \frac{\alpha_\p r^2}{3\pi}|\nabla_{\mathbb{H}}v(q)|^{2-\p}\Delta_{\mathbb{H},
  \p}v(q) + o(r^2),
\end{split}
\end{equation*}
because $3\pi \beta_\p/ (2\alpha_\p)=\p-2$ and $\alpha_\p + \beta_\p=1$, proving (\ref{dpp1}).

\medskip

{\bf 2.} To show (\ref{dpp2}), consider the function
$u(p)=\mathcal{A}_3(v,r)(p)$ and note that since
$\nabla_{\mathbb{H}}v(q)\neq 0$ we also have: $\nabla_{\mathbb{H}}u(q) =
\mathcal{A}_3(\nabla_{\mathbb{H}}v, r)(q)\neq 0$. We may thus apply (\ref{MVminmax}) to
$u$ and obtain:
$$ \frac{1}{2}\Big(\inf_{p\in B_{\gamma_{\p} r}(q)}  \mathcal{A}_3(v,r)(p) + \sup_{p\in
  B_{\gamma_{\p} r}(q)}  \mathcal{A}_3(v,r)(p) \Big) =
\mathcal{A}_3(v,r)(q) + \frac{1}{2}\gamma_\p^2r^2\Delta_{\mathbb{H}, \infty}u(q) + o(r^2).$$
Since:
$$\Delta_{\mathbb{H}, \infty}u(q) = \Big\langle \mathcal{A}_3(\nabla^2_{\mathbb{H}} v,r)(q) :
\Big(\frac{\mathcal{A}_3(\nabla_{\mathbb{H}}
  v,r)(q)}{|\mathcal{A}_3(\nabla_{\mathbb{H}} v,r)(q)|}\Big)^{\otimes 2} \Big\rangle =
\Delta_{\mathbb{H},\infty}v(q) + o(1),$$
it follows in view of (\ref{MV-1P3}) that:
\begin{equation*}
\begin{split}
\frac{1}{3}\mathcal{A}_3(v,r)(q) & + \frac{1}{3}\inf_{B_{\gamma_{\p}
    r}(q)}  \mathcal{A}_3(v,r) +\frac{1}{3}\sup_{B_{\gamma_{\p} r}(q)}
\mathcal{A}_3(v,r) \\ & = \mathcal{A}_3(v,r)(q) + \frac{\gamma_\p^2 r^2}{3}\Delta_{\mathbb{H},
  \infty}v(q) + o(r^2) \\ & = v(q) + \Big( \frac{1}{3\pi}
\Delta_{\mathbb{H}}v(q) + \frac{\gamma_\p^2}{3}\Delta_{\mathbb{H}, 
  \infty}v(q)\Big)r^2 + o(r^2) \\ & = v(q) + \frac{r^2}{3\pi}\Big( 
\Delta_{\mathbb{H}}v(q) + (\p-2)\Delta_{\mathbb{H}, 
  \infty}v(q)\Big) + o(r^2)\\ &
= v(q) + \frac{r^2}{3\pi}|\nabla_{\mathbb{H}, \p}v(q)|^{2-\p}\Delta_{\mathbb{H}, \p}v(q) + o(r^2),
\end{split}
\end{equation*}
because $\gamma_\p^2\pi=\p-2$. The proof is done.
\end{proof}

\begin{Rem}\label{rem_variab}
Statement (i) of Theorem \ref{teo41} also holds for $\p=1$, as noted by the reviewers.
The same expansion (\ref{dpp2}) holds if we replace the constant
exponent $\p$ by a variable exponent $\p(\cdot)>2$, retaining the
scaling factor $\gamma_\p=\big(\p(\cdot)-2)/\pi\big)^{1/2}$. This formulation
can be applied to the so-called {\em strong $\p(\cdot)$-Laplacian}: 
$$ \Delta^S_{\mathbb{H}, \p(\cdot)}v(q) = |\nabla_\heis v(q)|^{\p(q)-2}
\big(\Delta_{\heis} v(q) +(\p(q)-2)\Delta_{\mathbb{H}, \infty}v(q)\big).$$
We remark that there are different and non-equivalent ways of extending
the constant exponent $\p$-$\heis$-Laplacian $\Delta_{\mathbb{H}, \p}$
to the variable exponent case \cite{M}. 
The strong $\p(\cdot)$-Laplacian was introduced in the Euclidean
setting in \cite{AH}, and for regular functions it satisfies: 
$$\Delta^S_{\mathbb{H},
  \p(\cdot)}v(q)=\Delta_{\mathbb{H},\p(\cdot)}v(q)-
|\nabla_\heis{v(q)}|^{\p(q)-2}\log(|\nabla_\heis{v(q)}|)\langle\nabla_\heis{v(q)},
\nabla_\heis{\p(q)}\rangle. $$
This connection has also been studied for weak solutions
in the Euclidean case, \cite{Sita}.
Here, $\Delta_{\mathbb{H},\p(\cdot)}$ is a particular version of the $\p(\cdot)$-$\heis$-Laplacian, resulting
by taking the Euler-Lagrange equation of the functional
$\mathcal{E}(v) = \int_{\mathcal{D}} \frac{1}{\p(q)} |\nabla_\heis
v(q)|^{\p(q)}\dd q$, namely: 
$$\Delta_{\mathbb{H},\p(\cdot)}v(q)=\hordiv\big(|\nabla_\heis{v(q)}|^{\p(q)-2}\nabla_\heis{v(q)}\big).$$
A version of random Tug of War game in the context of
the parabolic strong $\p(x,t)$-equation in the Euclidean setting,
has been developped in \cite{PR}. There,  the process is modelled on the asymptotic 
expansion (\ref{dpp1}) and results in the
discontinuous approximations $u_\epsilon$. In our work, the game
values in (\ref{DPP2}), modelled on the expansion (\ref{dpp2}), have boundary-implied regularity.
\end{Rem}

\section{The anisotropic mean value expansion for $\Delta_{\mathbb{H},\p}$: $1<\p<\infty$}  
\label{sec_anyp}

We now propose another mean value expansion that, unlike (\ref{dpp1})
and (\ref{dpp2}), leads to the dynamic programming
principle that works for any exponent $\p\in (1,\infty)$. 
The key idea, developed in the Euclidean setting in \cite{Lew}, is to superpose: 
\begin{itemize}
\item[(i)] ``deterministic averaging'' $\frac{1}{2}(\sup + \inf)$
on  Kor\'{a}nyi balls, with 
\item[(ii)] ``stochastic averaging'' $\mathcal{A}_3$ on
the ``Kor\'{a}nyi ellipsoids'' defined as the images of a unit ball under suitable
linear transformations. 
\end{itemize}
We begin by the counterpart of Proposition
\ref{PropositionMV1Expansion} on such ellipsoids, defined as follows.

\medskip 

For a radius $r>0$, an aspect ratio $\alpha>0$ and an orientation
vector $\nu=(\nu_{hor},\nu_3)\in\heis$ that we normalize to be of unit
Euclidean length: $|\nu_{hor}|^2+\nu_3^2=1$, we set the {\em
 Kor\'{a}nyi ellipsoid} centered at a given $q\in\heis$ to be:
\begin{equation}\label{KE}
E(q,r;\alpha, \nu) = q * \rho_r\big\{p+(\alpha-1)\langle p, \nu\rangle\nu;~ p\in B_1(0)\big\}.
\end{equation}
The open, bounded, smooth set $E(q,r;\alpha, \nu) \subset\heis$ is thus obtained by applying the linear map:
$$p\mapsto L(p;\alpha, \nu)=(p-\langle p, \nu\rangle\nu) +
\alpha\langle p, \nu\rangle \nu$$
to the unit  Kor\'{a}nyi ball $B_1(0)$, then scaling via Heisenberg
dilation $\rho_r$, and centering the image at $q$ by the
group operation.
Given a continuous function $v\in\mathcal{C}^2(\heis)$, define the average:
$$\mathcal{A}_3(v,r;\alpha,\nu)(q) = \fint_{E(q,r;\alpha,\nu)} v(p)\dd p =
\fint_{B_1(0)}v\big(q*\delta_r L(p;\alpha,\nu)\big)\dd p.$$
Observe that $E(q,r;1,\nu)=B_r(q)$, so likewise:
$\mathcal{A}_3(v,r;1,\nu) = \mathcal{A}_3(v,r)$ for all orientations $\nu$.

\smallskip

\begin{Prop}\label{PropMV_anyp}
Let $v\in \mathcal{C}^2(\mathbb{H})$ and $q\in\mathbb{H}$. 
We have the following expansion, as $r\to 0$:
\begin{equation}\label{MV_anyp}
 {\mathcal{A}}_3(v,r;\alpha,\nu)(q)  =  v(q)+\frac{r^2}{3\pi}\Big(\Delta_{\mathbb{H}} v(q)
+ (\alpha^2-1)\big\langle \nabla^2_\heis v(q) : \nu_{hor}^{\otimes 2}\big\rangle\Big) +o(r^2).
\end{equation}
\end{Prop}
\begin{proof}
As in the proof of Proposition \ref{PropositionMV1Expansion}, define
the auxiliary function $\phi(r)={\mathcal{A}}_3(v, r;\alpha, \nu)(q)$. Then $\phi\in
\mathcal{C}^2(0,\infty)$ and it is easy to compute that:
\begin{equation*}
\begin{split}
\phi (r)  & = \fint_{B_1(0)} v\Big(q* \big(rp_{hor} + r(\alpha-1)\langle
p, \nu\rangle \nu_{hor}, r^2p_3 + r^2(\alpha-1)\langle p, \nu\rangle\nu_3\big)\Big) \dd p,\\
\phi '(r)  & = \fint_{B_1(0)} \Big\langle p_{hor} + (\alpha-1)\langle
p, \nu\rangle \nu_{hor}, \nabla_\heis\Big\rangle \cdot v\big(q*\rho_r L(p; \alpha, \nu)\big) 
\\ & \qquad\qquad\quad + 2r \big(p_3 + (\alpha-1)\langle p,
\nu\rangle\nu_3\big)\cdot Z v\big(q*\rho_r L(p; \alpha, \nu)\big) \dd p.
\end{split}
\end{equation*}
Further, $B_1(0)$ being symmetric implies:
\begin{equation*}
\begin{split}
\lim_{r\to 0}\phi'' (r) & =  \lim_{r\to 0} \fint_{B_1(0)} \Big\langle \big(p_{hor} + (\alpha-1)\langle
p, \nu\rangle \nu_{hor}\big)^{\otimes 2} : \nabla^2_\heis v\big(q*\rho_r L(p; \alpha, \nu)\big) \Big\rangle 
\\ & \qquad\qquad\qquad \quad + 2\big(p_3 + (\alpha-1)\langle p,
\nu\rangle\nu_3\big)\cdot Z v\big(q*\rho_r L(p; \alpha, \nu)\big) \dd p \\ & =
\Big\langle \fint_{B_1(0)}\big(p_{hor} + (\alpha-1)\langle
p, \nu\rangle \nu_{hor}\big)^{\otimes 2}\dd p : \nabla^2_\heis v(q)\Big\rangle. 
% + 2 \fint_{B_1(0)} p_3 + (\alpha-1)\langle p, \nu\rangle\nu_3\dd p \cdot Z v(q)
\end{split}
\end{equation*}
Expanding the first matrix expression in the right hand side above to: 
$$ \fint_{B_1(0)}p_{hor} ^{\otimes 2}\dd p + 2(\alpha-1)\fint_{B_1(0)}\langle p, \nu\rangle
(p_{hor}\otimes\nu_{hor})_{sym}\dd p + (\alpha-1)^2\fint_{B_1(0)} \langle p,
\nu\rangle^2\dd p \cdot \nu_{hor}^{\otimes 2}$$ 
and using $\fint_{B_1(0)}a^2\dd (a,b,c) = \frac{2}{3\pi}$ with
$|\nu|=1$, to compute $\fint_{B_1(0)}p_{hor} ^{\otimes 2}\dd p
=\frac{2}{3\pi}Id_2$, and:
$$\fint_{B_1(0)}\langle p, \nu\rangle (p_{hor}\otimes\nu_{hor})_{sym}\dd
p = \frac{2}{3\pi}\nu_{hor}^{\otimes 2},\qquad 
\fint_{B_1(0)} \langle p, \nu\rangle^2\dd p =\frac{2}{3\pi},$$
we conclude that:
\begin{equation*}
\lim_{r\to 0}\phi'' (r) =  \frac{2}{3\pi} \Big\langle Id_2+
(\alpha^2-1) \nu_{hor}^{\otimes 2}  : \nabla^2_\heis v(q)\Big\rangle. 
\end{equation*}
As $\displaystyle{\lim_{r\to 0}\phi (r) = v(q)}$ and
$\displaystyle{\lim_{r\to 0}\phi '(r) = 0}$, the result follows by
applying Taylor's theorem at $r=0$.
\end{proof}

\medskip

It is clear that by choosing $\alpha=\sqrt{\p-1}$ and $\nu = \big(\frac{\nabla_\heis
  v(q)}{|\nabla_\heis v(q)|},0\big)$, in virtue of the interpolation (\ref{pLap}) we obtain:
$\mathcal{A}_3(v,r;\alpha,\nu) = v(q) +\frac{r^2|\nabla_\heis v(q)|^{2-\p}}{3\pi}\Delta_{\heis,\p}v(q) + o(r^2)$.
In order to derive a mean value expansion where the left hand side
averaging does not require the knowledge of $\nabla_\heis v(q)$ and
allows for the identification of a $\p$-$\heis$-harmonic function that is a
priori only continuous, we need to additionally average over all
equally probable horizontal vectors $\nu_{hor}$. Since only such
horizontal orientations are relevant, we observe that the related  Kor\'{a}nyi ellipsoid in (\ref{KE}):
$$E(0,1;\alpha, (\nu_{hor},0))= \Big\{\big(p_{hor}+(\alpha-1)\langle
p_{hor}, \nu_{hor}\rangle \nu_{hor}, p_3\big); ~ p\in B_1(0)\Big\},$$
can be interpreted as the ``Kor\'{a}nyi lift'' of the two-dimensional ellipse with radius $1$:
$$ \Big\{p_{hor}+(\alpha-1)\langle
p_{hor}, \nu_{hor}\rangle \nu_{hor}; ~ p_{hor}\in B_1^2(0)\Big\}\subset\mathbb{R}^2.$$
We remark that the expansion (\ref{dpp3}) is related to
another interpolation property of $\Delta_{\heis, \p}$:
\begin{equation*}
\Delta_{\heis,\p} v = |\nabla v|^{\p-2}\Big(|\nabla_\heis v|\Delta_{\heis,
  1} v + (\p-1)\Delta_{\heis, \infty} v\Big),
\end{equation*}
which has first appeared, in the context of the applications of
$\Delta_\p$ to image recognition, in \cite{Kaw}.

\begin{figure}[htbp]
\centering
\includegraphics[scale=0.5]{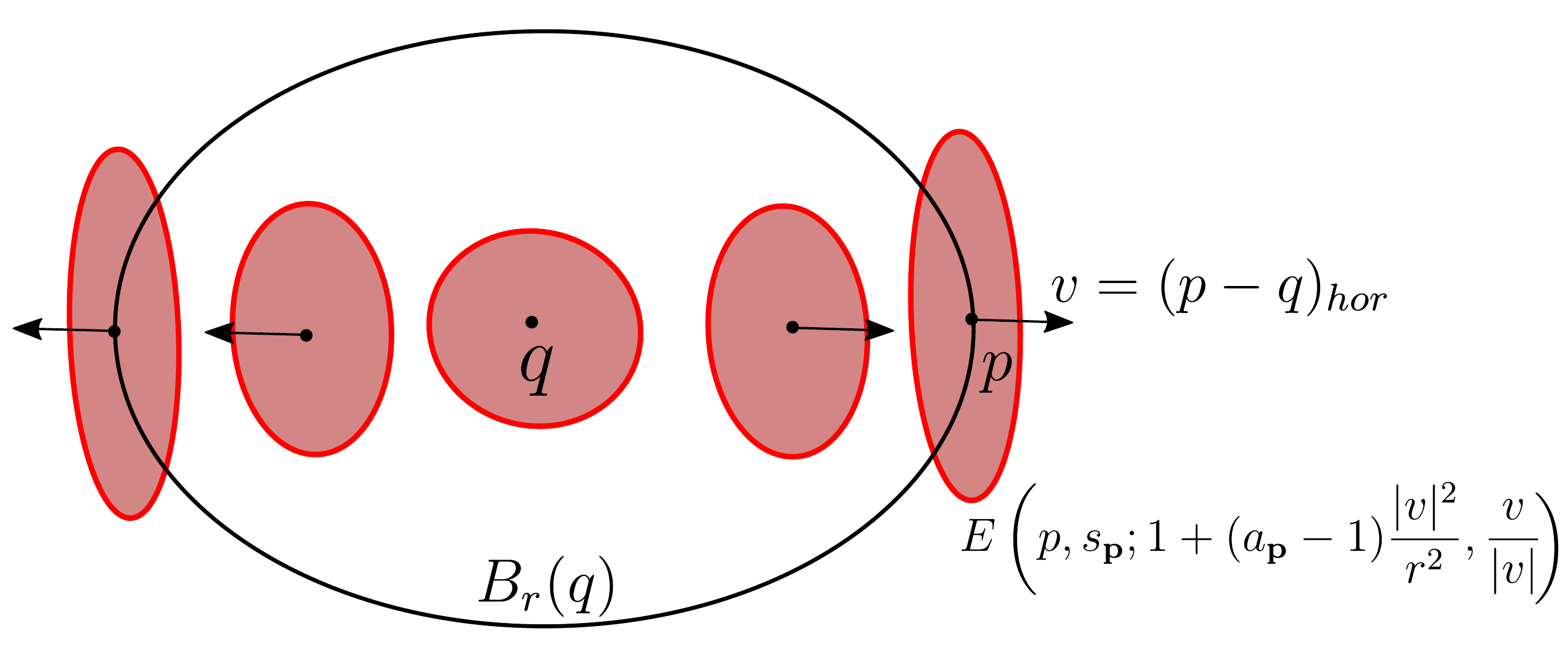}
    \caption{{The ``stochastic sampling'' domains, centered at
        various positions $p$ within the ``deterministic sampling''
        domain at $q$, in the expansion (\ref{dpp3}).}}
\label{f:dpp3_distrib}
\end{figure}

\begin{Teo}\label{th_anyp}
Let $v\in \mathcal{C}^2(\mathbb{H})$ and  $q\in\mathbb{H}$ be such that
$\nabla_{\mathbb{H}}v(q)\neq 0$. Given $1<\p<\infty$, let the scaling exponents $s_\p, a_\p>0$ satisfy:
\begin{equation}\label{spAndap}
\frac{3\pi}{2s_\p^2}+a_\p^2=\p-1.
\end{equation}

Then, the following expansion is valid as $r\to 0$: 
\begin{equation}\label{dpp3}
\begin{split}
\frac{1}{2}\Big(\inf_{p\in B_r(q)} + \sup_{p\in B_r(q)}\Big)
\mathcal{A}_3& \Big(v,s_\p r; 1+(a_\p -1)
\frac{|(p-q)_{hor}|^2}{r^2}, \frac{(p-q)_{hor}}{|(p-q)_{hor}|}\Big)(p)\\ & = v(q) +
\frac{r^2}{\p-1}\cdot \frac{\Delta_{\mathbb{H}, 
  \p}v(q)}{|\nabla_{\mathbb{H}}v(q)|^{\p-2}} +  o(r^2).
\end{split}
\end{equation}
\end{Teo}
\begin{proof}
{\bf 1.} In the statement (\ref{dpp3}) and below, we often write
$\nu_{hor}$ instead of $(\nu_{hor}, 0)\in\heis$, if no ambiguity
arises. We define the following continuous function $B_r(q)\ni p\mapsto f_r(p)$:
$$f_r(p) = \mathcal{A}_3 \Big(v,s_\p r; 1+(a_\p -1)
\frac{|(p-q)_{hor}|^2}{r^2}, \frac{(p-q)_{hor}}{|(p-q)_{hor}|}\Big)(p).$$
In particular, when $(p-q)_{hor}=0$, the above formula still makes sense and returns:
$f_r(p) = \mathcal{A}_3 (v,s_\p r)(p)=\fint_{B_{s_\p
    r}(p)}v$. Applying Proposition \ref{PropMV_anyp} and the Taylor
expansion in (\ref{TH}), we get:
\begin{equation*}
\begin{split}
f_r(p) & = v(p) + \frac{r^2s_\p^2}{3\pi}\Delta_{\heis}v(p) \\
& \qquad\quad \; + \frac{s_\p^2(a_\p-1)}{3\pi}\Big(2 + (a_\p-1)
\frac{|(p-q)_{hor}|^2}{r^2}\Big) \big\langle\nabla^2_{\heis}v(p) :
(p-q)_{hor}^{\otimes 2}\big\rangle + o(r^2) \\ & = 
v(q) + \frac{r^2s_\p^2}{3\pi}\Delta_{\heis}v(q) + \big\langle
(\nabla_\heis, Z) v(q), q^{-1}*p\big\rangle \\
& \qquad\quad \; + \Big(\frac{1}{2} + \frac{s_\p^2(a_\p-1)}{3\pi}\Big(2 + (a_\p-1)
\frac{|(p-q)_{hor}|^2}{r^2}\Big)\Big) \big\langle\nabla^2_{\heis}v(q) :
(p-q)_{hor}^{\otimes 2}\big\rangle + o(r^2) \\ & = \bar f_r(p) + o(r^2),
\end{split}
\end{equation*}
because $o(|q^{-1}*p|^2_K)$ can be replaced by $o(r^2)$ for $p\in
B_r(q)$. The left hand side of (\ref{dpp3}) is thus:
\begin{equation}\label{gru}
\frac{1}{2}\Big(\inf_{p\in B_r(q)} + \sup_{p\in B_r(q)}\Big) f_r(p) =
\frac{1}{2}\Big(\inf_{p\in B_r(q)} + \sup_{p\in B_r(q)}\Big) \bar f_r(p) + o(r^2).
\end{equation}
Observe that Lemma \ref{basic_inf} cannot be used directly to find the
principal term in the expansion of the right hand side above, even
though $ \nabla_\heis \bar f_r(q) = \nabla_\heis v(q)\neq 0$, simply because
the function to be minimized/maximized depends on $r$.
We may however use the argument in the second step of proof
of (\ref{MVminmax}), as completed below.

\medskip

{\bf 2.} We write $\bar f_r(q*\rho_r (p)) = v(q) +
\frac{r^2s_\p^2}{3\pi}\Delta_\heis v(q) + rg_r(p)$ for $p\in B_1(0)$, so that:
\begin{equation}\label{defig}
\begin{split}
& g_r(p) = \langle a_{hor}, p_{hor}\rangle + r\Big( a_3p_3 + \langle A:
p_{hor}^{\otimes 2}\rangle +|p_{hor}|^2\langle B: p_{hor}^{\otimes 2}\rangle\Big),\\
& \mbox{where: }\; a=(a_{hor}, a_3) = (\nabla_{\heis}, Z)v(q)\in \mathbb{R}^3,
\\ & \mbox{and: }\;\; \;\; A= \Big(\frac{1}{2}+\frac{2s_\p^2(a_\p-1)}{3\pi}\Big)\nabla^2_\heis
v(q)\in \mathbb{R}^{2\times 2}, \quad  B= \frac{s_\p^2(a_\p-1)^2}{3\pi}\nabla^2_\heis
v(q)\in \mathbb{R}^{2\times 2}.
\end{split}
\end{equation}
Let $\bar p^r, p^r\in \bar B_1(0)$ be, respectively, a minimizer and a
maximizer of $g_r$ on $\bar B_1(0)$. Since in view of $a_{hor}\neq 0$
there holds $\nabla g_r\neq 0$ in
$B_1(0)$ for all $r$ sufficiently small, it follows that $ \bar p^r, p^r\in \partial B_1(0)$. Further, the method of
Lagrange multipliers yields: $\nabla g_r(p^r) \parallel \nabla
(|p|_K^4)(p^r)$ so as in (\ref{LM}): 
$$p_3^r = \frac{1}{8}ra_3 \frac{|p_{hor}^r|^3}{\big |a_{hor}+
  2r\big(Ap_{hor}^r+|p_{hor}^r|^2 Bp_{hor}^r + \langle B
  :(p_{hor}^r)^{\otimes 2}\rangle  p_{hor}^r \big)\big|}.$$
In particular, for all $r$ sufficiently small, we obtain:
\begin{equation}\label{LM2}
|p_3^r|\leq \frac{1}{4}r |a_3|.
\end{equation}

\begin{figure}[htbp]
\centering
\includegraphics[scale=0.49]{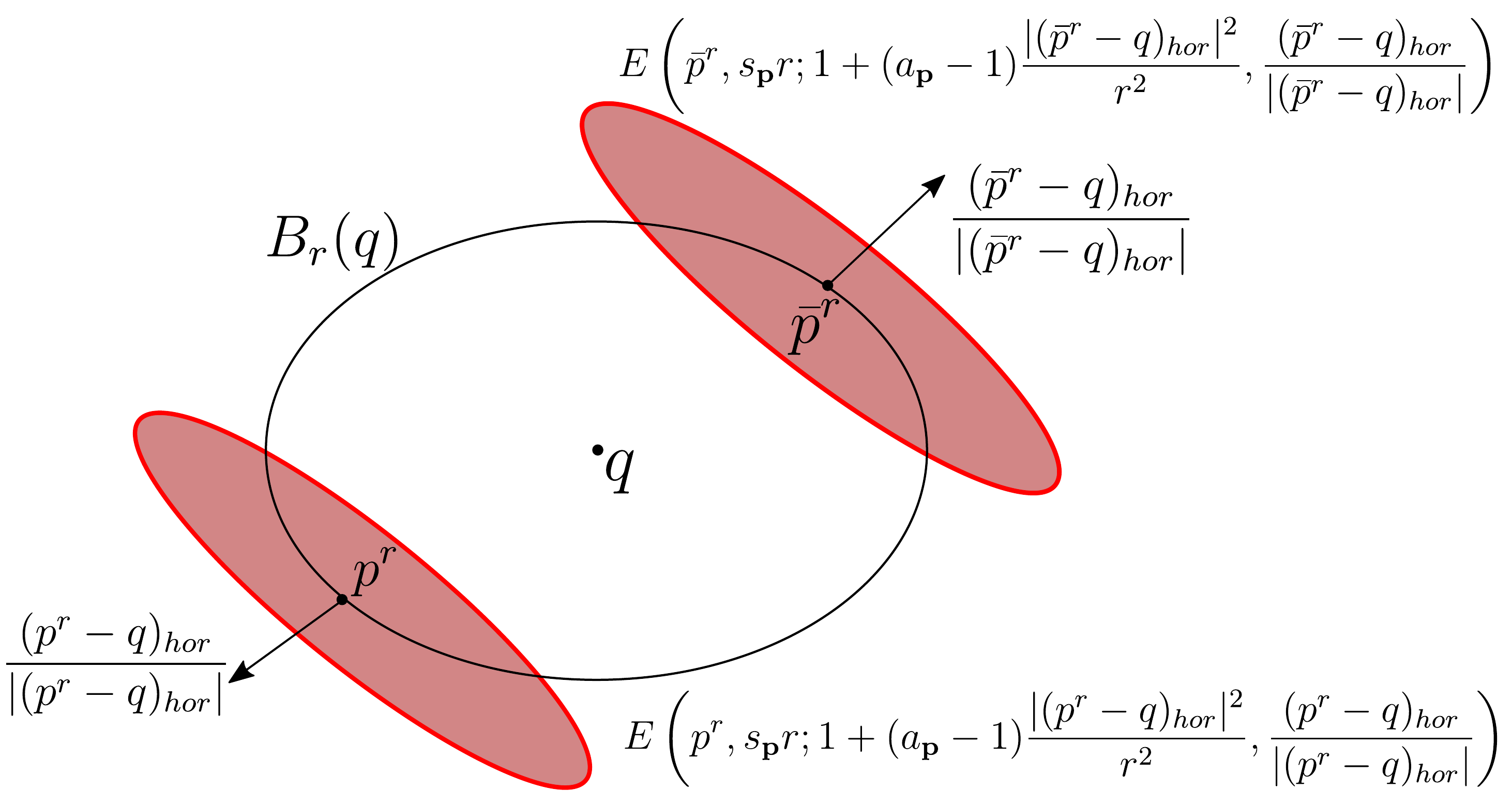}
    \caption{{The two averaging contributions in the formula (\ref{dpp3}).}}
\label{f:dpp3}
\end{figure}

We now observe that:
\begin{equation*}
\begin{split}
0 & \leq g_r(p^r) - g_r\big(\frac{a_{hor}}{|a_{hor}|}\big) \\ & = \langle
a_{hor}, p^r_{hor}\rangle - |a_{hor}| + ra_3p_3^r \\ & \qquad + r \big\langle A:
(p^r_{hor})^{\otimes 2}- \big(\frac{a_{hor}}{|a_{hor}|}\big)^{\otimes 2}\big\rangle
+ |p^r_{hor}|^2\langle B: (p^r_{hor})^{\otimes 2}\rangle
- \big\langle B: \big(\frac{a_{hor}}{|a_{hor}|}\big)^{\otimes 2}\big\rangle
\\ & \leq \langle a_{hor}, p^r_{hor}\rangle - |a_{hor}| + \frac{1}{4}r^2a_3^2 +
2r\big(|A|+2|B|\big)\big|p_{hor}^r - \frac{a_{hor}}{|a_{hor}|}\big|,
\end{split}
\end{equation*}
where we have used (\ref{LM2}) and $\big| |p_{hor}^r| p_{hor}^r -
\frac{a_{hor}}{|a_{hor}|}\big| \leq 2 \big| p_{hor}^r -\frac{a_{hor}}{|a_{hor}|}\big|$.
It thus follows that:
\begin{equation*}
\big|p_{hor}^r - \frac{a_{hor}}{|a_{hor}|}\big|^2  \leq 
\frac{1}{2} r^2 \frac{a_3^2}{|a_{hor}|} + 4 \frac{r}{|a_{hor}|} \big(|A|+2|B|\big)\big|p_{hor}^r - \frac{a_{hor}}{|a_{hor}|}\big|,
\end{equation*}
resulting in: $\big|p_{hor}^r - \frac{a_{hor}}{|a_{hor}|}\big| \leq
Cr$ with $C>0$ depending only on $|A|, |B|, |a_3|, \frac{1}{|a_{hor}|}$.
In conclusion:
$$0\leq g_r(p^r) - g_r\big(\frac{a_{hor}}{|a_{hor}|}\big) \leq
\frac{1}{4} r^2 a_3^2 + Cr^2 \big(|A|+2|B|\big).$$

\medskip

{\bf 3.} Likewise, for the maximizer $\bar p^r$ of $g_r$ on $\bar B_1(0)$, we get:
$$0\geq g_r(\bar p^r) - g_r\big(-\frac{a_{hor}}{|a_{hor}|}\big) \geq
-\frac{1}{4} r^2 a_3^2 + Cr^2 \big(|A|+2|B|\big).$$
The two above inequalities imply:
\begin{equation*}
\begin{split}
\frac{1}{2}\Big(\inf_{p\in B_1(0)} + & \sup_{p\in B_1(0)}\Big) g_r(p) 
= \frac{1}{2}\Big( g_r(p^r) +  g_r(\bar p^r) \Big) = \frac{1}{2}\Big( g_r\big(\frac{a_{hor}}{|a_{hor}|}\big) +
g_r\big(-\frac{a_{hor}}{|a_{hor}|}\big)\Big) + O(r^2) \\ & =
\frac{r}{|a_{hor}|^2}\langle A+B : a_{hor}^{\otimes 2}\rangle + O(r^2) =
r\left(\frac{1}{2}+\frac{s_\p^2(a_\p^2-1)}{3\pi}\right) \Delta_{\heis, \infty}v(q) + O(r^2). 
\end{split}
\end{equation*}

Consequently, recalling the definition of $g_r$, we get:
\begin{equation}
\begin{split}
\frac{1}{2}\Big(\inf_{p\in B_r(q)} + \sup_{p\in B_r(q)}\Big) \bar
f_r(p) 
&= v(q) +
\frac{r^2s_\p^2}{3\pi}\Delta_{\mathbb{H}}v(q)+r^2\left(\frac{1}{2}+\frac{s_\p^2(a_\p^2-1)}{3\pi}\right)\Delta_{\heis,
  \infty}v(q)+ o(r^2)\\ 
&= v(q) + \frac{r^2s_\p^2}{3\pi}\left[\Delta_{\mathbb{H}}v(q)+\left(\frac{3\pi}{2s_\p^2}+a_\p^2-1\right)\Delta_{\heis,
    \infty}v(q)\right]+ o(r^2) 
\end{split}
\end{equation}
which directly yields (\ref{dpp3}) in virtue of (\ref{gru}) and \eqref{spAndap}.
\end{proof}

\bigskip

\begin{center}
{\bf PART II: The $\heis$-Laplacian $\Delta_{\mathbb{H}}$ and
random walks in $\heis$}
\end{center}

\section{Horizontal $\epsilon$-walk in the Heisenberg group}\label{sec5}

Let $\mathcal{D}\subset\heis$ be an open, bounded and connected set. In this section,
we develop a probability setting related to the expansion
(\ref{MV-1P0})$_2$. The key role is played by the $3$-dimensional
process $\{Q_n\}_{n=0}^\infty$, whose increments are $2$-dimensional,
with the third variable slaved to the first two via the Levy
area. We then apply the classical argument and argue that $\{Q_n\}$
accumulates a.s. on $\partial\mathcal{D}$, and that its expectation
yields a $\heis$-harmonic extension of a given boundary data $F$.

\medskip

Let $\Omega_1 = B_1^2(0)$ and define:
$$\Omega= (\Omega_1)^{\mathbb{N}} = \big\{\omega=\{w_i\}_{i=1}^\infty; ~ w_i=(a_i, b_i) \in
B_1^2(0) ~~\mbox{ for all } i\in\mathbb{N}\big\}.$$
The probability space $(\Omega, \mathcal{F}, \mathbb{P})$ 
is given as the countable product of $(\Omega_1, \mathcal{F}_1, \mathbb{P}_1)$, where: 
$$\mathbb{P}_1(D) = \frac{|D|}{|B_1^2(0)|} \quad \mbox{ for all }~ D\in \mathcal{F}_1$$
is the normalized Lebesgue measure on the $\sigma$-algebra $\mathcal{F}_1$ of Borel subsets
of $B_1^2(0)$. For any $n\in\mathbb{N}$, we also define the
probability space $(\Omega_n, \mathcal{F}_n, \mathbb{P}_n)$ as the
product of $n$ copies of $(\Omega_1, \mathcal{F}_1,
\mathbb{P}_1)$. We always identify the $\sigma$-algebras $\mathcal{F}_n$
with the corresponding sub-$\sigma$-algebras of $\mathcal{F}$,
consisting of sets of the form $F\times \prod_{i=n+1}^\infty\Omega_1$
for all $F\in\mathcal{F}_n$. Note that the sequence $\{\mathcal{F}_n\}_{n=0}^\infty$, where we set
$\mathcal{F}_0= \{\emptyset, \Omega\}$, is a filtration of $\mathcal{F}$. 

\smallskip

Given $q_0\in \mathcal{D}$ and a parameter $\epsilon\in (0,1)$,
we now recursively define the sequence of random
variables $\{Q_n^{\epsilon, q_0}:\Omega\to
\mathcal{D}\}_{n=0}^\infty$, that will converge as $n\to\infty$ to a
limiting random variable $Q^{\epsilon, q_0}$. We use $\epsilon\ll 1$ as ultimately we will consider the
behavior of $Q^{\epsilon, q_0}$ as $\epsilon\to 0$.
Also, for simplicity of notation, we often suppress the superscripts
$\epsilon$ and $q_0$ and write $Q_n$ or $Q$ instead of $Q_n^{\epsilon,
  q_0}$ or $Q^{\epsilon, q_0}$, if no ambiguity arises. Define:
\begin{equation}\label{processM}
\begin{split}
& Q_0\equiv q_0, \\ & Q_n(w_1, \ldots, w_n) = q_{n-1}*  (\epsilon\wedge
d(q_{n-1}, \partial\mathcal{D})) (a_n, b_n, 0)\\
& \qquad\qquad \qquad \quad ~~ = q_{n-1}+ (\epsilon\wedge
d(q_{n-1}, \partial\mathcal{D})) \big(a_n, b_n,
\frac{1}{2}(x_{n-1}b_n- y_{n-1}a_n)\big), \\
& \mbox{ where } ~~ q_{n-1}= (x_{n-1}, y_{n-1}, z_{n-1}) =
Q_{n-1}(w_1, \ldots, w_{n-1}) ~~~ \mbox{ and } ~~ ~w_n = (a_n, b_n).
\end{split}
\end{equation}
That is, the position $q_{n-1}$ is advanced uniformly within the
$2$-dimensional Kor\'aniy ellipse in $T_{q_{n-1}}$ determined by the
horizontal radius that is the minimum of $\epsilon$ and the distance
$d(q_{n-1}, \partial\mathcal{D})$ of the current position from the boundary of $\mathcal{D}$.

\begin{Lemma}\label{lem_convergenceQn}
The sequence $\{Q_n\}_{n=0}^\infty$ is a martingale
relative to the filtration $\{\mathcal{F}_n\}_{n=0}^\infty$ and it
converges, pointwise a.s., to some random variable $Q: \Omega\to \partial\mathcal{D}$.
\end{Lemma}
\begin{proof}
The martingale property follows directly from the definition (\ref{processM}):
\begin{equation*}
\begin{split}
\mathbb{E}&(Q_n \mid \mathcal{F}_{n-1}) (w_1,\ldots, w_{n-1}) =
\int_{\Omega_1}Q_n(w_1,\ldots, w_n)\dd\mathbb{P}_1(w_n) \\ & = 
q_{n-1} +  (\epsilon\wedge d(q_{n-1}, \partial\mathcal{D})) \fint_{B_1^2(0)}\big(a_n, b_n,
\frac{1}{2}(x_{n-1}b_n- y_{n-1}a_n)\big)\dd(a_n,b_n) = q_{n-1} \quad \mbox{a.s. in } \Omega_{n-1},
\end{split}
\end{equation*}
because the added linear term integrates to $0$ in all components on the
symmetric $B_1^2(0)$. 

Being a bounded martingale, the sequence
$\{Q_n\}_{n=0}^\infty$ converges to some random variable $Q:\Omega\to\bar{\mathcal{D}}$.
It remains to show that $\mathbb{P}$-a.s. we have:
$Q\in\partial\mathcal{D}$. To this end, observe that:
\begin{equation}\label{proof_inboundary}
\begin{split}
& \{\lim_{n\to\infty} Q_n=Q\}\cap\{Q\in\mathcal{D}\} \subset
 \bigcup_{n\in\mathbb{N}, ~ \delta\in (0,\epsilon)\cap\mathbb{Q}} A(n, \delta),\\
& \mbox{where } ~~ A(n, \delta) = \{d(Q_i, \partial\mathcal{D})\geq \delta ~~\mbox{ and }
~~|Q_{i+1}-Q_i|\leq \frac{\delta}{2} \quad \mbox{for all } i\geq n\}.
\end{split}
\end{equation}
Also, if $\omega=\{w_i\}_{i=1}^\infty\in A(n, \delta)$ then for all $i\geq n$ we have:
$$\frac{\delta}{2}\geq |Q_{i+1}(\omega) - Q_i(\omega)| \geq
(\epsilon\wedge d(q_{i}, \partial\mathcal{D}))|w_{i+1}|\geq
(\epsilon\wedge \delta)|w_{i+1}| = \delta |w_{i+1}|,$$
which implies:
$A(n, \delta)\subset \{\omega=\{w_i\}_{i=1}^\infty\in\Omega;~~ |w_i|\leq \frac{1}{2}
~ \mbox{ for all } i\geq n\}. $
We conclude that:
$$\mathbb{P}(A(n, \delta))\leq \lim_{i\to\infty}
\mathbb{P}_1(B_{1/2}^2(0))^{i-n} = 0 \qquad \mbox{for all }
n\in\mathbb{N} ~\mbox{ and all } \delta\in (0, \epsilon)$$
so that the event in the left hand side of (\ref{proof_inboundary}) has probability $0$ as well.
\end{proof}

\medskip

Given a continuous function $F:\partial\mathcal{D}\to\mathbb{R}$, we define now:
\begin{equation}\label{ue_def}
u^\epsilon(q) = \mathbb{E}[F\circ Q^{\epsilon, q}] = \lim_{n\to\infty}
\mathbb{E}[F\circ Q_n^{\epsilon,q}]  \quad\mbox{ for all } q\in\mathcal{D},
\end{equation}
where in the last limiting expression above we have identified $F$
with some continuous extension of itself on $\bar{\mathcal{D}}$. Since for
every $n\in\mathbb{N}$ the random variable $F\circ Q_n^{\epsilon, q}$
is jointly Borel in the variables $q$ and $\omega$, it follows that
$u^\epsilon:\mathcal{D}\to\mathbb{R}$ is bounded (by
$\|F\|_{\mathcal{C}(\partial\mathcal{D})}$) and Borel. It is also
clear that this construction is monotone in $F$, in the sense that
if $F_1\leq F_2$ on $\partial\mathcal{D}$ then $u^\epsilon_{F_1}\leq
u^\epsilon_{F_2}$, with obvious notation.

\medskip

\begin{Lemma}\label{Qmeanval}
The function $u^\epsilon$ satisfies:
\begin{equation}\label{MV2_var}
u^\epsilon(q) = {\mathcal{A}}_2(u^\epsilon,\epsilon\wedge
d(q, \partial\mathcal{D}))(q) \quad \mbox{ for all } q\in\mathcal{D}.
\end{equation}
Moreover, the sequence $\{u^\epsilon\circ Q_n\}_{n=0}^\infty$ is a
martingale relative to the filtration $\{\mathcal{F}_n\}_{n=0}^\infty$.
\end{Lemma}
\begin{proof} An 
application of Fubini's theorem, in view of the definition in (\ref{processM}), gives directly:
$$\mathbb{E}[F\circ Q_n] = \int_{\Omega_1} \mathbb{E}[F\circ
Q_{n-1}^{\epsilon, Q_1(w_1)}]\dd\mathbb{P}(w_1),$$
which implies (\ref{MV2_var}) by passing to the limit with
$n\to\infty$ and recalling the definitions (\ref{ue_def}) and (\ref{processM}).
To show the martingale property, we similarly check that for every $n\in\mathbb{N}$:
\begin{equation*}
\begin{split}
\mathbb{E}(u^\epsilon\circ Q_n\mid \mathcal{F}_{n-1}) & = \int_{\Omega_1}
u^\epsilon\circ Q_n \dd \mathbb{P}_1(w_n) \\ & = \int_{\Omega_1}
u^\epsilon(q_{n-1}*(\epsilon\wedge
d(q_{n-1}, \partial\mathcal{D}))(w_n, 0) \dd \mathbb{P}_1(w_n) 
= u^\epsilon(Q_{n-1})
\end{split}
\end{equation*}
is valid $\mathbb{P}_{n-1}$-a.s. in $\Omega_{n-1}$.
\end{proof}

\begin{Cor}\label{coruni}
Assume that $u\in \mathcal{C}^2(\mathcal{D}) \cap\mathcal{C}(\bar{\mathcal{D}})$ satisfies:
\begin{equation}\label{Heisenberg_harm}
\Delta_{\mathbb{H}}u=0 \quad\mbox{ in } \mathcal{D}, \qquad \qquad u=F
\quad \mbox{ on } \partial\mathcal{D}.
\end{equation}
Then $u^\epsilon=u$ for all $\epsilon\in (0,1)$. In particular,
(\ref{Heisenberg_harm}) has at most one solution.
\end{Cor}
\begin{proof}
We first observe that the sequence $\{u\circ Q_n\}_{n=0}^\infty$ is a
martingale relative to the filtration
$\{\mathcal{F}_n\}_{n=0}^\infty$. This property follows by the same
calculation as in the proof of Lemma \ref{Qmeanval}, where
$u^\epsilon$ is now replaced by $u$ and where (\ref{MV-1P}) is used
for $u$ instead of the averaging formula
(\ref{MV2_var}). Consequently, Doob's theorem yields:
$$u= \mathbb{E}[u\circ Q_0] = \mathbb{E}[u\circ Q_n] \qquad \mbox{for
  all } n\in\mathbb{N}.$$
The right hand side above converges to $u^\epsilon$, as $n\to\infty$,
which proves the first claim.

For the second claim, recall that the functions $u^\epsilon$ depend only on the
boundary values $F = u_{\mid \partial\mathcal{D}} $ and not on their
particular extension $u$ on $\bar{\mathcal{D}}$. This yields
uniqueness of the harmonic extension in (\ref{Heisenberg_harm}).
\end{proof}

\section{Convergence of $u^\epsilon$ as $\epsilon\to 0$}\label{secconv_2}

In this section we are concerned with the limiting properties of the
family $\{u^\epsilon\}_{\epsilon\to 0}$. Following \cite{Dbook}, we will
introduce the process-related definition of regularity of the boundary
points $q\in\partial\mathcal{D}$, which is the notion essentially
equivalent to that of convergence to the $\heis$-harmonic function
with prescribed boundary data.
The first observation is on transferring the estimate at the boundary
of $\mathcal{D}$ to its interior, by walk-coupling. 
This probabilistic
technique utilizes translation invariance of solutions, similar to its
analytic counterpart presented in section
\ref{sec_conv} in connection with the Tug of War with noise, modelled
on (\ref{dpp2}).

\begin{Lemma}\label{A0}
In the context of definition (\ref{ue_def}), assume that for every
$\eta>0$ there exist $\delta, \epsilon_0>0$ such 
that for all $\epsilon\in (0,\epsilon_0)$ there holds:
\begin{equation}\label{bd_as}
|u^\epsilon(q')-u^\epsilon(q)|\leq \eta \quad \mbox{ for all }\;
q,q'\in\mathcal{D} \; \mbox{ satisfying: }\;
d(q, \partial\mathcal{D})<\delta, \;\; |q-q'|\leq \delta.
\end{equation}
Then, the same uniformity property is likewise valid away from $\partial\mathcal{D}$.
Namely, for  every $\eta>0$ there exist $\hat\delta, \hat\epsilon>0$ such 
that for all $\epsilon\in (0,\hat\epsilon)$ and all
$q,q'\in\mathcal{D}$ satisfying $|q-q'|\leq \delta$, there holds:
$|u^\epsilon(q')-u^\epsilon(q)|\leq \eta$.
\end{Lemma}
\begin{proof}
Fix $\eta>0$ and let $\delta> \epsilon_0>0$ be as in
(\ref{bd_as}). Given $\epsilon\in (0, \epsilon_0)$ and
$q=(x,y,z),q'=(x',y',z')\in\mathcal{D}$ such that: 
$$d(q,\partial\mathcal{D}), ~ d(q',\partial\mathcal{D})\geq \delta,
\qquad |q'-q|\leq\frac{\delta}{1+\frac{1}{2}\mbox{diam}\mathcal{D}},$$ 
define the stopping time $\tau:\Omega\to\mathbb{N}\cup\{\infty\}$ by:
$$\tau(\omega) = \min\big\{n\geq 1;~ d(Q_n^{\epsilon, q}(\omega),\partial\mathcal{D})<\delta 
\;\mbox{ or }\; d(Q_n^{\epsilon, q'}(\omega),\partial\mathcal{D})<\delta \big\}.$$
Indeed, $\tau$ is finite a.s. in $\Omega$ in virtue of Lemma \ref{lem_convergenceQn}.
For a given $\omega\in\Omega$ with $\tau(\omega)<\infty$, assume
without loss of generality that
$d(Q^{\epsilon,q}(\omega),\partial\mathcal{D})<\delta$. It is not hard
to show (by induction on $i=1\ldots \tau$) that:
\begin{equation*}
Q_\tau^{\epsilon, q'}(\omega) - Q_\tau^{\epsilon, q}(\omega) = q'-q +
\epsilon\big(0,0,\frac{1}{2}\big\langle (x'-x, y'-y),\sum_{i=1}^\tau
w_i^\perp \big\rangle\big)\quad \mbox{and}\quad \epsilon \sum_{i=1}^\tau w_i =(x_\tau,y_\tau)-(x,y),
\end{equation*}
where we write $w_i^\perp=(a_i, b_i)^\perp=(b_i, -a_i)$, and as usual: $Q_\tau^{\epsilon, q} =(x_\tau, y_\tau,
z_\tau)$. Consequently:
$$\big|Q_\tau^{\epsilon, q'}(\omega) - Q_\tau^{\epsilon, q}(\omega)
\big|\leq |q'-q| + \frac{1}{2}|q'-q|\cdot \big|\epsilon \sum_{i=1}^\tau
w_i \big|\leq \big(1+\frac{1}{2}\mbox{diam}\mathcal{D}\big) |q'-q|\leq \delta.$$
By Lemma \ref{Qmeanval}, the sequence $\{u^\epsilon\circ
Q_n^{\epsilon,q'} - u^\epsilon\circ Q_n^{\epsilon,q}\}_{n=0}^\infty$
is a bounded martingale, so Doob's theorem yields: 
\begin{equation*}
\begin{split}
|u^\epsilon(q') - u^\epsilon(q)|= \big|\mathbb{E}\big[u^\epsilon\circ
Q_\tau^{\epsilon,q'} - u^\epsilon\circ Q_\tau^{\epsilon,q}\big]\big|\leq
\int_\Omega\big| u^\epsilon\circ
Q_\tau^{\epsilon,q'} - u^\epsilon\circ Q_\tau^{\epsilon,q}\big|\dd \mathbb{P}\leq\eta,
\end{split}
\end{equation*}
by (\ref{bd_as}). This concludes the proof, with
$\hat\epsilon=\epsilon_0$ and $\hat\delta=\frac{\delta}{1+\frac{1}{2}\mbox{diam}\mathcal{D}}$.
\end{proof}

\begin{Def}\label{d_wr}
Consider the $\epsilon$-walk in (\ref{processM}). 
\begin{itemize}
\item[(a)] We say that a boundary point $q_0\in\partial\mathcal{D}$ is
  {\em walk-regular} if for every $\eta, \delta>0$ there exists
  $\hat\delta\in(0,\delta)$ and $\hat\epsilon\in (0,1)$ such that:
$$\mathbb{P}\big(Q^{\epsilon, q}\in B_\delta(q_0)\big)\geq 1-\eta\quad
\mbox{for all } \; \epsilon\in (0,\hat\epsilon) \;\;\mbox{ and } \;\; 
q\in B_{\hat\delta}(q_0)\cap\mathcal{D}.$$
\item[(b)] We say that $\mathcal{D}$ is walk-regular if every
  $q_0\in\partial\mathcal{D}$ is walk-regular. 
\end{itemize}
\end{Def}

Observe that when $\mathcal{D}$ is walk-regular, then (by
compactness), $\hat\delta$ and $\hat\epsilon$ can be chosen
independently of $q_0$ (i.e. they depend only on the prescribed
thresholds $\eta,\delta$).

\medskip

The walk-regularity is essentially equivalent to the validity of
(\ref{bd_as}) with $q\in\partial\mathcal{D}$.

\begin{Teo}\label{A}
\begin{itemize}
\item[(a)] Assume that $q_0\in\partial\mathcal{D}$ is walk-regular. Then for
every $\eta>0$ there exists $\hat\epsilon, \hat\delta>0$ such that for
every $\epsilon\in (0,\hat\epsilon)$ there holds:
$$|u^\epsilon(q) - F(q_0)|\leq\eta \quad \mbox{for all }\; q\in
B_{\hat\delta}(q_0)\cap\mathcal{D}.$$
\item[(b)] If $q_0\in\partial\mathcal{D}$ is not walk-regular, then there exists
a continuous function $F:\partial\mathcal{D}\to\mathbb{R}$ such that:
$$\lim_{\hat\epsilon,\hat\delta\to 0}\sup \big\{|u^\epsilon(q) -
F(q_0)|; ~ \epsilon\in (0,\hat\epsilon), ~ q\in B_{\hat\delta}(q_0)\cap\mathcal{D}\big\}> 0.$$
\end{itemize}
\end{Teo}
\begin{proof}
To show (a), let $\eta>0$, and choose $\delta>0$  so that $|F(q_0') -
F(q_0)|\leq\frac{\eta}{2}$ for all $q_0'\in\partial\mathcal{D}$ with
$d(q_0', q_0)<\delta$. Further, choose $\hat\delta, \hat\epsilon$ in
Definition \ref{d_wr} (a) corresponding to
$\frac{\eta}{4\|F\|_{\mathcal{C}(\partial\mathcal{D})}+1}$ and
$\delta$. Then, for all $q\in B_{\hat\delta}(q_0)\cap\mathcal{D}$ and
$\epsilon\in (0,\hat\epsilon)$ we have:
$$ |u^\epsilon(q) - F(q_0)|\leq\int_\Omega \big|F\circ
Q^{\epsilon,q}-F(q_0)\big|\dd\mathbb{P}\leq
\mathbb{P}\big(Q^{\epsilon, q}\not\in B_\delta(q_0)\big)\cdot
2\|F\|_{\mathcal{C}(\partial\mathcal{D})}+\frac{\eta}{2} \leq \eta.$$

\medskip

For (b), define $F(q)=d(q, q_0)$ for all $q\in\partial\mathcal{D}$. By
assumption, there exists $\eta,\delta>0$ such that for some sequences
$\epsilon_i\to 0^+$ and $\mathcal{D}\ni q_i\to q_0$ we have:
$$\mathbb{P}\big(Q^{\epsilon_i, q_i}\not\in
B_\delta(q_0)\big)>\eta\quad\mbox{as }\; i\to\infty.$$
By nonnegativity of $F$, it follows that:
$$u^{\epsilon_i}(q_i)-F(q_0) = u^{\epsilon_i}(q_i) = \int_\Omega
F\circ Q^{\epsilon_i, q_i}\dd\mathbb{P} \geq \int_{\{Q^{\epsilon_i,
    q_i}\not\in B_\delta(q_0)\}}\delta \dd\mathbb{P}>\eta\delta,$$
proving the claim.
\end{proof}

\begin{Teo}\label{lem_lim}
Assume that $\mathcal{D}$ is walk-regular. Then every sequence in the family
$\{u^\epsilon\}_{\epsilon\to 0}$ has a further subsequence that
converges uniformly to a continuous function
$u\in\mathcal{C}(\bar{\mathcal{D}})$ such that $u=F$ on $\partial\mathcal{D}$.
\end{Teo}
\begin{proof}
By Lemma \ref{A} and the uniformity observation following Definition
\ref{d_wr}, we obtain that for every $\eta>0$ there exist
$\hat\epsilon, \hat\delta>0$ such that for all $\epsilon\in (0,\hat\epsilon)$:
\begin{equation}\label{gz}
|u^\epsilon(q) - F(q_0)|\leq \eta \quad \mbox{ for all }\;
q_0\in\partial\mathcal{D}, \quad q\in B_{\hat\delta}(q_0)\cap\mathcal{D}.
\end{equation}
Consequently, by uniform continuity of $F$ on the compact metric space $(\partial\mathcal{D},d)$
it follows that the assumption (\ref{bd_as}) of Lemma \ref{A0} is
valid. Thus the equibounded family $\{u^\epsilon\}_{\epsilon\to 0}$ is
also equi-oscillatory, so the Ascoli-Arzel\`{a} theorem implies existence
of a sequence that converges uniformly to some
$u\in\mathcal{C}(\mathcal{D})$. By (\ref{gz}) we finally conclude that
$u$ is continuous up the boundary where $u_{\mid\partial\mathcal{D}}=F$. 
\end{proof}

We remark that the limit $u$ above will be identified as the
$\heis$-harmonic function in section \ref{sec_ide}.

\section{The exterior $\mathbb{H}$-corkscrew condition implies
  walk-regularity}\label{seccork_2}

\begin{Def}\label{cork_def}
We say that a given boundary point $q_0\in \partial\mathcal{D}$
satisfies the {\em exterior $\mathbb{H}$-corkscrew condition} provided
that there exists $\mu\in (0,1)$ such that for all sufficiently small $r>0$ there exists a
Kor\'{a}niy ball $B_{\mu r}(p_0)$ such that:
$$B_{\mu r}(p_0)\subset B_r(q_0)\setminus \mathcal{D}.$$
\end{Def}

One can show (see \cite[Theorem 1.3]{FSP}) that every bounded domain
$\mathcal{D}$ of Euclidean regularity $\mathcal{C}^{1,1}$ satisfies
Definition \ref{cork_def} at each boundary point $q_0$. In fact, all $\mathcal{C}^{1,1}$
domains in Carnot groups of step $2$, are NTA (non-tangentially
accessible), which means that they satisfy both the exterior and
interior $\mathbb{H}$-corkscrew condition, plus a Harnack chain
condition. This regularity is optimal, in the sense that
$\mathcal{C}^{1,\alpha}$ domains, for $\alpha<1$, do not in general
satisfy even a one-sided $\mathbb{H}$-corkscrew condition.

\begin{example}
For $\alpha\in [0,1)$, define $\mathcal{D}=\big\{q=(x,y,z)=(q_{hor},z) \in\heis;~
|q_{hor}|^{1+\alpha}>z\big\}$. Then the domain $\mathcal{D}$ is
$\mathcal{C}^{1,\alpha}$-regular, but the exterior
$\mathbb{H}$-corkscrew condition does not hold at $q_0=0$. Indeed,
take any $q\not\in\bar{\mathcal{D}}$ and compute:
\begin{equation*}
\mbox{dist}(q, \partial\mathcal{D}) \leq d\Big(q,
\big(z^{\frac{1}{1+\alpha}}\frac{q_{hor}}{|q_{tan|}}, z\big)\Big) =
\big|z^{\frac{1}{1+\alpha}}-|q_{hor}|\big| < z^{\frac{1}{1+\alpha}}.
\end{equation*}
Thus, if $q\in B_r(0)$, we obtain: 
$$\mbox{dist}(q, \partial\mathcal{D}) < C r^{\frac{2}{1+\alpha}},$$
with a universal constant $C$ depending only on $\alpha$. This
contradicts $\mbox{dist}(q, \partial\mathcal{D}) >\mu r$ for all
$\mu>0$ as $r\to 0$. \endproof
\end{example}

We also remark that (similarly to the Euclidean case) all bounded
intrinsic Lipschitz domains are NTA, and hence satisfy the
$\mathbb{H}$-corkscrew condition (see \cite[Theorem 3]{FSP}). 
The intrinsic Lipschitz domains, studied in \cite{FSP}, are domains whose boundaries are locally
graphs of intrinsic Lipschitz functions acting between appropriate homogeneous
subgroups of a Carnot group.

\medskip

The main statement of this section is the following:

\begin{Teo}\label{cork_good_2}
If $q_0\in\partial\mathcal{D}$ satisfies
the { exterior $\mathbb{H}$-corkscrew condition}, then $q_0$ is walk-regular.
\end{Teo}

\medskip

Towards the proof we necessitate an inductive technique, see \cite{PS}:

\begin{Lemma}\label{rings_inductive}
For a given $q_0\in \partial\mathcal{D}$, assume that there exists a
constant $\theta_0\in (0,1)$ with the property that for all $\delta>0$
there exists $\hat\delta\in (0,\delta)$ and $\hat\epsilon\in (0,1)$
such that:
\begin{equation}\label{ind0}
\mathbb{P}\big(\exists n\geq 0~ Q_n^{\epsilon, q}\not\in
B_\delta(q_0)\big)\leq \theta_0 \qquad \mbox{for all }\; \epsilon\in
(0,\hat\epsilon), ~~ q\in B_{\hat\delta}(q_0)\cap\mathcal{D}.
\end{equation}
Then $q_0$ is walk-regular.
\end{Lemma}
\begin{proof}
Fix $\eta, \delta>0$ and let $m\in\mathbb{N}$ be such that
$\theta_0^m\leq \eta$. Define the tuples $\{\hat\epsilon_k\}_{k=0}^m$, 
$\{\hat \delta_k\}_{k=0}^{m-1}$, $\{\delta_k\}_{k=0}^m$ inductively, by setting:
\begin{equation}\label{ind1}
\begin{split}
& \delta_m=\frac{\delta}{2}, \quad \epsilon_m=1,\\
& \hat \delta_{k-1}\in (0,\delta_k),\quad \epsilon_{k-1}\in
(0,\epsilon_k) \quad \mbox{for all }\; k=1\ldots m\\
& \quad \mbox{ so that: }   ~~\mathbb{P}\big(\exists n\geq 0~ Q_n^{\epsilon, q}\not\in
B_{\delta_k}(q_0)\big)\leq \theta_0 \quad \mbox{for all }\; \epsilon\in
(0,\epsilon_{k-1}), ~~ q\in
B_{\hat\delta_{k-1}}(q_0)\cap\mathcal{D},\\
& \delta_{k-1}\in (0,\hat\delta_{k-1}) \quad \mbox{ for all }\; k=2\ldots m.
\end{split}
\end{equation}
We finally set:
$$\hat\epsilon = \min\big\{\epsilon_0,
\{|\delta_k-\hat\delta_k|\}_{k=1}^{m-1}\big\}\quad \mbox{and} \quad
\hat\delta = \hat\delta_0.$$
Fix $q\in B_{\hat\delta}(q_0)\cap\mathcal{D}$ and $\epsilon\in
(0,\hat\epsilon)$. Then the application of Fubini's theorem yields:
\begin{equation*}\label{ind2}
\mathbb{P}\big(\exists n\geq 0~ Q_n^{\epsilon, q}\not\in
B_{\delta_k}(q_0)\big)\leq \theta_0 \mathbb{P}\big(\exists n\geq 0~ Q_n^{\epsilon, q}\not\in
B_{\delta_{k-1}}(q_0)\big) \quad \mbox{for all }\; k=2\ldots m.
\end{equation*}
Together with the inequality in (\ref{ind1}) for $k=1$, the above
bound results in:
\begin{equation*}
\begin{split}
\mathbb{P}\big(\exists n\geq 0~ ~Q_n^{\epsilon, q}\not\in
B_{\delta}(q_0)\big)&\leq \mathbb{P}\big(\exists n\geq 0~~ Q_n^{\epsilon, q}\not\in
B_{\delta_m}(q_0)\big)\\ & \leq\theta_0^{m-1} \mathbb{P}\big(\exists n\geq 0~~ Q_n^{\epsilon, q}\not\in
B_{\delta_1}(q_0)\big)\leq \theta_0^m\leq \eta,
\end{split}
\end{equation*}
proving the validity of condition (a) in Definition \ref{d_wr}.
\end{proof}

\medskip

\noindent {\bf Proof of Theorem \ref{cork_good_2}.}
Fix $\delta>0$ sufficiently small and set $\hat\epsilon=1$,
$\hat\delta=\delta/4$. By assumption, there exists a subball
$B_{\mu\hat\delta}(p_0)\subset B_{\hat\delta}(q_0)\setminus
\mathcal{D}$. We will show that condition (\ref{ind0}) holds, with
constant $\theta_0 = \frac{1-\mu^2/4}{1-\mu^2/9}\in (0,1)$, as
identified below.

Fix now $q\in B_{\hat\delta}(q_0)\cap\mathcal{D}$ and $\epsilon\in (0,\hat\epsilon)$.
Since the function $\phi(p)=v(d(p,p_0))$ with
$v(t)=\frac{1}{t^2}$ satisfies: $\Delta_{\mathbb{H}}\phi=0$ in
$\mathbb{H}\setminus\{p_0\}$, the sequence of random variables $\{v\circ
d(Q_n^{\epsilon, q},p_0)\}_{n=0}^\infty$ is a martingale relative to
the filtration $\{\mathcal{F}_n\}_{n=0}^\infty$. Define
$\tau:\Omega\to\mathbb{N}\cup\{\infty\}$ by:
$$\tau = \inf\big\{n\geq 0;~ Q_n^{\epsilon, q}\not\in B_\delta (q_0)\big\}.$$
For every $n\geq 0$, the random variable $\tau\wedge n$ is a bounded
stopping time, so:
$$v\big(d(q,p_0)\big)=\mathbb{E} \big[v\circ
d(Q_{n\wedge\tau}^{\epsilon, q},p_0)\big]\qquad\mbox{for all}\; n\geq 0$$
follows by Doob's theorem. Passing to the limit with $n\to\infty$ we obtain:
$$v(d(q,p_0))=\int_{\{\tau<\infty\}}v\big(d(Q_{\tau}^{\epsilon, q},p_0)\big)\dd\mathbb{P}
+ \int_{\{\tau=\infty\}}v\big(d(Q^{\epsilon, q},p_0)\big)\dd\mathbb{P}.$$
Since $d(Q_{\tau}^{\epsilon, q},p_0)\geq d(Q_{\tau}^{\epsilon, q},q_0)
- d(q_0, p_0)\geq \delta-\hat\delta = 3\hat\delta$ and
$d(Q^{\epsilon, q},p_0)\geq \mu\hat\delta$, together with $d(q, p_0)\leq
d(q,q_0)+d(q_0, p_0)\leq 2\hat\delta$, the last displayed formula becomes:
$$v(2\hat\delta)\leq \mathbb{P}(\tau<\infty) v(3\hat\delta) + \mathbb{P}(\tau=\infty) v(\mu\hat\delta) =
\mathbb{P}(\tau<\infty) \big(v(3\hat\delta) - v(\mu\hat\delta)\big) +v(\mu\hat\delta).$$
Equivalently:
$$\mathbb{P}(\tau<\infty) \leq\frac{v(\mu\hat\delta) - v(2\hat\delta)}{v(\mu\hat\delta) - v(3\hat\delta)}
=\frac{1/\mu^2 - 1/4}{1/\mu^2-1/9},$$
which ends the proof.
\endproof

\section{Identification of the limit $u$: a viscosity solutions proof}\label{sec_ide}

In this section we show that the uniform limit of the whole family
$\{u^\epsilon\}_{\epsilon\to 0}$, defined in (\ref{ue_def}), coincides with the unique
$\heis$-harmonic extension to the given continuous data $F$, provided that $\mathcal{D}$ is walk-regular.
We present a viscosity solutions proof of this statement, expandable to the
case of arbitrary exponent $\p\in (1,\infty)$. Indeed, in section \ref{uniq_p_sec} we will carry
out in detail the parallel construction for $\p>2$ in connection to the mean value
property (\ref{dpp2}). The construction for $1<\p<\infty$,
feasible in the framework of (\ref{dpp3}), is conceptually identical
and left as an exercise for an interested reader; the details
pertaining to the Euclidean case can be found in \cite{Lew}. We
point out that another proof of Theorem \ref{limit_good} is available
in connection with the discrete Levy area process.

\smallskip

We start with a simple general lemma about the minima of uniform approximations:

\begin{Lemma}\label{appromin}
Assume that a sequence of  bounded functions
$\{u_n:\bar{\mathcal{D}}\to\mathbb{R}\}_{n=1}^\infty$ converges 
uniformly to some $u\in\mathcal{C}(\bar{\mathcal{D}})$, as $n\to\infty$.
Then, for every sequence of positive numbers $\{\delta_n\}_{n=1}^\infty$
converging to $0$, every $q_0\in\mathcal{D}$ and every
$\phi\in\mathcal{C}^2(\bar{\mathcal{D}})$ such that: 
\begin{equation}\label{assu_super1}
\phi(q_0) = u(q_0), \quad \phi<u ~~ \mbox{in } \bar{\mathcal{D}}\setminus \{q_0\},
\end{equation}
there exists a sequence $\{q_n\in\mathcal{D}\}_{n=1}^\infty$, satisfying:
$$ \lim_{n\to\infty} q_n = q_0 \quad \mbox{ and } \quad
u_n(q_n) - \phi(q_n) \leq \inf_{\bar{\mathcal{D}}} \,(u_n - \phi) + \delta_n.$$
\end{Lemma}
\begin{proof}
For every large integer $j$ define $\eta_j>0$ and $n_j\in\mathbb{N}$ such that:
$$\eta_j = \min_{\bar{\mathcal{D}}\setminus B^3_{1/j}(q_0)} (u-\phi) \qquad\mbox{and}
\qquad \sup_{\bar{\mathcal{D}}}|u_n - u| \leq \frac{1}{2}\eta_j \quad
\mbox{ for all } n\geq n_j.$$
Without loss of generality, the sequence $\{n_j\}_{j=1}^\infty$ is
strictly increasing. Now, for all $n\in [n_j, n_{j+1})$ let $q_n\in B^3_{1/j}(q_0)$ satisfy: 
\begin{equation}\label{4.71} 
u_n(q_n) - \phi(q_n) \leq \inf_{B_{1/j}^3(q_0)}(u_n - \phi) +\delta_n.
\end{equation}
Observe that the following bound is valid for every
$q\in\bar{\mathcal{D}}\setminus B_{1/j}^3(q_0)$:
 \begin{equation*}
\begin{split}
u_n(q) - \phi(q) & \geq u(q) - \phi(q) - \sup_{\bar{\mathcal{D}}}|u_n
- u|  \geq \eta_j - \frac{1}{2}\eta_j \geq
\sup_{\bar{\mathcal{D}}}|u_n - u| \\ & \geq u_n(q_0)- \phi(q_0)\geq
\inf_{B^3_{1/j}(q_0)}(u_n - \phi).
\end{split}
\end{equation*}
This proves the claim in view of (\ref{4.71}).
\end{proof}

\begin{Teo}\label{limit_good}
The limit function $u$ in Theorem \ref{lem_lim} solves
the Dirichlet problem (\ref{Heisenberg_harm}). Automatically, the whole family
$\{u^\epsilon\}_{\epsilon\to 0}$ converges uniformly to such unique
solution $u$, provided that $\mathcal{D}$ is walk-regular.
\end{Teo}
\begin{proof}
Let $\epsilon_i\to 0^+$ be such that $\{u^{\epsilon_i}\}_{i=1}^\infty$
converges uniformly to $u$ on $\bar{\mathcal{D}}$. Fix
$q_0\in\mathcal{D}$ and take $\phi\in \mathcal{C}^2(\bar{\mathcal{D}})$ satisfying
(\ref{assu_super1}). Choose a sequence
$\{q_i\in\mathcal{D}\}_{i=1}^\infty$ with the properties guaranteed by
Lemma \ref{appromin} when applied to the error sequence $\delta_i = \epsilon_i^3$. 
Recalling (\ref{MV2_var}) we obtain:
$$\mathcal{A}_2(\phi, \epsilon_i)(q_i) \leq \mathcal{A}_2(u^{\epsilon_i}, \epsilon_i)(q_i) -
\big(u^{\epsilon_i}(q_i) - \phi (q_i)\big) +\epsilon_i^3= \phi (q_i) +\epsilon_i^3.$$
From (\ref{MV-1P0})$_2$ we thus conclude: $\phi(q_i)
+\frac{\epsilon_i^2}{8}\Delta_\heis\phi(q_i) \leq
\phi(q_i)+o(\epsilon_i^2)$, which upon passing to the limit
$i\to\infty$ implies: $\Delta_\heis\phi(q_0)\leq 0$.

By a symmetric reasoning, we get that if $\phi\in
\mathcal{C}^2(\bar{\mathcal{D}})$ satisfies: $\phi(q_0)=u(q_0)$ and
$\phi>u$ in $\bar{\mathcal{D}}\setminus\{q_0\}$, then $\Delta_\heis\phi(q_0)\geq 0$.
Finally, the same arguments as in the second part of the proof of
Proposition \ref{MV3expansion} imply that $u$ coincides with its own
$\heis$-harmonic extension in any $\bar B_r(q_0)\subset\mathcal{D}$.
Therefore, $u$ is the $\heis$-harmonic extension of the continuous boundary data
$F=u_{\mid\partial\mathcal{D}}$ and as such it is unique, completing the proof. 
\end{proof}

\bigskip

\begin{center}
{\bf PART III: The $\p$-$\heis$-Laplacian $\Delta_{\mathbb{H},\p}$ and
the random Tug of Wars in $\heis$}
\end{center}

\section{The random Tug of War game in the Heisenberg group}\label{sec_setup_p}

In this section we develop the probability setting similar to that of
section \ref{sec5}, but related to the expansion (\ref{dpp2}) rather
than (\ref{MV-1P0})$_2$. We remark that an identical construction can
be carried out for the dynamic programming principle modelled on
(\ref{dpp3}), where the advantage is that it covers any exponent
$1<\p<\infty$, see Remark \ref{rem_new}. We leave the details to the interested reader; in the
Euclidean setting we point to the paper \cite{Lew}.
Here, we always assume that $\p>2$, whereas
parallel statements for $\p=2$ follow by approximation $\p\to 2^+$.

\bigskip

{\bf 1.} Let $\Omega_1 = B_1(0)\times \{1,2,3\}\times (0,1)$ and define:
\begin{equation*}
\begin{split}
\Omega= (\Omega_1)^{\mathbb{N}} = \big\{&\omega=\{(w_i, s_i,
t_i)\}_{i=1}^\infty; \\ & ~ w_i=(a_i, b_i, c_i) \in
B_1(0),~ s_i\in \{1,2,3\},~ t_i\in (0,1) ~~\mbox{ for all }
i\in\mathbb{N}\big\}.
\end{split}
\end{equation*}
The probability space $(\Omega, \mathcal{F}, \mathbb{P})$ 
is given as the countable product of $(\Omega_1, \mathcal{F}_1,
\mathbb{P}_1)$. Here, $\mathcal{F}_1$ is the smallest $\sigma$-algebra
containing all products $D\times S\times B$ where $D\subset
B_1(0)\subset\mathbb{H}$ and $B\subset (0,1)$ are Borel, and
$S\subset\{1,2,3\}$. The probability measure $\mathbb{P}_1$ is given
as the product of: the normalized Lebesgue measure on $B_1(0)$, the
uniform counting measure on $\{1,2,3\}$ and the Lebesgue measure on $(0,1)$:
$$\mathbb{P}_1(D\times S\times B) = \frac{|D|}{|B_1(0)|}\cdot \frac{|S|}{3}\cdot |B|.$$
For each $n\in\mathbb{N}$, we consider the
probability space $(\Omega_n, \mathcal{F}_n, \mathbb{P}_n)$ that is the
product of $n$ copies of $(\Omega_1, \mathcal{F}_1,
\mathbb{P}_1)$. The $\sigma$-algebras $\mathcal{F}_n$
is always identified with the corresponding sub-$\sigma$-algebra of $\mathcal{F}$,
consisting of sets of the form $A\times \prod_{i=n+1}^\infty\Omega_1$
for all $A\in\mathcal{F}_n$. The sequence $\{\mathcal{F}_n\}_{n=0}^\infty$, where we set
$\mathcal{F}_0= \{\emptyset, \Omega\}$, is a filtration of
$\mathcal{F}$. 

\medskip

{\bf 2.} Given are two family of functions $\sigma_I=\{\sigma_I^n\}_{n=0}^\infty$ and
$\sigma_{II}=\{\sigma_{II}^n\}_{n=0}^\infty$, defined on the
corresponding spaces of ``finite histories'' $H_n=\mathbb{H}\times (\mathbb{H}\times\Omega_1)^n$:
$$\sigma_I^n, \sigma_{II}^n:H_n\to B_1(0)\subset\mathbb{H},$$
assumed to be measurable with respect to the (target) Borel
$\sigma$-algebra in $B_1(0)$ and the (domain) product $\sigma$-algebra on $H_n$.
For every $q_0\in\heis$ and $\epsilon\in (0,1)$ we now recursively
define the sequence of random variables:
$$\big\{Q_n^{\epsilon, q_0, \sigma_I, \sigma_{II}}:\Omega\to\heis\big\}_{n=0}^\infty.$$
For simplicity of notation, we often suppress some of the superscripts $\epsilon, q_0, \sigma_I, \sigma_{II}$
and write $Q_n$ (or $Q_n^{q_0}$, or $Q_n^{\sigma_I, \sigma_{II}}$,
etc) instead of $ Q_n^{\epsilon, q_0, \sigma_I, \sigma_{II}}$, if no
ambiguity arises. Let:
\begin{equation}\label{processMp}
\begin{split}
& \, Q_0\equiv q_0, \\ & \, Q_n\big((w_1,s_1,t_1), \ldots,
(w_n,s_n,t_n)\big) = q_{n-1} * \left\{\begin{array}{ll} 
\rho_{\gamma_\p \epsilon}\big(\sigma_I^{n-1}(h_{n-1})\big) * \rho_\epsilon(w_n) &
\mbox{for } s_n=1\\ \rho_{\gamma_\p
  \epsilon}\big(\sigma_{II}^{n-1}(h_{n-1})\big) * \rho_\epsilon(w_n) &
\mbox{for } s_n=2\\ \rho_\epsilon(w_n) & \mbox{for } s_n=3\end{array} \right.\\
& \mbox{ where } ~~ q_{n-1}= (x_{n-1}, y_{n-1}, z_{n-1}) =
Q_{n-1}\big((w_1,s_1,t_1), \ldots,(w_{n-1},s_{n-1},t_{n-1})\big)  \\ &
\mbox{ and } ~~ ~ h_{n-1}= \big(q_0, (q_1,w_1,s_1,t_1), \ldots, (q_{n-1},w_{n-1},s_{n-1},t_{n-1})\big)\in H_{n-1}.
\end{split}
\end{equation}
In this ``game'', the position $q_{n-1}$ is first advanced
(deterministically) according to the two players' ``strategies''
$\sigma_I$ and $\sigma_{II}$ by a shift in
$B_{\gamma_\p\epsilon}(0)$, and then (randomly) uniformly by a further
shift in the $3$-dimensional Kor\'aniy ball $B_\epsilon(0)$. The
deterministic shifts $\rho_{\gamma_\p \epsilon}\circ\sigma_I^{n-1}$ 
and $\rho_{\gamma_\p \epsilon}\circ\sigma_{II}^{n-1}$
are activated according to the value of the equally probable outcomes $s_n\in\{1,2,3\}$.
Namely, $s_n=1$ results in activating $\sigma_I$ and $s_n=2$ in
activating $\sigma_{II}$, whereas $s_n=3$ corresponds to not
activating any of these strategies.

\medskip

{\bf 3.} The auxiliary variables $t_n\in (0,1)$ serve as a threshold
for reading the eventual value from the prescribed boundary data.
Let $\mathcal{D}\subset\heis$ be an open, bounded and connected
set. Define the random variable $\tau^{\epsilon, q_0, \sigma_I,
\sigma_{II}}:\Omega\to\mathbb{N}\cup\{\infty\}$ in: 
\begin{equation}\label{tautau}
\tau^{\epsilon, q_0, \sigma_I,
  \sigma_{II}}\big((w_1,s_1,t_1),(w_2,s_2,t_2),\ldots\big)=\min\big\{n\geq
1;~ t_n>d_\epsilon(q_{n-1})\big\}
\end{equation}
where:
$$d_\epsilon(q) = \frac{1}{\epsilon}\min\big\{\epsilon, \mbox{dist}(q,
\mathbb{H}\setminus\mathcal{D})\big\}$$ 
is the scaled Kor\'aniy distance from the complement of the domain
$\mathcal{D}$. As before, we drop the superscripts and write $\tau$
instead of $\tau^{\epsilon, q_0, \sigma_I, \sigma_{II}}$ if there is
no ambiguity. Clearly, $\tau$ is $\mathcal{F}$-measurable and, in
fact, it is a stopping time relative to the filtration
$\{\mathcal{F}\}_{n=0}^\infty$ because:

\begin{Prop}\label{stop_time_p}
In the above setting: $\mathbb{P}(\tau<\infty)=1$.
\end{Prop}
\begin{proof}
Consider the following set of ``advancing'' shifts: $D_{adv} =
\{w=(a,b,c)\in B_1(0); ~ a>\frac{1}{2}\}$. Since $\mathcal{D}$ is
bounded, there exists $n\geq 1$ (depending on $\epsilon$) such that:
$$q_0* \rho_\epsilon(w_1)*\ldots *\rho_\epsilon(w_n) \not\in \mathcal{D} \qquad \mbox{ for all }\;
  q_0\in \mathcal{D}\;\; \mbox{ and }\; w_1\in D_{adv},~ i=1\ldots n.$$
Define $\eta={ \displaystyle \Big(\frac{|D_{adv}|}{|B_1(0)|}\cdot
  \frac{1}{3}\Big)^n}>0$ and note that:
$$\mathbb{P}(\tau\leq n)\geq
\mathbb{P}\Big(\big(D_{adv}\times\big\{\frac{1}{3}\big\}\times (0,1)\big)^n\times
\prod_{i=n+1}^\infty\Omega_1\Big)=\eta.$$
It follows by induction that: $\mathbb{P}(\tau>kn)\leq (1-\eta)^k$ for
all $k\in\mathbb{N}$. The proof is concluded by observing:
$\displaystyle\mathbb{P}(\tau=\infty) = \mathbb{P}\big(\bigcap_{k=1}^\infty\{\tau>
kn\}\big) = \lim_{k\to \infty}\mathbb{P}(\tau>kn)\leq \lim_{k\to
  \infty}(1-\eta)^k = 0.$
\end{proof}

\begin{Rem}\label{rem_new}
For $1<\p<\infty$, one can easily  construct a similar process as in
(\ref{processMp}), modelled on the expansion \eqref{dpp3}. In this case, the position of
the token $q_{n-1}$ is first advanced deterministically
according to the strategies $\sigma_I$ or $\sigma_{II}$ by a
shift $\rho_\epsilon(y)\in B_\epsilon(0)$,
and then randomly uniformly by a shift in the Kor\'aniy ellipsoid
$\displaystyle{\rho_\epsilon\Big(E\big(0, s_\p, 1+(a_\p-1)|y_{hor}|^2,
  \frac{y_{hor}}{|y_{hor}|}\big)\Big)}$.
The deterministic shifts $\rho_{\epsilon}\circ\sigma_I^{n-1}$ 
and $\rho_{\epsilon}\circ\sigma_{II}^{n-1}$
are activated according to the value of the sequence of i.i.d. random variables  $s_n\in\{1,2\}$.  
The random stopping time $\tau=\tau^{\epsilon, q_0, \sigma_I, \sigma_{II}}$ is
defined as in (\ref{tautau}) through the auxiliary variables $t_n$. One can check
that sufficient conditions  on the admissible parameters $a_\p$ and $s_{\p}$ in
\eqref{spAndap}  in order to have $\mathbb{P}(\tau<\infty)=1$ are:  
$$ a_\p\leq 1 \quad \text{and}\quad  s_\p a_\p>1
\qquad\quad\text{or} \qquad\quad a_\p\geq
1\quad\text{and}\quad s_\p>1.  $$ 
The proof can be adapted from the Euclidean case (see Lemma 4.1
in \cite{Lew}) by noting that the projection of a Kor\'aniy
ellipsoid onto the horizontal plane reduces to a Euclidean ellipse. 
\end{Rem}

\medskip

{\bf 4.} Given the data  $F\in\mathcal{C}(\heis)$, define the functions:
\begin{equation}\label{ue_def_p}
\begin{split}
& u_I^\epsilon(q) =
\sup_{\sigma_I}\inf_{\sigma_{II}}\mathbb{E}\Big[F\circ \big(Q^{\epsilon, q,
\sigma_I, \sigma_{II}}\big)_{\tau^{\epsilon, q, \sigma_I,  \sigma_{II}}-1}\Big], \\
& u_{II}^\epsilon(q) =
\inf_{\sigma_{II}}\sup_{\sigma_{I}}\mathbb{E}\Big[F\circ \big(Q^{\epsilon, q,
\sigma_I, \sigma_{II}}\big)_{\tau^{\epsilon, q, \sigma_I, \sigma_{II}}-1}\Big]. 
\end{split}
\end{equation}
The main result in Theorem \ref{are_equal} will show that for each
$\epsilon\ll 1$ we have: $u_I^\epsilon = u_{II}^\epsilon
\in\mathcal{C}(\heis)$ coincide with the unique solution to the
dynamic programming principle in section \ref{sec_dpp_analysis},
modelled on the expansion (\ref{dpp2}). It is also clear that the
values of $u_{I, II}^\epsilon$ depend only on the values of $F$ in the
$(\gamma_\p+1)\epsilon$-neighbourhood of $\partial\mathcal{D}$. In
section \ref{sec_convp} we will prove that as $\epsilon\to 0$, the
uniform limit of $u_{I, II}^\epsilon$ that depends only on $F_{\mid\partial\mathcal{D}}$,
is $\p$-$\heis$-harmonic in $\mathcal{D}$ and coincides with $F$ on
$\partial\mathcal{D}$.

\section{The dynamic programming principle modelled on (\ref{dpp2})}\label{sec_dpp_analysis}

Let $\mathcal{D}\subset\heis$ be an open, bounded,
connected domain and let $F\in\mathcal{C}(\heis)$ be a bounded data function.  We have the following: 

\begin{Teo}\label{thD}
For every $\epsilon\in (0,1)$ there exists a unique function
$u:\heis\to\mathbb{R}$ (denoted further by $u_\epsilon$),
automatically continuous and bounded, such that:
\begin{equation}\label{DPP2} 
\begin{split}
u(q) = & ~d_\epsilon(q) \bigg(\frac{1}{3}\mathcal{A}_3(u,\epsilon)(q) +
\frac{1}{3}\inf_{p\in B_{\gamma_\p \epsilon}(q)}\mathcal{A}_3(u,\epsilon)(p) 
+ \frac{1}{3}\sup_{p\in B_{\gamma_\p
    \epsilon}(q)}\mathcal{A}_3(u,\epsilon)(p)\bigg) \\ & ~+
\big(1-d_\epsilon(q)\big)F(q)\quad\qquad\qquad \quad\qquad\qquad
\quad\qquad\qquad \quad\qquad \mbox{ for all }\;q\in\heis.
\end{split}
\end{equation}
The solution operator to (\ref{DPP2}) is monotone, i.e. if $F\leq \bar
F$ then the corresponding solutions satisfy: $u_\epsilon\leq \bar u_\epsilon$.
\end{Teo}
\begin{proof}
{\bf 1.} We remark that by continuity of: the averaging operator
$p\mapsto \mathcal{A}_3(u,\epsilon)(p)$, the weight function
$d_\epsilon$ and the data $F$, the solution of (\ref{DPP2})
is indeed automatically continuous. Define the operators
$T,S:\mathcal{C}(\heis)\to \mathcal{C}(\heis)$ in:
$$(Sv)(q) = \frac{1}{3}\mathcal{A}_3(u,\epsilon)(q) +
\frac{1}{3}\inf_{B_{\gamma_\p \epsilon}(q)}\mathcal{A}_3(u,\epsilon)
+ \frac{1}{3}\sup_{B_{\gamma_\p \epsilon}(q)}\mathcal{A}_3(u,\epsilon), \qquad
Tv = d_\epsilon Sv + (1-d_\epsilon)F.$$
Clearly $S$ (and likewise $T$) is monotone, namely: $Sv\leq S\bar v$ 
if $v\leq \bar v$. Observe further that:
\begin{equation}\label{Sfact}
\begin{split}
|Sv(q)- S\bar v(q)| & \leq \frac{1}{3}\Big( |\mathcal{A}_3(v-\bar
v,\epsilon)(q)| + |\inf_{B_{\gamma_\p \epsilon}(q)}\mathcal{A}_3(v,\epsilon)
- \inf_{B_{\gamma_\p \epsilon}(q)}\mathcal{A}_3(\bar v,\epsilon)| \\ &
\qquad \qquad \qquad \qquad \quad \;\; + |\sup_{B_{\gamma_\p \epsilon}(q)}\mathcal{A}_3(v,\epsilon) -
\sup_{B_{\gamma_\p \epsilon}(q)}\mathcal{A}_3(\bar v,\epsilon)|\Big)
\\ & \leq \frac{1}{3} \mathcal{A}_3(|v-\bar
v|,\epsilon)(q) + \frac{2}{3} \sup_{B_{\gamma_\p \epsilon}(q)}\mathcal{A}_3(|v-\bar v|,\epsilon)\leq
\sup_{B_{\gamma_\p \epsilon}(q)}\mathcal{A}_3(|v-\bar v|,\epsilon).
\end{split}
\end{equation}
The solution $u$ of (\ref{DPP2}) is obtained as the limit of
iterations $u_{n+1}=Tu_n$, where we set $u_0\equiv const \leq \inf
F$. Since $u_1=Tu_0 \geq u_0$ on $\heis$ and in view of the monotonicity of $T$, the
sequence $\{u_n\}_{n=0}^\infty$ is non-decreasing.  It is also bounded
(by $\|F\|_{\mathcal{C}(\heis)}$) and thus it converges pointwise to a
(bounded) limit $u:\heis\to\mathbb{R}$. By the calculation  in
(\ref{Sfact}), $u$ must be a fixed point of $T$, hence a solution to (\ref{DPP2}).
We also remark that the monotonicity of $S$ yields the monotonicity of the solution
operator to (\ref{DPP2}).

\medskip

{\bf 2.} It remains to show uniqueness. If  $u, \bar u$ both solve
(\ref{DPP2}), then define
$M=\sup_{q\in\heis}|u(q)-\bar u(q)| = \sup_{q\in\mathcal{D}}|u(q)-\bar
u(q)|$ and consider any maximizer $q_0\in\mathcal{D}$, where $|u(q_0)-\bar u(q_0)| =M$.
By (\ref{Sfact}) we obtain:
\begin{equation*}
\begin{split}
M & =|u(q_0)-\bar u(q_0)| = d_\epsilon(q_0) |Su(q_0)- S\bar u(q_0)|
\\ & \leq \frac{d_\epsilon(q_0)}{3} \mathcal{A}_3(|u-\bar 
u|,\epsilon)(q_0) + \frac{2 d_\epsilon(q_0)}{3} \sup_{B_{\gamma_\p
    \epsilon}(q_0)}\mathcal{A}_3(|u-\bar u|,\epsilon)\leq d_\epsilon(q_0)M \leq M,
\end{split}
\end{equation*}
yielding $\mathcal{A}_3(|u-\bar u|,\epsilon)(q_0)=M$. Consequently,
$B_\epsilon(q_0)\subset D_M= \{|u-\bar u|=M\}$ and hence 
the set $D_M$ is open in $\heis$. Since $D_M$ is 
obviously closed and nonempty, there must be $D_M=\heis$ and since
$u-\bar u=0$ on $\heis\setminus \mathcal{D}$,  it follows that $M=0$.
Thus $u=\bar u$, proving the claim.
\end{proof}

\begin{Rem}
It is not hard to observe that the sequence
$\{u_n\}_{n=1}^\infty$ in the proof of Theorem \ref{thD} converges to
$u=u_\epsilon$ uniformly. In fact, the iteration procedure $u_{n+1}=Tu_n$
started by any bounded and continuous function $u_0$ converges uniformly
to the unique solution $u_\epsilon$. 
\end{Rem}

\bigskip

\begin{Teo}\label{are_equal}
%Given $F\in\mathcal{C}(\heis)$, $\epsilon\in (0,1)$ and an open, bounded and connected domain
%$\mathcal{D}\subset\mathbb{H}$, let
For every $\epsilon\in (0,1)$, let $u_I^\epsilon$, $u_{II}^\epsilon$
be as in (\ref{ue_def_p}) and $u_\epsilon$ as in Theorem \ref{thD}. Then: 
$$u_I^\epsilon =  u_\epsilon =u_{II}^\epsilon.$$
\end{Teo}
\begin{proof}
{\bf 1.}  We drop the sub/superscript $\epsilon$ for notational
convenience. To show that $u_{II}\leq u$, fix $q_0\in\heis$ and
$\eta>0$. We first observe that there exists a strategy
$\sigma_{0,II}$ where $\sigma_{0,II}^n(h_n) = \sigma_{0,II}^n(q_n)$
satisfies for every $n\geq 0$ and $h_n\in H_n$:
\begin{equation}\label{infimize}
\mathcal{A}_3(u,\epsilon)\big(q_n*\rho_{\gamma_\p
\epsilon}(\sigma_{0,II}^n(q_n))\big) \leq \inf_{B_{\gamma_\p
\epsilon}(q_n)} \mathcal{A}_3(u,\epsilon) + \frac{\eta}{2^{n+1}}.
\end{equation}
Indeed, it suffices to show, in view of continuity of $\mathcal{A}_3(u,\epsilon)$, that given
$v\in\mathcal{C}(\heis)$ and $r,\eta>0$, there exists an
infimizing-related Borel measurable ``selection'' 
function $\sigma:\heis\to \heis$ such that $v(\sigma(q))<
\inf_{B_r(q)}v+\eta$ and $\sigma(q)\in B_r(q)$ for all
$q\in\heis$. Using continuity of $v$ and a localisation argument, if
necessary, we note that there exists $\delta>0$ such that:
$$\big|\inf_{B_r(q)}v - \inf_{B_r(p)}v\big| < \frac{\eta}{2} \quad \mbox{for
all }\; |p-q|<\delta.$$
Let $\{B_\delta^3(p_i)\}_{i=1}^\infty$ be a locally finite covering of
$\heis$. For each $i=1\ldots \infty$, choose $q_i\in B_r(p_i)$
satisfying: $|\inf_{B_r(p_i)}v - v(q_i)|<\frac{\eta}{2}$. Finally, define:
$$\sigma(q) = q_i \quad \mbox{for }\; q\in B_\delta^3(p_i)\setminus
\bigcup_{j=1}^{i-1} B_\delta^3(p_j).$$
Clearly, the piecewise constant function $\sigma$ is Borel  regular and
infimizing-related with the prescribed parameters $r,\eta$.

\medskip

{\bf 2.} Fix a strategy $\sigma_I$ and consider the following sequence of
random variables $M_n:\Omega\to\mathbb{R}$:
$$M_n=(u\circ Q_n)\mathbbm{1}_{\tau>n} + (F\circ
Q_{\tau-1})\mathbbm{1}_{\tau\leq n} + \frac{\eta}{2^n}.$$
We show that $\{M_n\}_{n=0}^\infty$ is a supermartingale with respect
to the filtration $\{\mathcal{F}_n\}_{n=0}^\infty$. Clearly:
\begin{equation}\label{decompo}
\begin{split}
\mathbb{E}\big(M_n\mid\mathcal{F}_{n-1}\big) = ~ &
\mathbb{E}\big((u\circ Q_n)\mathbbm{1}_{\tau>n}\mid\mathcal{F}_{n-1}\big) + 
\mathbb{E}\big((F\circ Q_{n-1})\mathbbm{1}_{\tau=n}\mid\mathcal{F}_{n-1}\big) \\ & + 
\mathbb{E}\big((F\circ Q_{\tau -1})\mathbbm{1}_{\tau<n}\mid\mathcal{F}_{n-1}\big) +
\frac{\eta}{2^n}\qquad \mbox{a.s.} 
\end{split}
\end{equation}
We readily observe that: $\mathbb{E}\big((F\circ Q_{\tau
  -1})\mathbbm{1}_{\tau<n}\mid\mathcal{F}_{n-1}\big) 
= (F\circ Q_{\tau -1})\mathbbm{1}_{\tau<n}$. Further, writing
$\mathbbm{1}_{\tau=n} = \mathbbm{1}_{\tau\geq n}
\mathbbm{1}_{t_n>d_\epsilon(q_{n-1})}$, it follows that:
\begin{equation*}
\begin{split}
\mathbb{E}\big((F\circ Q_{n-1})&\mathbbm{1}_{\tau=n}\mid\mathcal{F}_{n-1}\big) =
\mathbb{E}\big(\mathbbm{1}_{t_n>d_\epsilon(q_{n-1})}\mid\mathcal{F}_{n-1}\big) \cdot
(F\circ Q_{n-1}) \mathbbm{1}_{\tau\geq n} \\ & = \mathbb{P}_1 \big(t_n>d_\epsilon(q_{n-1})\big)
\cdot (F\circ Q_{n-1}) \mathbbm{1}_{\tau\geq n} = \big(1- d_\epsilon(q_{n-1})\big)
(F\circ Q_{n-1}) \mathbbm{1}_{\tau\geq n}  \qquad \mbox{a.s.} 
\end{split}
\end{equation*}
Similarly, since $\mathbbm{1}_{\tau>n} = \mathbbm{1}_{\tau\geq n}
\mathbbm{1}_{t_n\leq d_\epsilon(q_{n-1})}$, we get in view of (\ref{infimize}):
\begin{equation*}
\begin{split}
\mathbb{E}&\big((u\circ
Q_{n})\mathbbm{1}_{\tau>n}\mid\mathcal{F}_{n-1}\big) = \int_{\Omega_1} u\circ
Q_n\dd \mathbb{P}_1 \cdot d_\epsilon(q_{n-1}) \mathbbm{1}_{\tau\geq n} 
\\ & = \Big(\mathcal{A}_3(u, \epsilon)(q_{n-1})+
\mathcal{A}_3(u, \epsilon)(q_{n-1}*\rho_{\gamma_\p \epsilon}(\sigma_{I}^{n-1})) +
\mathcal{A}_3(u, \epsilon)(q_{n-1}*\rho_{\gamma_\p \epsilon}(\sigma_{0,II}^{n-1}))\Big)
\frac{ d_\epsilon(q_{n-1})}{3} \mathbbm{1}_{\tau\geq n}  \\ & \leq
\big( (Su)\circ Q_{n-1} +\frac{\eta}{2^n}\big) d_\epsilon(q_{n-1}) \mathbbm{1}_{\tau\geq n} 
\qquad \mbox{a.s.} 
\end{split}
\end{equation*}

Concluding, by (\ref{DPP2}) the decomposition (\ref{decompo}) yields:
\begin{equation*}
\begin{split}
\mathbb{E}\big(M_n\mid\mathcal{F}_{n-1}\big)  \leq
& \Big(d_\epsilon(q_{n-1})\big((Su)\circ Q_{n-1}\big) +
(1-d_\epsilon(q_{n-1}))\big(F\circ Q_{n-1}\big)\Big)
\mathbbm{1}_{\tau\geq n} \\ & + (F\circ Q_{\tau-1})
\mathbbm{1}_{\tau\leq n-1} + \frac{\eta}{2^{n-1}} = M_{n-1} \qquad \mbox{a.s.} 
\end{split}
\end{equation*}

\medskip

{\bf 3.} The supermartingale property of  $\{M_n\}_{n=0}^\infty$ being
established, we conclude that:
$$u(q_0) +\eta = \mathbb{E}\big[ M_0 \big] \geq \mathbb{E}\big[ M_\tau
\big] = \mathbb{E}\big[ F\circ Q_{\tau-1}\big] +\frac{\eta}{2^\tau}. $$
Thus: 
$$u_{II}(q_0) \leq \sup_{\sigma_{I}}\mathbb{E}\big[ F\circ
(Q^{\sigma_I, \sigma_{II,0}})_{\tau-1}\big] \leq u(q_0)+\eta.$$
As $\eta>0$ was arbitrary, we obtain the claimed comparison
$u_{II}(q_0)\leq u(q_0)$. For the reverse inequality $ u(q_0)\leq u_{I}(q_0)$,
we use a symmetric argument, with an almost-maximizing strategy
$\sigma_{0,I}$ and the resulting submartingale $\bar M_n=(u\circ Q_n)\mathbbm{1}_{\tau>n} + (F\circ
Q_{\tau-1})\mathbbm{1}_{\tau\leq n} - \frac{\eta}{2^n}$, along a given
yet arbitrary strategy $\sigma_{II}$. The obvious estimate
$u_{I}(q_0)\leq u_{II}(q_0)$ concludes the proof.
\end{proof}

\section{The first convergence theorem}\label{sec_conv}

We prove the first convergence result below, via 
an analytical argument, although a probabilistic one is possible as well,
in view of the interpretation of $u_\epsilon$ in Theorem
\ref{are_equal}. Our proof mimics the construction for the
Euclidean case in \cite{Lew}, which is based on the observation  that for $s$ sufficiently
large, the mapping $q\mapsto |q|_K^s$ yields the variation that pushes the
$\p$-$\heis$-harmonic function $F$ into the region of $\p$-$\heis$-subharmonicity. 

\begin{Teo}\label{conv_nondegene}
Let $F\in\mathcal{C}^2(\heis)$ be a bounded data function that satisfies on some
open set $U$,  compactly containing $\mathcal{D}$: 
\begin{equation}\label{assu1}
\Delta_{\heis,\p}F = 0\quad \mbox{ and } \quad \nabla_{\heis}F\neq
0 \qquad \mbox{in }\;  U.
\end{equation}
Then the solutions $u_\epsilon$ of (\ref{DPP2}) converge to $F$ uniformly in $\heis$, namely:
\begin{equation}\label{unicon1}
\|u_\epsilon - F\|_{\mathcal{C}(\mathcal{D})}\leq C\epsilon \qquad \mbox{ as }\; \epsilon\to 0,
\end{equation}
with a constant $C$ depending on $F$, $U$, $\mathcal{D}$ and $\p$, but
not on $\epsilon$.
\end{Teo}
\begin{proof}
{\bf 1.} We first note that since $u_\epsilon=F$ on $\heis\setminus
\mathcal{D}$ by construction, (\ref{unicon1}) indeed implies the uniform
convergence of $u_\epsilon$ in $\heis$. 
Also, by applying a left translation it not restrictive to assume that
$U$ does not intersect the interior of the cylinder $\{
q=(x,y,z)=(q_{hor}, z)\in\mathbb{H}\,:\, |q_{hor}|^2=x^2+y^2\leq 1 \}$. 
In particular, this implies $|q|_K\geq 1$ for all $q\in\mathcal{D}$. 
 
\medskip

We now show that there exists $s\geq 4$ and $\hat{\epsilon}>0$ such that the following functions:
\begin{equation*}
v_\epsilon(q)=F(q)+\epsilon|q|_K^s
\end{equation*}
satisfy, for every $\epsilon\in (0, \hat\epsilon)$:
\begin{equation}\label{zzz}
\nabla_\heis v_\epsilon \neq 0 \quad \mbox{ and }\quad \Delta_{\heis, \p} v_\epsilon \geq
\epsilon s \cdot |\nabla_\heis{v_\epsilon}|^{{\p}-2} \quad \mbox{ in }\;\bar{\mathcal{D}}.
\end{equation}
Fix $q\in\bar{\mathcal{D}}$, $\epsilon\in (0,1)$ and denote
$a=\nabla_\heis{v_\epsilon(q)}$ and $b=\nabla_\heis{F(q)}$. By (\ref{assu1}) it follows that:
\begin{equation}\label{pLaplEstimate}
\Delta_{\heis,\p}  v_\epsilon(q) = |\nabla_\heis{v_\epsilon}(q)|^{\p-2} \left(  I+II+III \right),
\end{equation}
where:
\begin{equation*}
\begin{split}
I &= \epsilon\Delta_\heis(|q|_K^s),\\
II&= \epsilon(\p-2) \big\langle \nabla^{2}_{\mathbb{H}}( |q|_K^s
  ):\frac{a}{|a|}\otimes\frac{a}{|a|} \big\rangle,\\ 
III&= (\p-2)\big \langle  \nabla^{2}_{\mathbb{H}} F(q):
\frac{a}{|a|}\otimes\frac{a}{|a|}-\frac{b}{|b|}\otimes\frac{b}{|b|}\big\rangle. 
\end{split}
\end{equation*}
Denoting $\xi = |q_{hor}|^2q_{hor}+4zq_{hor}^\perp= \big(x(x^2+y^2)-4yz,
y(x^2+y^2)+4xz\big )\in\mathbb{R}^2$, a further computation shows that:
\begin{equation*}
\begin{split}
& \nabla_\heis{ ( |q|_K^s ) } = s |q|_K^{s-4} \xi, \\
& \Delta_\heis{ (|q|_K^s) } = s(s+2)|q_{hor}|^2|q|_K^{s-4},\\
& \nabla^{2}_{\mathbb{H}}( |q|_k^s ) = s |q|_K^{s-8}\big( (s-4)\xi\otimes\xi+3|q_{hor}|^2 |q|_K^4 Id_2 \big).
\end{split}
\end{equation*}
Consequently, we have:
\begin{equation*}
\big\langle \nabla^{2}_{\mathbb{H}} ( |q|_K^s ) : \frac{a}{|a|}\otimes \frac{a}{|a|} \big\rangle
= s |q |_K^{s-8} \Big(  (s-4)\big\langle \xi\otimes\xi :
\frac{a}{|a|}\otimes\frac{a}{|a|}\big \rangle+3 |q_{hor}|^2|q|_K^4 \Big )
\geq 3s |q_{hor}|^2 |q|_K^{s-4} . 
\end{equation*}
Also, since $|\xi|=|q_{hor}||q|_K^2\leq |q_{hor}|^2|q|_K^2$, observe that:
\begin{equation*}
\begin{split}
\Big |\big\langle \nabla^{2}_{\mathbb{H}} F(q)\,:\, \frac{a}{|a|}\otimes\frac{a}{|a|}
-\frac{b}{|b|}\otimes\frac{b}{|b|}\big\rangle\Big|  
&\leq 4| \nabla^{2}_{\mathbb{H}} F(q)|\frac{|a-b|}{|b|}
\leq 4\epsilon s |\xi|\,|q|_K^{s-4} \frac{|\nabla^{2}_{\mathbb{H}} F(q)|}{|\nabla_\heis{F}(q)|}\\
&\leq 4\epsilon s |q_{hor}|^2 \cdot |q|_K^{s-2} \frac{|\nabla^{2}_{\mathbb{H}} F(q)|}{|\nabla_\heis{F}(q)|}
\end{split}
\end{equation*}  

We gather the estimates above to get, in view of \eqref{pLaplEstimate}:
\begin{equation*}
\Delta_{\heis, \p} v_\epsilon(q) \geq |\nabla_\heis{v_\epsilon(q)}|^{\p-2}
\epsilon s \cdot |q_{hor}|^2|q|_K^{s-4}\Big( (s+2)+3(\p-2)-4(\p-2)
|q|_K^2\cdot \frac{|
    \nabla^{2}_{\mathbb{H}} F(q)|}{|\nabla_\heis{F(q)}|} \Big), 
\end{equation*}
It is clear that for $s$ large enough, the quantity in parentheses
above is uniformly bounded from below by $1$ on $\bar{\mathcal{D}}$.
This justifies the second bound in (\ref{zzz}), since $|q|_K,
|q_{hor}|\geq 1$ on $\bar{\mathcal{D}}$.
Finally, choosing $\hat\epsilon$ sufficiently small we ensure that
$\nabla_\heis{v_\epsilon}\neq 0$ in $\bar{\mathcal{D}}$ for $\epsilon\in(0,\hat{\epsilon})$.

\medskip

{\bf 2.} We claim that $s$ and $\hat{\epsilon}$ in step 1 can further
be chosen in a way that for all $\epsilon\in(0,\hat{\epsilon})$:
\begin{equation}\label{want}
v_\epsilon\leq S_\epsilon v_\epsilon \quad \text{in}\,\,\bar{\mathcal{D}}.
\end{equation}
Indeed, a careful analysis of the remainder terms in Taylor's expansion \eqref{dpp2} reveals that: 
\begin{equation}\label{want2}
S_\epsilon v_\epsilon(q)= v_\epsilon(q)+\frac{r^2}{3\pi}\cdot \frac{\Delta_{\heis, \p}
v_\epsilon(q)}{|\nabla_\heis{v}_\epsilon(q)|^{\p-2}} + R_2(\epsilon,q), 
\end{equation} 
where: 
$$|R_2(\epsilon,q)|\leq C_\p \epsilon^2
\displaystyle{\osc_{B_\epsilon(q)}}(|\nabla^{2}_\heis v_\epsilon|+|Zv_\epsilon|)+C \epsilon^3.$$ 
Above, we denoted by $C_\p$ a constant depending only on $\p$, whereas
$C$ is a constant depending only on the quantities $|\nabla_\heis{v_\epsilon}(q)|$
and $\|\nabla^2v_\epsilon\|_{\mathcal{C}(B_{\epsilon}(q))}$, that
remain uniformly bounded in $q\in \bar{\mathcal{D}}$ for small $\epsilon$.
Since $v_\epsilon$ is the sum of the smooth on $U$ function $q\mapsto \epsilon|q|_K^s$,
and a $\p$-harmonic function $u$ that is also smooth in virtue of its non vanishing horizontal
gradient  (see \cite{R}), we obtain that:
$\osc_{B_\epsilon(q)}(|\nabla^{2}_Hv_\epsilon|+|Tv_\epsilon|)\leq
C_{u}\epsilon$, with $C_u$ depending only on the third derivatives of
$u$ (and on $\mathcal{D}$). In conclusion, (\ref{want2}) and
(\ref{zzz}) imply (\ref{want}) for $s$ sufficiently large, because:
\begin{equation*}
v_\epsilon(q)-S_\epsilon v_\epsilon(q)
\leq -\frac{\epsilon^2}{3\pi}\frac{\Delta_{\heis, \p} v_\epsilon(q)}{|\nabla_\heis{v_\epsilon}(q)|^{\p-2}}+
\epsilon^3(C_\p C_u+C) \leq  \epsilon^3\Big(-\frac{s}{3\pi} + C_\p C_u + C \Big) \leq 0.
\end{equation*}

\medskip

{\bf 3.} Let $A$ be a compact set in: $\mathcal{D}\subset
A\subset U$. Fix $\epsilon\in (0,\hat\epsilon)$ and for each $q\in A$ consider:
$$\phi_\epsilon(q)=v_\epsilon(q)-u_\epsilon(q) = F(q)-u_\epsilon(q)+\epsilon|q|_K^s.$$ 
By (\ref{want}) and (\ref{DPP2}) we get:
\begin{equation}\label{13}
\begin{split}
\phi_\epsilon(q)
&= d_\epsilon(q)( v_\epsilon(q)-S_\epsilon u_\epsilon(q) ) + (1-d_\epsilon(q))(v_\epsilon(q)-F(q))\\
&\leq d_\epsilon(q)( S_\epsilon v_\epsilon(q)-S_\epsilon u_\epsilon(q) )
+( 1-d_\epsilon(q) )(v_\epsilon(q)-F(q))\\
&\leq d_\epsilon(q) \Big(\frac{1}{3}\mathcal{A}_3(\phi_\epsilon,
\epsilon)(q)+\frac{2}{3}\sup_{B_{(1+\gamma_\p)\epsilon}(q)}\phi_\epsilon
\Big) +( 1-d_\epsilon(q) )\big(v_\epsilon(q)-F(q)\big).
\end{split}
\end{equation}
Define: 
$$M_\epsilon=\max_{A}\phi_\epsilon.$$ 
We claim that there exists $q_0\in A$ with $d_\epsilon(q_0)<1$ and such that
$\phi_\epsilon(q_0)=M_\epsilon$.
To prove the claim, define
$\mathcal{D}^\epsilon=\big\{q\in\mathcal{D};~ \mbox{dist}(q,\partial\mathcal{D})\geq \epsilon \}$. 
We can assume that the closed set $\mathcal{D}^\epsilon\cap\{\phi_\epsilon=M_\epsilon\}$ is nonempty; 
otherwise the claim would be obvious. Let $\mathcal{D}_0^\epsilon$ be
a nonempty connected component of $\mathcal{D}^\epsilon$ and denote
$\mathcal{D}_M^\epsilon=\mathcal{D}_0^\epsilon\cap
\{\phi_\epsilon=M_\epsilon\}$. Clearly,  $\mathcal{D}_M^\epsilon$ is
closed in $\mathcal{D}_0^\epsilon$; we now show that it is also open. Let
$q\in\mathcal{D}_M^\epsilon$. Since $d_\epsilon(q)=1$ from (\ref{13}) it follows that: 
\begin{equation*}
M_\epsilon =\phi_\epsilon(q) \leq
\frac{1}{3}\mathcal{A}_3(\phi_\epsilon,\epsilon)(q)+\frac{2}{3}\sup_{B_{(1+\gamma_\p)\epsilon}(q)}\phi_\epsilon 
\leq M_\epsilon.
\end{equation*} 
Consequently, $\mathcal{A}_3(\phi_\epsilon,\epsilon)(q)=M_\epsilon$,
implying $\phi_\epsilon\equiv M_\epsilon$ in $B_\epsilon(q)$ and thus openness 
of $\mathcal{D}_M^\epsilon$ in $\mathcal{D}_0^\epsilon$.
In particular, $\mathcal{D}_M^\epsilon$ contains a point $\bar{q}\in\partial
\mathcal{D}^\epsilon$. Repeating the previous argument for $\bar q$
results in $\phi_\epsilon\equiv M_\epsilon$ in $B_\epsilon(\bar{q})$,
proving the claim.

\medskip

We now complete the proof of Theorem \ref{conv_nondegene} by deducing a bound on $M_\epsilon$.
If $M_\epsilon=\phi_\epsilon(q_0)$ for some $q_0\in\bar{\mathcal{D}}$
with $d_\epsilon(q_0)<1$, then (\ref{13}) yields:
$ M_\epsilon=\phi_\epsilon(q_0)
\leq d_\epsilon(q_0)M_\epsilon +(1-d_\epsilon(q_0))\big(v_\epsilon(q_0)-F(q_0)\big), $
which implies:
$$ M_\epsilon\leq v_\epsilon(q_0)-F(q_0)=\epsilon|q_0|_K^s .$$
On the other hand, if $M_\epsilon=\phi_\epsilon(q_0)$ for some $q_0\in
A\setminus \mathcal{D}$, then $d_\epsilon(q_0)=0$, hence likewise:
$ M_\epsilon=\phi_\epsilon(q_0)= v_\epsilon(q_0)-F(q_0)=
\epsilon|q_0|_K^s.$ In either case:
$$ \max_{\bar{\mathcal{D}}}(F-u_\epsilon)\leq
\max_{\bar{\mathcal{D}}}\phi_\epsilon + C\epsilon \leq 2C\epsilon $$
where $C=\max_{q\in U}|q|_K^s$ is independent of $\epsilon$. A
symmetric argument applied to $-u$ after noting that
$(-u)_\epsilon=-u_\epsilon$ gives:  $  \min_{\bar{\mathcal{D}}}( u-u_\epsilon)\geq -2C\epsilon$.
The proof is done.
\end{proof}

\section{Convergence of $u_\epsilon$ and game-regularity}\label{sec_convp}

Towards checking convergence of $\{u_\epsilon\}_{\epsilon\to 0}$, we
first prove a counterpart of Lemma \ref{A0} for the case of
$\p>2$. Namely, we will show that equicontinuity of $\{u_\epsilon\}_{\epsilon\to 0}$ on
$\heis$ is a consequence of equicontinuity ``at $\partial\mathcal{D}$''. This
last property will be, in turn, implied by the ``game-regularity''
condition, which mimics the ``walk-regularity'' Definition \ref{d_wr}
in the context of the stochastic Tug of War as in
section \ref{sec_setup_p}. The aforementioned condition, given in Definition
\ref{def_gamereg}, and the following Lemma \ref{thm_gamenotreg_noconv} and Theorem
\ref{thm_gamereg_conv} are adapted from the same statements in
the seminal paper \cite{PS}, where another Tug of War
was proposed in the Euclidean setting, for $\p\in (1, \infty)$.

\medskip

Let $\mathcal{D}\subset\heis$ be an open, bounded,
connected domain and let $F\in\mathcal{C}(\heis)$ be a bounded
data function.  We have the following:

\begin{Teo}\label{transfer_p}
Let $\{u_\epsilon\}_{\epsilon\to 0}$ be the family of solutions to
(\ref{DPP2}). Assume that for every $\eta>0$ there exists $\delta>0$
and $\hat\epsilon\in (0,1)$ such that for all $\epsilon\in (0,\hat\epsilon)$ there holds:
\begin{equation}\label{bd_asp}
|u_\epsilon(q_0')- u_\epsilon(q_0)|\leq \eta \qquad \mbox{for all }\;
q_0'\in\mathcal{D}, ~ q_0\in\partial\mathcal{D} \; \mbox{ satisfying
}\; |q_0-q_0'|\leq \delta.
\end{equation}
Then the family $\{u_\epsilon\}_{\epsilon\to 0}$ is equicontinuous in $\bar{\mathcal{D}}$.
\end{Teo}
\begin{proof}
{\bf 1.} We present an analytical proof. A probabilistic argument is
available as well, based on a game translation argument as in the
proof of Lemma \ref{A0}. For every small $\hat\delta>0$ the 
set $\mathcal{D}^{\hat\delta}$ below is open, bounded and connected, where
we define:
\begin{equation*}
\mathcal{D}^{\hat\delta} = \big\{q\in\mathcal{D}; ~ \mbox{dist}(q,
\mathbb{H}\setminus \mathcal{D})>{\hat\delta}\big\} \quad\mbox{ and
}\quad d^{\hat\delta}_\epsilon(q) = \frac{1}{\epsilon}\min\{\epsilon,
\mbox{dist}(q,\heis\setminus \mathcal{D}^{\hat\delta})\}.
\end{equation*}
Fix $\eta>0$. In view of (\ref{bd_asp}) and since without loss
of generality the data function $F$ is constant outside of some large bounded
superset of $\mathcal{D}$ in $\heis$, there exists $\hat\delta>0$  satisfying:
\begin{equation}\label{bd_asp2}
|u_\epsilon(w*q)- u_\epsilon(q)|\leq \eta \qquad \mbox{for all }\;
q\in\heis\setminus \mathcal{D}^{\hat\delta}, ~~|w|\leq\hat\delta, ~~
\epsilon\in (0,\hat\epsilon).
\end{equation}
Further, let $\delta\in (0, \hat\delta)$ be such that:
\begin{equation}\label{bd_asp3}
\begin{split}
|q_0'*q_0^{-1}|\leq\hat\delta \;\mbox{ and }\;&
|q^{-1}*q_0'*q_0^{-1}*q|_K\leq\frac{\hat\delta}{2} \\ & \mbox{ for all }\;
q\in\mathcal{D} \; \mbox{ and all }\; q_0, q_0'\in
\bar{\mathcal{D}}\;\mbox{ satisfying }\; |q_0-q'_0|\leq \delta.
\end{split}
\end{equation}
Fix $q_0, q_0'\in \bar{\mathcal{D}}$ as above and let $\epsilon\in (0,
\frac{\hat\delta}{2})$. Consider the following function $\tilde u_\epsilon\in\mathcal{C}(\heis)$:
$$\tilde u_\epsilon(q) = u_\epsilon(q_0'* q_0^{-1}*q) + \eta.$$
Then $\mathcal{A}_3(\tilde u_\epsilon, \epsilon)(q) =
\mathcal{A}_3(u_\epsilon, \epsilon)(q_0'*q_0^{-1} * q) + \eta$ and
 $\displaystyle{\inf_{B_{\gamma_\p
      \epsilon}(q)}\mathcal{A}_3(\tilde u_\epsilon, \epsilon) =\inf_{B_{\gamma_\p
    \epsilon}(q_0'*q_0^{-1}*q)}\mathcal{A}_3(u_\epsilon, \epsilon)} + \eta,$
with the same identity valid for the supremum as well. Consequently: 
\begin{equation}\label{bd_bd}
\begin{split}
(S_\epsilon \tilde u_\epsilon)(q) & =\frac{1}{3}\mathcal{A}_3(\tilde
u_\epsilon, \epsilon)(q) + \frac{1}{3} \inf_{B_{\gamma_\p \epsilon}(q)}\mathcal{A}_3(\tilde u_\epsilon, \epsilon) 
+ \frac{1}{3} \sup_{B_{\gamma_\p \epsilon}(q)}\mathcal{A}_3(\tilde u_\epsilon, \epsilon) 
\\ & = (S_\epsilon u_\epsilon)(q_0'*q_0^{-1}*q)+\eta =
u_\epsilon(q_0'*q_0^{-1}*q) + \eta = \tilde u_\epsilon(q) \quad
\mbox{for all }\; q\in\mathcal{D}^{\hat\delta}.
\end{split}
\end{equation}
Indeed, in view of $q\in\mathcal{D}^{\hat\delta}$ and (\ref{bd_asp3}) we have:
$$ d(q_0'*q_0^{-1}*q, \heis\setminus \mathcal{D})\geq d(q,
\heis\setminus \mathcal{D}) - d(q_0'*q_0^{-1}*q, q)\geq
{\hat\delta}- \frac{\hat\delta}{2}= \frac{\hat\delta}{2}>\epsilon.$$

\medskip

{\bf 2.} It follows now from (\ref{bd_bd}) that for all $\epsilon\in
(0,\frac{\hat\delta}{2})$ and all $q_0, q_0'\in\bar{\mathcal{D}}$ satisfying
$|q_0-q'_0|\leq \delta$:
\begin{equation*}
\begin{split}
\tilde u_\epsilon= d_\epsilon^{\hat\delta}  (S_\epsilon \tilde u_\epsilon)
+ \big(1 -\tilde d_\epsilon^{\hat\delta}\big) \tilde u_\epsilon\qquad  \mbox{in }\; q\in\heis.
\end{split}
\end{equation*}
On the other hand, $u_\epsilon$ itself similarly solves the same
problem above, subject to its own data $u_\epsilon$ on $\heis\setminus
\mathcal{D}^{\hat\delta}$. Since for every $q\in\heis\setminus
\mathcal{D}^{\hat\delta}$ we have: $\tilde u_\epsilon(q) - u_\epsilon(q)
= u_\epsilon(q_0'*q_0^{-1}*q)-u_\epsilon(q) + \eta \geq 0$ in view of
(\ref{bd_asp2}) and (\ref{bd_asp3}), the monotonicity property in
Theorem \ref{thD} yields:
$$u_\epsilon\leq \tilde u_\epsilon \qquad\mbox{in }\; \heis.$$
Thus, in particular: $u_\epsilon(q_0)-u_\epsilon(q_0')\leq
\eta$. Exchanging $q_0$ with $q_0'$ we get the opposite inequality,
and hence $|u_\epsilon(q_0)-u_\epsilon(q_0')|\leq\eta$,
establishing the claimed equicontinuity of
$\{u_\epsilon\}_{\epsilon\to 0}$ in $\bar{\mathcal{D}}$.
\end{proof}

\medskip

Following \cite{PS} we introduce the following definition.
A point $q_0\in\partial\mathcal{D}$ will
be called {game-regular} if, whenever the game starts near $q_0$,
one of the ``players'' has a strategy for making the game terminate near the same $q_0$,
with high probability. More precisely:

\begin{Def}\label{def_gamereg}
Consider the Tug of War game with noise in (\ref{processMp}) and (\ref{ue_def_p}).
\begin{itemize}
\item[(a)] We say that a point $q_0\in\partial\mathcal{D}$ is {\em game-regular} if for
every $\eta, \delta>0$ there exist $\hat\delta\in (0, \delta)$ and $\hat\epsilon\in
(0, 1)$ such that the following holds. Fix
$\epsilon\in (0, \hat\epsilon)$ and $p_0\in B_{\hat\delta}(q_0)$; there exists then a strategy
$\sigma_{0, I}$ with the property that for every strategy $\sigma_{II}$ we have:
\begin{equation}\label{gar}
\mathbb{P}\big( Q_{\tau-1}\in B_\delta (q_0)\big) \geq 1-\eta.
\end{equation}

\item[(b)] We say that  {$\mathcal{D}$ is game-regular} if every
  boundary point $q_0\in\partial\mathcal{D}$ is game-regular.
\end{itemize}
\end{Def}

\begin{Rem}\label{usp}
As in Definition \ref{d_wr} of walk-regularity, if condition (b)
holds, then $\hat\delta$ and $\hat\epsilon$ in part (a) can be chosen
independently of $q_0$. Also, game-regularity is symmetric with
respect to $\sigma_I$ and $\sigma_{II}$.
\end{Rem}

\medskip

\begin{Lemma}\label{thm_gamenotreg_noconv}
Assume that for every bounded data
$F\in\mathcal{C}(\heis)$, the family of solutions
$\{u_\epsilon\}_{\epsilon\to 0}$ of (\ref{DPP2})  is equicontinuous in
$\bar{\mathcal{D}}$. Then $\mathcal{D}$ is game-regular. 
\end{Lemma}
\begin{proof}
Fix $q_0\in\partial\mathcal{D}$ and let $\eta, \delta\in (0,1)$. Define the
data function:  $F(q) = -\min\big\{1, d(q, q_0)\big\}.$
By assumption and since $u_\epsilon(q_0)=F(q_0)=0$, there exists $\hat\delta\in
(0,\delta)$ and $\hat\epsilon\in (0,1)$ such that:
$$ |u_\epsilon (p_0)|< \eta\delta \;\;\;\mbox{ for all } \;p_0\in
B_{\hat\delta}(q_0)  \;\mbox{ and } \; \epsilon\in (0,\hat\epsilon).$$
Consequently:
$$\sup_{\sigma_I}\inf_{\sigma_{II}}\mathbb{E}\big[F\circ (Q^{\epsilon, p_0})_{\tau-1}\big]
= u_I^\epsilon(p_0) > -\eta\delta,$$
and thus there exists  $\sigma_{0,I}$ with the property
that: $\mathbb{E}\big[F\circ (Q^{\epsilon, p_0, \sigma_{0,I},
  \sigma_{II}})_{\tau-1}\big]>-\eta\delta$ for every strategy $\sigma_{II}$. Then:
$$\mathbb{P}\big(Q_{\tau-1}\not\in B_\delta(q_0)\big) \leq
-\frac{1}{\delta}\int_\Omega F(Q_{\tau-1})\dd\mathbb{P}<\eta,$$
proving (\ref{gar}) and hence game-regularity of $q_0$.
\end{proof}

\begin{Teo}\label{thm_gamereg_conv}
Assume that $\mathcal{D}$ is game-regular. Then, for every bounded data
$F\in\mathcal{C}(\heis)$, the family $\{u_\epsilon\}_{\epsilon\to 0}$ of solutions to (\ref{DPP2})
is equicontinuous in  $\bar{\mathcal{D}}$.
\end{Teo}
\begin{proof}
In virtue of Theorem \ref{transfer_p} it is enough to validate the condition
(\ref{bd_asp}). To this end, fix $\eta>0$ and let $\delta>0$ be such that:
\begin{equation}\label{ga}
|F(p)-F(q_0)|\leq \frac{\eta}{3} \;\;\mbox{ for all }
\;\;q_0\in\partial\mathcal{D} \; \mbox{ and } \; p\in B_\delta(q_0).
\end{equation}
By Remark \ref{usp} and Definition \ref{def_gamereg}, we may choose
$\hat\delta\in (0,\delta)$ and $\hat\epsilon\in (0, \delta)$ such that
for every $\epsilon\in (0, \hat\epsilon)$, every
$q_0\in\partial\mathcal{D}$ and every $p_0\in B_{\hat\delta}(q_0)$,
there exists a strategy $\sigma_{0,II}$ with the property that for
every $\sigma_I$ there holds:
\begin{equation}\label{ga1}
\mathbb{P}\big((Q^{\epsilon, p_0, \sigma_I,
  \sigma_{0,II}})_{\tau-1}\in B_\delta(q_0)\big) \geq 1 - \frac{\eta}{6\|F\|_{\mathcal{C}(\heis)}+1}.
\end{equation}
Let $q_0\in\partial\mathcal{D}$ and $q_0'\in\mathcal{D}$ satisfy
$|q_0-q_0'|\leq\hat\delta$. Then:  
\begin{equation*}
\begin{split}
u_\epsilon(q_0') - u_\epsilon(q_0) & =u_{II}^\epsilon(q_0')-F(q_0)\leq\sup_{\sigma_I}\mathbb{E}
\big[F\circ (Q^{\epsilon, q_0', \sigma_I, \sigma_{0,II}})_{\tau-1} -
F(q_0)\big]\\ & \leq \mathbb{E}\big[F\circ (Q^{\epsilon, q_0', \sigma_{0,I},
\sigma_{0,II}})_{\tau-1} - F(q_0)\big] + \frac{\eta}{3},
\end{split}
\end{equation*}
for some almost-supremizing strategy $\sigma_{0,I}$. Thus, by (\ref{ga}) and (\ref{ga1}): 
\begin{equation*}
\begin{split}
u_\epsilon(q_0') - u_\epsilon(q_0) & \leq \int_{\{Q_{\tau-1} \in
B_\delta(q_0)\}} |F(Q_{\tau-1}) - F(q_0)|\dd\mathbb{P} \\ & 
\qquad\qquad + \int_{\{Q_{\tau-1} \not\in
B_\delta(q_0)\}} |F(Q_{\tau-1}) - F(q_0)|\dd\mathbb{P} +
\frac{\eta}{3}\\ & \leq \frac{\eta}{3} +
2\|F\|_{\mathcal{C}(\heis)}\mathbb{P}\big(Q_{\tau-1}\not\in B_\delta(q_0)\big)+\frac{\eta}{3}\leq\eta.
\end{split}
\end{equation*}
The remaining inequality $u_\epsilon(q_0') - u_\epsilon(q_0)>-\eta$ is
obtained by a reverse argument.
\end{proof}

\section{Concatenating strategies, the annulus walk and the exterior
  $\mathbb{H}$-corkscrew condition as sufficient for
  game-regularity} \label{corkpH} 

We start with a result on concatenating strategies, which
contains a condition equivalent to the game-regularity criterion in
Definition \ref{def_gamereg} (a). This is similar to the proof of
Theorem \ref{th_concat}, both derived from the construction in
\cite{PS} for Euclidean setting. Let $\mathcal{D}\subset\heis$ be an open, bounded
connected domain.

\begin{Teo}\label{th_concat}
For a given $q_0\in\partial\mathcal{D}$, assume that there
exists $\theta_0\in (0,1)$ such that for every $\delta>0$ there exists
$\hat\delta\in (0, \delta)$ and  $\hat\epsilon\in (0,1)$ with the following property. Fix
$\epsilon\in (0, \hat\epsilon)$ and choose an initial position
$p_0\in B_{\hat\delta}(q_0)$; there exists a strategy $\sigma_{0,II}$ such that for every $\sigma_I$ we have:
\begin{equation}\label{name33}
\mathbb{P}\big(\exists n< \tau \quad Q_n\not\in B_\delta(q_0)\big)\leq\theta_0.
\end{equation}
Then $q_0$ is game-regular.
\end{Teo}
\begin{proof}
{\bf 1.} Under condition (\ref{name33}), construction of an optimal strategy realising the
(arbitrarily small) threshold $\eta$ in (\ref{gar}) is carried out by
concatenating the $m$ optimal strategies corresponding to the
achievable threshold $\eta_0$,
on $m$ concentric balls, where $(1-\eta_0)^m= 1-\theta_0^m\geq 1-\eta$.

Fix $\eta,\delta>0$. We want to find $\hat\epsilon$ and
$\hat\delta$ such that (\ref{gar}) holds.
Observe first that for $\theta_0\leq \eta$ the claim
follows directly from (\ref{name33}). In the general case, let $m\in \{2,3,\ldots\}$ be such that:
\begin{equation}\label{power}
\theta_0^m \leq \eta.
\end{equation}
Below we inductively define the radii $\{\delta_k\}_{k=1}^m$, together with the quantities
$\{\hat\delta(\delta_k)\}_{k=1}^m$, $\{\hat\epsilon(\delta_k)\}_{k=1}^m$ 
from the assumed condition (\ref{name33}). Namely,
for every initial position in $B_{\hat\delta(\delta_k)}(q_0)$
in the Tug of War game with step less than
$\hat\epsilon(\delta_k)$, there exists a strategy $\sigma_{0, II, k}$
guaranteeing exiting $B_{\delta_k}(q_0)$ (before the process
is stopped) with probability at most $\theta_0$. 
We set $\delta_m=\delta$ and find $\hat\delta(\delta_m)\in (0,\delta)$ and
$\hat\epsilon(\delta_m)\in (0,1)$, with the
indicated choice of the strategy $\sigma_{0,II,m}$. Decreasing the
value of $\hat\epsilon(\delta_m)$ if necessary, we then set:
$$ \delta_{m-1}=\hat\delta(\delta_m) - (1+\gamma_\p)\hat\epsilon(\delta_m)>0.$$
Similarly, having constructed $\delta_{k}>0$, we find
$\hat\delta(\delta_{k})\in (0,\delta_k)$ and
$\hat\epsilon(\delta_{k})\in (0, \hat\epsilon(\delta_{k+1}))$ and define: 
$$\delta_{k-1}=\hat\delta(\delta_{k}) - (1+\gamma_\p) \hat\epsilon(\delta_{k})>0.$$ 
Eventually, we call:
$$\hat\delta = \hat\delta(\delta_1), \qquad \hat\epsilon = \hat\epsilon(\delta_1).$$

To show that the condition of game-regularity at $q_0$ is
satisfied, we will concatenate the strategies
$\{\sigma_{0,II,k}\}_{k=1}^m$  by switching to $\sigma_{0,II,k+1}$
immediately after the token exits $B_{\delta_k}(q_0)\subset B_{\hat\delta(\delta_{k+1})}(q_0)$.
This construction is carried out in the next step.

\medskip 

{\bf 2.} Fix $p_0\in B_{\hat\delta}(q_0)$ and let $\epsilon\in (0,
\hat\epsilon)$. Define the strategy $\sigma_{0,II}$:
$$\sigma_{0,II}^n = \sigma_{0,II}^n\big(q_0, (q_1, w_1, s_1,
t_1),\ldots , (q_n, w_n, s_n, t_n)\big) \quad \mbox{ for all } \; n\geq 0,$$
separately in the following two cases.

\smallskip

\underline{Case 1.} If $q_k\in B_{\delta_1}(q_0)$ for all $k\leq n$, then we set: 
\begin{equation*}
\sigma_{0,II}^n = \sigma_{0,II,1}^n\big(q_0, (q_1, w_1, s_1, t_1),\ldots , (q_n, w_n, s_n, t_n)\big).
\end{equation*}

\smallskip

\underline{Case 2.} Otherwise, define:
\begin{equation*}
\begin{split} 
k &  \doteq k(x_0, x_1,\ldots, x_n) = \max\Big\{ 1\leq k\leq m-1;~ \exists
\;0\leq i\leq n~~ q_i\not\in B_{\delta_k}(q_0)\Big\}\\
i &  \doteq \min\Big\{ 0\leq i\leq n; ~ q_i\not\in B_{\delta_k}(q_0)\Big\}.
\end{split}
\end{equation*}
and set:
\begin{equation*}
\sigma_{0,II}^n = \sigma_{0,II,k+1}^{n-i}\big(q_i, (q_{i+1}, w_{i+1}, s_{i+1},
t_{i+1}),\ldots, (q_n, w_n, s_n, t_n)\big).
\end{equation*}
It is not hard to check that each $\sigma_{0,II}^n:H_n\to
B_{1}(0)\subset\heis$ is Borel measurable, as required. 
Let $\sigma_I$ be now any opposing strategy. Then, a classical
argument via Fubini's theorem, gives:
\begin{equation*}
\mathbb{P}\big(\exists n< \tau \quad  Q_n\not\in B_{\delta_k}(q_0)\big)
\leq\theta_0 \mathbb{P}\big(\exists n< \tau \quad Q_n\not\in B_{\delta_{k-1}}(q_0)\big)
\quad \mbox{for all } ~~k=2\ldots m,
\end{equation*}
so consequently:
\begin{equation*}
\begin{split}
\mathbb{P}\big(\exists n< \tau \quad Q_n\not\in B_{\delta}(q_0)\big)
\leq\theta_0^{m-1}  \mathbb{P}\big(\exists n\leq \tau \quad Q_n\not\in
B_{\delta_{1}}(q_0)\big)\leq \theta_0^m.
\end{split}
\end{equation*}
This yields the result by (\ref{power}) and completes the proof.
\end{proof}

\medskip

The proof of game-regularity in what follows will be based on
the concatenating strategies technique above and
the analysis of the annulus walk below. Namely, we
will derive an estimate on the probability of exiting a given 
annular domain $\tilde{\mathcal{D}}$ through the external portion of its
boundary. We show
that when the ratio of the annulus thickness and the distance of the
initial token position $q_0$ from the internal boundary is large
enough, then this probability may be bounded by a universal constant $\theta_0<1$.
When $\p\geq 4$, then $\theta_0$ converges to $0$ as the indicated
ratio goes to $\infty$.

\begin{Teo}\label{annulus}
For given radii $0<R_1<R_2<R_3$, consider the annulus
$\tilde{\mathcal{D}} = B_{ R_3}(0)\setminus \bar B_{R_1}(0)\subset
\heis$. For every $\xi>0$, there exists $\hat\epsilon\in (0,1)$ depending on $R_1, R_2,
R_3$ and $\xi, \p$, such that for every $p_0\in \tilde{\mathcal{D}} \cap B_{R_2}(0)$ and every
$\epsilon\in (0,\hat\epsilon)$, there exists a strategy $\tilde \sigma_{0,II}$
with the property that for every strategy $\tilde\sigma_I$ there holds:
\begin{equation}\label{name}
\mathbb{P}\Big(\tilde Q_{\tilde\tau-1}\not\in \bar B_{R_3 -
  \epsilon}(0)\Big)\leq \frac{v(R_2) - v(R_1)}{v(R_3) - v(R_1)} + \xi.
\end{equation}
Here, $v:(0,\infty)\to\mathbb{R}$ is given by:
\begin{equation}\label{vdef}
v(t) = \left\{\begin{array}{ll} \displaystyle{\mbox{sgn}(\p-4) \,t^{\frac{\p-4}{\p-1}}} & \mbox{ for }
    \p\neq 4\\ \log t & \mbox{ for } \p=4,\end{array}
\right.
\end{equation}
and $\{\tilde Q_n = \tilde Q_n^{\epsilon, p_0, \tilde\sigma_I,
  \tilde\sigma_{0,II}}\}_{n=0}^\infty$ and $\tilde\tau =
\tilde\tau^{\epsilon, p_0, \tilde\sigma_I,
  \tilde\sigma_{0,II}}$ denote, as usual,  the random variables corresponding
to positions and stopping time in the Tug of War game
(\ref{processMp}) on $\tilde{\mathcal{D}}$.
\end{Teo}
\begin{proof}
Consider the radial function $u:\heis\setminus\{0\}\to\mathbb{R}$ given by
$u(q) = v(|q|_K)$, where $v$ is as in (\ref{vdef}). Recall that:
\begin{equation}\label{vdef2}
\Delta_{\heis, \p} u = 0 \quad \mbox{ and } \quad \nabla_\heis u\neq
0 \qquad \mbox{ in } \,\,\heis\setminus\{0\}. 
\end{equation}
Let $\tilde u_\epsilon$ be the family of solutions to (\ref{DPP2}) with the  data $F$ 
given by a smooth and bounded modification of $u$ outside of the
annulus $B_{2R_3}(0)\setminus \bar{B}_{R_1/2}(0)$.
By Theorem \ref{conv_nondegene}, there exists a constant $C>0$, depending
only on $\p, u$ and $\tilde{\mathcal{D}}$, such that:
$$\|\tilde u_\epsilon - u\|_{\mathcal{C}(\tilde{\mathcal{D}})}\leq C\epsilon
\qquad \mbox{as }\;\epsilon\to 0.$$

For a given $q_0\in \tilde{\mathcal{D}}\cap B_{R_2}(0)$, 
there exists thus a strategy $\tilde\sigma_{0,II}$ so that for every $\tilde\sigma_{I}$ we have: 
\begin{equation}\label{vdef3} 
\mathbb{E}\big[ u\circ (\tilde Q^{\epsilon, q_0,
  \tilde\sigma_I, \tilde\sigma_{0, II}})_{\tilde \tau-1}\big] - u(q_0) \leq 2C\epsilon.
\end{equation}
We now estimate:
\begin{equation*}
\begin{split}
\mathbb{E} \big[u\circ  \tilde  Q_{\tilde \tau-1}\big] - u(q_0) & = 
\int_{\{\tilde Q_{\tilde\tau-1}\not\in \bar B_{R_3 - \epsilon}(0)\}} u(\tilde
Q_{\tilde \tau-1})~\mbox{d}\mathbb{P} + 
\int_{\{\tilde Q_{\tilde\tau-1}\in B_{R_1 + \epsilon}(0)\}} u(\tilde
Q_{\tilde \tau-1})~\mbox{d}\mathbb{P} - u(q_0) \\ & \geq \mathbb{P}\big(\tilde
Q_{\tilde\tau-1}\not\in \bar B_{R_3-\epsilon}(0)\big)
v\big(R_3- \epsilon\big) \\ & \qquad + \Big(1- \mathbb{P}\big(\tilde
Q_{\tilde\tau-1}\not\in \bar B_{R_3 - \epsilon}(0)\big)\Big)
v\big(R_1 - \gamma_\p\epsilon\big) - v(R_2),
\end{split}
\end{equation*}
where we used the fact that $v$ in (\ref{vdef}) is an increasing
function. Recalling (\ref{vdef3}), this implies:
\begin{equation}\label{vdef4}
\mathbb{P}\big(\tilde Q_{\tilde\tau-1}\not\in \bar B_{R_3 - \epsilon}(0)\big)
\leq \frac{ v(R_2) -
  v(R_1 - \gamma_\p\epsilon) + 2C\epsilon}{v(R_3 - \epsilon) -
  v(R_1 - \gamma_\p\epsilon)}.
\end{equation}
The proof of (\ref{name}) is now complete, by continuity of the right hand side
with respect to $\epsilon$.
\end{proof}

\medskip

By inspecting the quotient in the right hand side of (\ref{name}) we obtain:

\begin{Cor}\label{annulus3}
The function $v$ in (\ref{vdef}) satisfies, for any fixed $0<R_1<R_2$:
\begin{itemize}
\item[(a)] $~~~\displaystyle{\lim_{R_3\to \infty}\,\frac{v(R_2) - v(R_1)}{v(R_3) -
  v(R_1)} = \left\{\begin{array}{ll}
\displaystyle{1- \Big(\frac{R_2}{R_1}\Big)^{\frac{\p-4}{\p-1}}} & \mbox{for } 2<\p<4\vspace{1mm}\\
0 & \mbox{for } \p\geq 4, \end{array}\right.\displaystyle}$
\item[(b)] $~~~\displaystyle{\lim_{M\to \infty}\,\frac{v(MR_1) - v(R_1)}{v(M^2R_1) -
  v(R_1)} = \left\{\begin{array}{ll} \displaystyle{\frac{1}{2}} & \mbox{for } \p=4\vspace{1mm}\\
 \,0 & \mbox{for } \p>4. \end{array}\right.\displaystyle}$
\end{itemize}
Consequently, the estimate (\ref{name}) can be replaced by:
\begin{equation}\label{name2}
\mathbb{P}\big(\tilde Q_{\tilde\tau-1}\not\in \bar B_{R_3 -
  \epsilon}(0)\big)\leq \theta_0
\end{equation}
valid for any $\theta_0> 1- \Big(\frac{R_2}{R_1}\Big)^{\frac{\p-4}{\p-1}}$
if $\p\in (2, 4)$, and any $\theta_0>0$ if $\p\geq 4$, upon choosing $R_3$
sufficiently large with respect to $R_1$ and $R_2$. Alternatively,
when $\p>4$, the same bound can be achieved by
setting $R_2=MR_1$, $R_3=M^2R_1$ with the ratio $M$ large enough.
\end{Cor}

\medskip

The results of Theorem \ref{annulus} and Corollary \ref{annulus3} are
invariant under scaling, i.e.:

\begin{Rem}\label{annulus2}
The bounds (\ref{name}) and (\ref{name2}) remain true if we replace
$R_1, R_2$, $R_3$ by $r R_1, r R_2$, $r R_3$, the domain
$\tilde{\mathcal{D}}$ by $\rho_r\tilde{\mathcal{D}}$ and $\hat\epsilon$
by $r\hat\epsilon$, for any $r>0$.
\end{Rem}

%\medskip

\begin{figure}[htbp]
\centering
\includegraphics[scale=0.5]{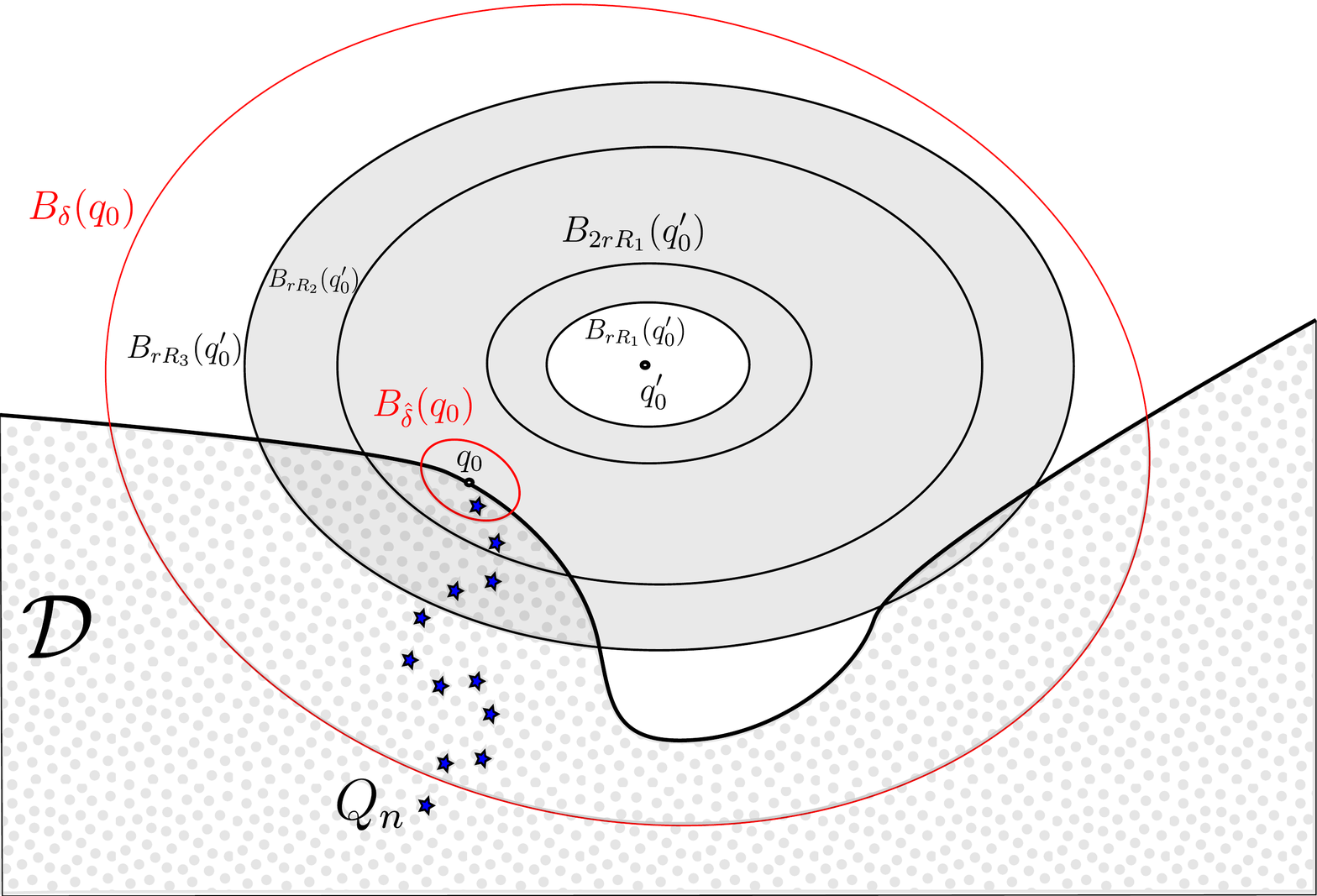}
    \caption{{The concentric balls and the annuli in the proof
        of Theorem \ref{corkthengamereg}.}}
\label{f:annuli_proof}
\end{figure}

\begin{Teo}\label{corkthengamereg}
Let $q_0\in\partial\mathcal{D}$ satisfy the exterior
$\mathbb{H}$-corkscrew condition, as in Definition \ref{cork_def}.
Then $q_0$ is game-regular.
\end{Teo}
\begin{proof}
With the help of Theorem \ref{annulus}, we will show that the
assumption of Theorem \ref{th_concat} is satisfied,
with proba\-bi\-li\-ty $\theta_0<1$ 
depending only on $\p>2$ and $\mu\in (0,1)$ in Definition \ref{cork_def}. Namely,
set $R_1=1$, $R_2= \frac{2}{\mu}$ and $R_3>R_2$ according to Corollary
\ref{annulus3} (a) in order to have $\theta_0= \theta_0(\p, R_1,
R_2)<1$. Further, set $r=\frac{\delta}{2R_3}$. Using the corkscrew
condition, we obtain:
$$B_{2rR_1}(q_0') \subset B_{\delta/(\mu R_3)}(q_0)\setminus \bar{\mathcal{D}},$$
for some $q_0'\in\heis$.
In particular: $d(q_0, q_0')<rR_2$, so $q_0\in B_{rR_2}(q_0')\setminus
\bar B_{2rR_1}(q_0')$. It now easily follows that there exists
$\hat\delta\in (0,\delta)$ with the property that:
$$B_{\hat\delta}(q_0) \subset B_{rR_2}(q_0')\setminus \bar B_{2rR_1}(q_0').$$
Finally, we observe that $B_{rR_3}(q_0')\subset B_\delta(q_0)$ because
$rR_3 + rR_2<2rR_3 = \delta$.

\medskip

Let $\hat\epsilon/r>0$ be as in Theorem \ref{annulus}, applied to the
annuli with radii $R_1, R_2, R_3$, in view of Remark \ref{annulus2}. For a given
$p_0\in B_{\hat\delta}(q_0)$ and $\epsilon\in (0,\hat\epsilon)$, let
$\tilde\sigma_{0,II}$ be the strategy ensuring validity of the bound
(\ref{name2}) in the annulus walk on $q_0'*\tilde{\mathcal{D}}.$ 
For a given strategy $\sigma_I$ there clearly holds:
\begin{equation*}
\begin{split}
\Big\{\omega\in\Omega; ~  \exists n<\tau^{\epsilon, p_0, \sigma_I,
  \sigma_{0,II}}&(\omega) \qquad Q_n^{\epsilon,p_0, \sigma_I,
  \sigma_{0,II}}(\omega)\not\in B_\delta(q_0)\Big\} \\ & \subset 
\Big\{\omega\in\Omega; ~ \tilde Q^{\epsilon, p_0, \tilde \sigma_I,
  \tilde \sigma_{0,II}}_{\tilde\tau-1}(\omega)\not\in B_{r R_3-\epsilon}(q_0')\Big\}.
\end{split}
\end{equation*}
The final claim follows by (\ref{name2}) and by applying  Theorem \ref{th_concat}.
\end{proof}

\begin{Rem}
By Corollary \ref{annulus3} (b) and adjusting the arguments in
\cite{PS} to the Heisenberg group geometry, one can show that every open,
bounded $\mathcal{D}\subset\heis$ is game-re\-gular for $\p>4$. 
This will be the content of a separate work.
\end{Rem}

\section{Uniqueness and identification of the limit $u$ in Theorem
  \ref{thm_gamereg_conv}} \label{uniq_p_sec}

Let $F\in\mathcal{C}(\heis)$ be a bounded data function and let
$\mathcal{D}$ be open, bounded and game-regular. In virtue of Theorem
\ref{thm_gamereg_conv} and the Ascoli-Arzel\`{a} theorem, every sequence in
the family $\{u_\epsilon\}_{\epsilon\to 0}$ of solutions to
(\ref{DPP2}) has a further subsequence converging uniformly to some
$u\in\mathcal{C}(\heis)$ and satisfying $u=F$ on $\heis\setminus \mathcal{D}$.
We will show that such limit $u$ is in fact unique.

\medskip

Recall first the definition of the $\p$-$\heis$-harmonic viscosity
solution, that should be compared with the definition in the statement
of Proposition \ref{MV3expansion}, valid for $\p=2$.

\begin{Def}\label{visco_def_p}
We say that $u\in\mathcal{C}(\bar{\mathcal{D}})$ is a {\em viscosity solution} to the following problem:
\begin{equation}\label{problem} 
\Delta_{\heis, \p}u =0 \quad \mbox{in }\mathcal{D},
\qquad u=F \quad \mbox{on }\partial\mathcal{D},
\end{equation}
if the latter boundary condition holds and if:
\begin{itemize}
\item[(i)] for every $q_0\in\mathcal{D}$ and every $\phi\in\mathcal{C}^2(\bar{\mathcal{D}})$ such that:
\begin{equation}\label{assu_super}
\phi(q_0) = u(q_0), \quad \phi<u ~~ \mbox{in }
\bar{\mathcal{D}}\setminus \{q_0\}\quad \mbox{ and } \quad  \nabla_{\heis} \phi(q_0)\neq 0,
\end{equation}
there holds: $\Delta_{\heis, \p} \phi(q_0)\leq 0$,
\item[(ii)] for every $q_0\in\mathcal{D}$ and every $\phi\in\mathcal{C}^2(\bar{\mathcal{D}})$ such that:
\begin{equation*}\label{assu_sub}
\phi(q_0) = u(q_0), \quad \phi>u ~~ \mbox{in }
\bar{\mathcal{D}}\setminus \{q_0\}\quad \mbox{ and } \quad  \nabla_{\heis}\phi(q_0)\neq 0,
\end{equation*}
there holds: $\Delta_{\heis, \p} \phi(q_0)\geq 0$.
\end{itemize}
\end{Def}

\begin{Teo}\label{unif-topharm}
Assume that the sequence $\{u_\epsilon\}_{\epsilon\in J, \epsilon\to 0}$ of solutions to
(\ref{DPP2}) with a bounded data function $F\in\mathcal{C}(\heis)$,
converges uniformly as $\epsilon\to 0$ to some limit  
$u\in\mathcal{C}(\heis)$. Then $u$ must be the viscosity solution to (\ref{problem}).
\end{Teo}
\begin{proof}
Fix $q_0\in\mathcal{D}$ and let $\phi$ be a test function as in (\ref{assu_super}).
Using the same argument as in the proof of Lemma \ref{appromin}, we
observe that there exists a sequence $\{q_\epsilon\}_{\epsilon\in J}\in\mathcal{D}$, such that: 
\begin{equation}\label{4.7}
\lim_{\epsilon\to 0, \epsilon\in J}q_\epsilon = q_0 \quad \mbox{ and } \quad
u_\epsilon(q_\epsilon) - \phi(q_\epsilon) = \min_{\bar{\mathcal{D}}} \,(u_\epsilon - \phi).
\end{equation}
Since by (\ref{4.7}) we have: $\phi(q) \leq u_\epsilon(q) +
\big(\phi(q_\epsilon) - u_\epsilon(q_\epsilon)\big)$ for all $q\in\bar{\mathcal{D}}$, it follows that:
\begin{equation}\label{4.8}
\begin{split}
 \frac{1}{3} \Big(\mathcal{A}_3&(\phi,\epsilon)(q_\epsilon) + \inf_{
  B_{\gamma_p\epsilon}(q_\epsilon)} \mathcal{A}_3(\phi, \epsilon) + \sup_{
B_{\gamma_\p\epsilon}(q_\epsilon)} \mathcal{A}_3(\phi, \epsilon) \Big) -
\phi(q_\epsilon) \\ & \leq  \frac{1}{3} \Big(\mathcal{A}_3(u_\epsilon,\epsilon) (q_\epsilon) + \inf_{
  B_{\gamma_\p\epsilon}(q_\epsilon)} \mathcal{A}_3(u_\epsilon, \epsilon) + \sup_{
B_{\gamma_\p\epsilon}(q_\epsilon)} \mathcal{A}_3(u_\epsilon, \epsilon) \Big) 
\\ & \qquad\qquad\qquad\qquad\qquad \quad
+ \big(\phi(q_\epsilon) - u_\epsilon(q_\epsilon)\big) - \phi(q_\epsilon)  = 0,
\end{split}
\end{equation}
for all $\epsilon$ small enough to guarantee that $d_\epsilon(q_\epsilon) = 1$.
On the other hand, (\ref{dpp2}) yields:
\begin{equation*}
\begin{split}
  \frac{1}{3} \Big(\mathcal{A}_3&(\phi,\epsilon)(q_\epsilon) + \inf_{
  B_{\gamma_\p\epsilon}(q_\epsilon)} \mathcal{A}_3(\phi, \epsilon) + \sup_{
B_{\gamma_\p\epsilon}(q_\epsilon)} \mathcal{A}_3(\phi, \epsilon) \Big)
- \phi(q_\epsilon) = \frac{\epsilon^2}{3\pi}\cdot
\frac{\Delta_{\heis, \p}\phi(q_\epsilon) }{|\nabla_\heis\phi(q_\epsilon)|^{\p-2}}+ o(\epsilon^2),
\end{split}
\end{equation*}
for $\epsilon$ small enough to get $\nabla_\heis\phi(q_\epsilon)\neq 0$.
Combining the above with (\ref{4.8}) gives:
$$\Delta_{\heis, \p}\phi(q_\epsilon) \leq o(1).$$
Passing to the limit with $\epsilon\to 0, \epsilon\in J$ establishes the desired
inequality $\Delta_{\heis, \p}\phi(q_0)\leq 0$ and proves part (i) of Definition  \ref{visco_def_p}.
The verification of part (ii) is done along the same lines.
\end{proof}

Since the viscosity solutions $u\in\mathcal{C}(\bar{\mathcal{D}})$ of (\ref{problem}) are
unique \cite{B2},
% \cite[Theorem 7.2]{KM} (see also the comparison principle
%developed  in \cite{B2}), 
in view of Theorems \ref{unif-topharm} and 
\ref{thm_gamereg_conv} we obtain:

\begin{Cor}
Let $F\in\mathcal{C}(\heis)$ be a bounded data function and let
$\mathcal{D}$ be open, bounded and game-regular. The family
$\{u_\epsilon\}_{\epsilon\to 0}$ of solutions to 
(\ref{DPP2}) converges uniformly in $\bar{\mathcal{D}}$ to the unique
viscosity solution of (\ref{problem}).
\end{Cor}

\end{document}